\documentclass[12pt]{amsart}
\usepackage{mathscinet}
\usepackage[shortlabels]{enumitem}
\usepackage{tikz}
\usetikzlibrary{graphs,quotes,arrows.meta}
\usetikzlibrary{arrows,calc,shapes,decorations.pathreplacing}
\usetikzlibrary{decorations.pathmorphing}
\usepackage{amsmath,amsfonts,amssymb,amsthm,epsfig,epstopdf,array}
\usepackage{wasysym}
\usepackage{scalerel}
\usepackage{graphicx}
\usepackage{float}
\usepackage{centernot}
\usepackage[colorlinks=true,pagebackref,hyperindex]{hyperref}
\usepackage[linewidth=1pt]{mdframed}
\usepackage{adjustbox}

\theoremstyle{plain}
\newtheorem{theorem}{Theorem}[section]
\newtheorem{lemma}[theorem]{Lemma}
\newtheorem{proposition}[theorem]{Proposition}
\newtheorem{corollary}[theorem]{Corollary}

\theoremstyle{definition}
\newtheorem{definition}[theorem]{Definition}
\newtheorem{remark}[theorem]{Remark}
\newtheorem{example}[theorem]{Example}

\newcommand{\source}{{\mathrm{s}}}
\newcommand{\target}{{\mathrm{t}}}
\newcommand{\mm}{{\mathbf{m}}}
\newcommand{\uu}{{\mathbf{u}}}

\begin{document}

\title{Bipartite determinantal ideals and concurrent vertex maps}

\author{Li Li}
\address{Department of Mathematics
and Statistics,
Oakland University, 
Rochester, MI 48309-4479, USA 
}
\email{li2345@oakland.edu}


\begin{abstract}
Bipartite determinantal ideals are introduced in \cite{IL} as a vast generalization of the classical determinantal ideals intensively studied in commutative algebra, algebraic geometry, representation theory and combinatorics.  We introduce a combinatorial model called concurrent vertex maps to describe the Stanley-Reisner complex of the initial ideal of any bipartite determinantal ideal, and study properties and applications of this model including vertex decomposability, shelling orders, formulas of the Hilbert series and $h$-polynomials.
\end{abstract}

\maketitle

%\tableofcontents

\section{Introduction}
The primary objective of this paper is to present a novel combinatorial model called ``concurrent vertex maps". This model is introduced to effectively describe the Stanley-Reisner complex of a bipartite determinantal ideal. Additionally, we use this combinatorial model to study properties of the Stanley-Reisner complex.

We begin by reviewing the classical determinantal ideals. Let ${\bf k}$ be a field, $A$ be an $m\times n$ matrix consisting of independent variables $x_{ij}$,  $r$ be a positive integer. The ideal $I$ of ${\bf k}[x_{ij}]$ generated by all $r\times r$ minors in $A$ is referred as a determinantal ideal. The quotient ring ${\bf k}[x_{ij}]/I$ is known as a determinantal ring. 
Geometrically, a determinantal ideal defines an affine variety called determinantal variety, whose ${\bf k}$-points are $m\times n$ matrices with entries in ${\bf k}$ and rank at most $r-1$. 
For example, take the case where $m=2, n=3, r=2$. The matrix $A$ and the determinantal ideal $I$ are: 
$$A=\begin{bmatrix}x_{11}&x_{12}&x_{13}\\ x_{21}&x_{22}&x_{33}\end{bmatrix},
\quad
I=\langle
\begin{vmatrix}x_{11}&x_{12}\\ x_{21}&x_{22}\end{vmatrix},
\begin{vmatrix}x_{11}&x_{13}\\ x_{21}&x_{23}\end{vmatrix},
\begin{vmatrix}x_{12}&x_{13}\\ x_{22}&x_{23}\end{vmatrix}
\rangle
$$
The corresponding determinantal variety consisits of all $2\times 3$ matrices that do not have full rank. This variety is irreducible, singular, and $4$-dimensional. 

Determinantal ideals have been a central topic, a test field, and a source of inspiration in various areas such as commutative algebra, algebraic geometry, representation theory, and combinatorics.
Researchers have extensively investigated determinantal ideals/rings, exploring topics such as minimal free resolution and syzygies (both algebraically and geometrically in terms of vector bundles), (Arithmetic) Cohen-Macaulayness, multiplicities, powers and products, Hilbert series, $a$-invariant, $F$-regularity and $F$-rationality, and local cohomology. The tools employed in these investigations include the Gr\"obner basis theory, simplicial complexes, the straightening law, the Knuth-Robinson-Schensted (KRS) correspondence, non-intersecting paths from algebraic combinatorics, and liaison theory from commutative algebra. Determinantal ideals/rings are closely connected to the study of Grassmannian and Schubert varieties, and there exist interesting specializations, generalizations, and variations, such as Segre varieties, linear determinantal varieties, rational normal scrolls, symmetric and skew-symmetric determinantal varieties, as well as classical, two-sided, or mixed ladder determinantal varieties. For further reading on these topics, interested readers can refer to \cite{Baetica, BCRV, Harris, MS, Miro-Roig,Narasimhan, Weyman} and the references therein.

In our previous work \cite{IL}, we introduced a generalization of determinantal ideals called the bipartite determinantal ideals, and showed that the natural generators form a Gr\"obner basis under a suitable monomial order. Let us recall the definition and the main results.

\begin{definition}\label{bipartitedeterminantalideal}(Bipartite determinantal ideal)
Let $\mathcal{Q}$ be a bipartite quiver with $d$ vertices and the vertex set ${\rm V}_\mathcal{Q}={\rm V}_{\rm source}\sqcup {\rm V}_{\rm target}$, and with $r$ arrows $h_i: \source(h_i)\to \target(h_i)$ (for $i=1,\dots,r$) such that $\source(h_i)\in {\rm V}_{\rm source}$ and $\target(h_i)\in {\rm V}_{\rm target}$.
Let $\mm=(m_1,\dots,m_d)$, $\uu=(u_1,\dots,u_d)$  be two $d$-tuples of nonnegative integers.
Let $X= \{ X^{(k)}=[x^{(k)}_{ij}]\}_{k=1}^r $ where  $X^{(k)}$ is an $m_{\target(h_k)}\times m_{\source(h_k)}$ matrix consisting of independent variables.  
For $\alpha\in {\rm V}_{\rm target}$ and $\beta\in {\rm V}_{\rm source}$,
 let $r_1<\cdots<r_s$ be all the indices such that $h_{r_1},\dots,h_{r_s}$ are all the arrows with target $\alpha$, 
 let $r'_1<\cdots<r'_t$ be all the indices such that $h_{r'_1},\dots,h_{r'_t}$ are all the arrows with source $\beta$, 
and define block matrices
	$$ A_\alpha := [ X^{(r_1)}|X^{(r_2)}|\hdots|X^{(r_s)} ],
	\quad
	A_\beta :=\left[
	\begin{array}{c}
	X^{(r'_1)}  \\ \hline
	X^{(r'_2)} \\ \hline
	\vdots \\ \hline
	X^{(r'_t)}
	\end{array}
	\right]. $$	
For $\gamma\in{\rm V}_\mathcal{Q}$, let $a_\gamma\times b_\gamma$ be the size of $A_\gamma$, that is, 
$(a_\alpha,b_\alpha) :=(m_\alpha,\sum_{h:\ \target(h)=\alpha} m_{\source(h)})$ for $\alpha\in {\rm V}_{\rm target}$, 
$(a_\beta,b_\beta) :=(\sum_{h:\ \source(h)=\beta} m_\alpha,m_{\target(h)})$ for $\beta\in {\rm V}_{\rm source}$. 
The \textit{bipartite determinantal ideal} $I_{\mathcal{Q},\mm,\uu}$ is the ideal of ${\bf k}[X]$ generated by $\bigcup_{\gamma\in{\rm V}_\mathcal{Q}}D_{u_\gamma+1}(A_\gamma)$, where $D_{u_\gamma+1}(A_\gamma)$ is the set of all $(u_\gamma+1)$-minors of $A_\gamma$.  These generators are called the \emph{natural generators} of $I_{\mathcal{Q},\mm,\uu}$. 
\end{definition}

As explained in \cite{IL}, the bipartite determinantal ideal specializes to the classical determinantal ideal $I^{\rm det}_{m,n,u}$ of $(u+1)$-minors in an $m\times n$ matrix for $\mathcal{Q}:2\to 1$, ${\bf m}=(m_1,m_2)=(m,n)$, ${\bf u}=(u_1,u_2)=(u,u)$; more generally, they specializes to a double determinantal ideal $I^{(r)}_{m,n,u,v}$ for $\mathcal{Q}:2\stackrel{r}{\to}1$ with $r$ arrows from $2$ to $1$ which we introduced in \cite{IL} and is generated by $(u+1)$-minors in $A_1$ and $(v+1)$-minors in $A_2$, where 
$$ A_1= [ X^{(1)}|X^{(2)}|\hdots|X^{(r)} ], 
	\hspace{12pt}
	A_2= \left[
	\begin{array}{c}
	X^{(1)}  \\ \hline
	X^{(2)} \\ \hline
	\vdots \\ \hline
	X^{(r)}
	\end{array}
	\right] $$	  
and each $X^{(k)}$ is a variable matrix of dimension $m\times n$.

Define 
$$v_\alpha :=\sum_{\target(h_i)=\alpha} u_{\source(h_{i})} \textrm{ for }\alpha\in {\rm V}_{\rm target}, 
\quad
v_\beta:=\sum_{\source(h_i)=\beta} u_{\target(h_i)} \textrm{ for }\beta\in {\rm V}_{\rm source}. $$
Without loss of generality, we assume
\begin{equation}\label{eq:condition on u} 
\forall \gamma\in {\rm V}_\mathcal{Q}: 0<u_\gamma\le \min(a_\gamma,b_\gamma),\quad
\forall\alpha\in {\rm V}_{\rm target}: u_\alpha\le v_\alpha, \quad
\forall\beta\in {\rm V}_{\rm source}: u_\beta\le v_\beta.
\end{equation}
Indeed, 
if $u_\gamma \ge \min(a_\gamma,b_\gamma)$ (resp. $u_\alpha>v_\alpha$, $u_\beta>v_\beta$), then we can replace $\mu_\gamma$ by $\min(a_\gamma,b_\gamma)$ (resp. $u_\alpha$ by $v_\alpha$, $u_\beta$ by $v_\beta$) and  $I_{\mathcal{Q},\mm,\uu}$ remains unchanged;
if $u_\gamma=0$, then let $\mathcal{Q}'$ be the quiver obtained from $\mathcal{Q}$ by removing the vertex $\gamma$ and all the arrows incident to $\gamma$, let $\mm'=(m_i)_{i\neq \gamma}$, $\uu'=(u_i)_{j\neq \gamma}$. 
We have
$I_{\mathcal{Q},\mm,\uu}= I_{\mathcal{Q}',\mm',\uu}'
+ 
\langle x^{(k)}_{ij}\in X^{(k)} \ | \ \textrm{ the arrow $h_k$ is incident to $\gamma$} \rangle $
and the natural generators of $I_{\mathcal{Q},\mm,\uu}$ consists of the natural generators of $I_{\mathcal{Q}',\mm',\uu'}$ together with variables $x^{(k)}_{ij}$ where the arrow $h_k$ is incident to $\gamma$.  So no essential information is lost when we replace 
$I_{\mathcal{Q},\mm,\uu}$ by $I_{\mathcal{Q}',\mm',\uu'}$.
We should point out that if $u_\gamma = \min(a_\gamma,b_\gamma)$, some of the proofs in this paper need to be modified, but they always degenerate to simpler or trivial cases, so we feel it is unnecessary and redundant to explain all the modifications, and is safe to omit.

For a matrix of variables  $Y=(y_{ij})$, we say that an order ``$>$'' is \textit{consistent} on $Y$ when $y_{ij}> y_{i,j+k}$ and $y_{ij}> y_{i+k,j}$ for all $i,j$.  
There exists a lexicographical monomial order that is consistent on all $A_\gamma$, for example we can take the  lex order ``$>$'' such that 
\begin{equation}\label{eq: favorite order}
\textrm{$x_{ij}^{(r)}>x_{kl}^{(s)}$ if : ``$r<s$'', or ``$r=s,i<k$'', or ``$r=s,i=k,j<l$''.}
\end{equation}

Recall the following theorem proved in \cite{IL}.

\begin{theorem} \label{general}\cite{IL}
Fix ${\mathcal{Q},\mm,\uu}$ satisfying \eqref{eq:condition on u}. 
Then the set of natural generators of  
{bipartite determinantal ideal} $I_{\mathcal{Q},\mm,\uu}$, 
if nonempty, forms a Gr\"obner basis with respect to any lexicographical monomial order that is consistent on all $A_i$. 
\end{theorem}

\begin{corollary}\label{cor:general}
The bipartite determinantal ideals $I_{\mathcal{Q},\mm,\uu}$ are prime.
\end{corollary}
\begin{proof}
It suffices to prove the statement under the assumption that ${\bf k}$ is algebraically closed. %because the contraction of a prime ideal is prime
By the Gr\"obner bases result (Theorem \ref{general}), the initial ideal of $I_{Q,\mm,\uu}$, denoted ${\rm init }(I_{\mathcal{Q},\mm,\uu})$,  is generated by the leading terms of its natural generators. Since each generator is the product of diagonal entries of a submatrix of $A_i$ for some $i$, the ideal ${\rm init }(I_{\mathcal{Q},\mm,\uu})$ must be a squarefree monomial, thus is radical. Therefore $I_{\mathcal{Q},\mm,\uu}$ is radical.
On the other hand, there exists a nonsingular irreducible variety  that maps surjectively to ${\rm Spec }({\bf k}[X]/I_{\mathcal{Q},\mm,\uu})$: for example we can construct a variety similar to Nakajima's nonsingular graded quiver variety $\tilde{\mathcal{F}}_{v,w}$ (\cite[\S4]{Li}) or similar to $\tilde{\mathcal{G}}_{v,w}$ given in \cite[\S6.2]{Li} (note that the construction works even if ${\bf k}$ has positive characteristic). it follows that ${\rm Spec }({\bf k}[X]/I_{\mathcal{Q},\mm,\uu})$ is irreducible and $I_{\mathcal{Q},\mm,\uu}$  is prime.
\end{proof}

By Stanley-Reisner correspondence (see \cite[Theorem 1.7]{MS}), there is a bijection from simplicial complexes on $n$ vertices to squarefree monomial ideals in $\mathbf{k}[x_1,\dots,x_n]$, sending $\Delta$ to
$$I_\Delta=\langle {\bf x}^\tau|\tau\notin \Delta \rangle
=\bigcap_{\sigma\in\Delta} \langle x_i\ | \ i\notin\sigma \rangle
=\bigcap_{\sigma\in{\rm facets}(\Delta)} \langle x_i\ | \ i\notin\sigma \rangle
$$
where the last one gives a minimal prime decomposition of $I_\Delta$. 
%Note that $\Delta$ is determined by the monomial ideal $I$ because $\Delta=\{\tau\  |\  x^\tau\notin I \}$. 
Let $\Delta_{\mathcal{Q},\mm,\uu}$ be the simplicial complex determined by $I_{\mathcal{Q},\mm,\uu}$ (so
${\rm init }(I_{\mathcal{Q},\mm,\uu})=I_{\Delta_{\mathcal{Q},\mm,\uu}}$). 

We will introduce a new combinatorial model, called concurrent vertex maps, as an analogue of nonintersecting lattice paths.\footnote{A concurrent vertex map consists of each point that is in a common section of a horizontal path and a vertical path. The name is chosen due to its analogy with a concurrency in a road network, which refers to ``an instance of one physical roadway bearing two or more different route numbers''. 
See Wikipidea page \url{https://en.wikipedia.org/wiki/Concurrency_(road)}} 
The precise definition is given in \S\ref{section:concurrent maps}. Roughly speaking,
 
(a) a  road map is, on each rectangular array of the size of the matrix $A_\gamma$,  
a family of nonintersecting paths $H^\gamma_p$ (for various $p$) from the left side to the right side (which we called horizontal paths), together with 
a family of nonintersecting paths $V^\gamma_q$ (for various $q$)  from the top side to the bottom side (which we called vertical paths); 

(b) a road map is called straight if each corner of a horizontal path lies in a vertical path, and each corner of a vertical path lies in a horizontal path;

(c) for a straight road map, the corresponding concurrent vertex map is the set of lattice points that are both in horizontal and vertical paths.

%Below is an example of a concurrent vertex map. 

\begin{example}\label{eg:234to1}
Consider the bipartite determinantal ideal $I_{\mathcal{Q},\mm,\uu}$  where $\mathcal{Q}$ is
$$
  
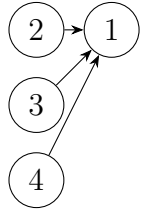
\begin{tikzpicture}
    \graph[nodes={draw,circle},edges={-{Stealth[]}}] {
      2 -> 1, 
      3 -> 1,
      4 -> 1
    };
  \end{tikzpicture}
$$
$\mm=(3,2,2,2)$, $\uu=(2,1,1,1)$. The ideal is generated by 

\noindent-- all 2 minors in 
$
A_2=X_1=\begin{bmatrix}
x^{(1)}_{11}&x^{(1)}_{12}\\
x^{(1)}_{21}&x^{(1)}_{22}\\
x^{(1)}_{31}&x^{(1)}_{32}\\
\end{bmatrix}
$, 
$
A_3=X_2=\begin{bmatrix}
x^{(2)}_{11}&x^{(2)}_{12}\\
x^{(2)}_{21}&x^{(2)}_{22}\\
x^{(2)}_{31}&x^{(2)}_{32}\\
\end{bmatrix}
$, 
$
A_4=X_3=\begin{bmatrix}
x^{(3)}_{11}&x^{(3)}_{12}\\
x^{(3)}_{21}&x^{(3)}_{22}\\
x^{(3)}_{31}&x^{(3)}_{32}\\
\end{bmatrix}
$,

\noindent -- and all 3 minors in 
$A_1=[X_1 \ X_2\ X_3 ] = 
\begin{bmatrix}
x^{(1)}_{11}&x^{(1)}_{12}&x^{(2)}_{11}&x^{(2)}_{12}&x^{(3)}_{11}&x^{(3)}_{12}\\
x^{(1)}_{21}&x^{(1)}_{22}&x^{(2)}_{21}&x^{(2)}_{22}&x^{(3)}_{21}&x^{(3)}_{22}\\
x^{(1)}_{31}&x^{(1)}_{32}&x^{(2)}_{31}&x^{(2)}_{32}&x^{(3)}_{31}&x^{(3)}_{32}\\
\end{bmatrix}
$. 

\noindent In Figure \ref{fig:example concurrent map} Left, 
\begin{figure}[ht]
\begin{center}
\begin{tikzpicture}[scale=.5]
    \begin{scope}[shift={(0,0)}] %10
 \draw  [gray] (0,0) grid (1,2);
 \draw  [gray] (2,0) grid (3,2);
 \draw  [gray] (4,0) grid (5,2);
 
\begin{scope}[shift={(.07,-.07)}]
\draw [blue,densely dotted,very thick](-.2,1)--(3,1)--(3,2)--(5,2);
\draw [blue,densely dotted,very thick](-.2,0)--(4,0)--(4,1)--(5,1);
\draw (-.2,1) node [left] {\tiny $H^1_1$};
\draw (-.2,0) node [left] {\tiny $H^1_2$};
\end{scope}

\begin{scope}[shift={(-.07,.07)}]
\draw [red,densely dotted,very thick](1,2)--(1,1)--(0,1)--(0,0);
\draw [red,densely dotted,very thick](3,2)--(3,0)--(2,0);
\draw [red,densely dotted,very thick](5,2)--(5,1)--(4,1)--(4,0);
\draw (1,2) node [above] {\tiny $V^1_1$};
\draw (3,2) node [above] {\tiny $V^1_2$};
\draw (5,2) node [above] {\tiny $V^1_3$};
\end{scope}
      \end{scope}

\begin{scope}[shift={(8,0)}] %10
      \draw[draw=gray] (-.5,-.5) rectangle ++(2,3);
      \draw[draw=gray] (1.5,-.5) rectangle ++(2,3);
      \draw[draw=gray] (3.5,-.5) rectangle ++(2,3);
      \node at (0,0) {.};
      \node at (1,0) {+};
      \node at (2,0) {.};
      \node at (3,0) {.};
      \node at (4,0) {.};
      \node at (5,0) {+};
      \node at (0,1) {.};
      \node at (1,1) {.};
      \node at (2,1) {+};
      \node at (3,1) {.};
      \node at (4,1) {.};
      \node at (5,1) {.};
      \node at (0,2) {+};
      \node at (1,2) {+};
      \node at (2,2) {+};
      \node at (3,2) {.};
      \node at (4,2) {+};
      \node at (5,2) {.};
\end{scope}
\end{tikzpicture}
\end{center}
\caption{A straight road map and the corresponding concurrent vertex map.}
\label{fig:example concurrent map}
\end{figure}
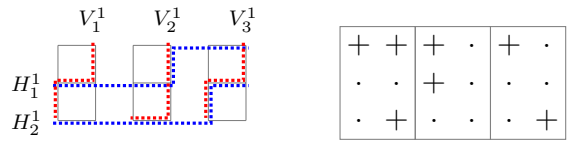
we illustrate an example of a straight road map on $A_1$, which consists of two nonintersecting horizontal paths $H^1_1, H^1_2$ and three nonintersecting vertical paths $V^1_1,V^1_2,V^1_3$. It satisfies the condition that each corner of a horizontal path lies in a vertical path, and each corner of a vertical path lies in a horizontal path.  Figure \ref{fig:example concurrent map} Right is the corresponding concurrent vertex map consisting of points on both a horizontal path and a vertical path. 
\end{example}

As we shall see in Proposition \ref{prop:concurrent equiv}, a concurrent vertex map uniquely determines a straight road map; so there is a unique way to ``connect the dots'' in a concurrent vertex map to obtain the corresponding horizontal and vertical paths. 

As the main results of the paper, we give 
a combinatorial description of ${\rm init }(I_{\mathcal{Q},\mm,\uu})$ in Theorem \ref{thm:prime decomposition of I} (the definitions of ``${\bf u}$-compatible'' and ``concurrent vertex map'' are given in Definition \ref{definition:main}, and the proof is given at the end of \S\ref{subsection:u-compatible sets}) and 
 several properties of $\Delta_{\mathcal{Q},\mm,\uu}$ and $I_{\mathcal{Q},\mm,\uu}$ in Theorem \ref{thm:complex} (proved in Proposition \ref{prop:vertex decomposable} and Proposition \ref{prop:ball}).  
\begin{theorem}\label{thm:prime decomposition of I}(Prime decomposition)
The faces and facets of the Stanley-Reisner simplicial complex $\Delta_{\mathcal{Q},\mm,\uu}$ are determined as follows:

$C\subseteq L$ is a face of $\Delta_{\mathcal{Q},\mm,\uu}$ if and only if $C$ is ${\bf u}$-compatible.

$C\subseteq L$ is a facet of $\Delta_{\mathcal{Q},\mm,\uu}$ if and only if $C$ is a concurrent vertex map.

\noindent Consequently, ${\rm init }(I_{\mathcal{Q},\mm,\uu})=\bigcap_C \; \big\langle x_{ij}^{(k)} \ | \  (i,j,k)\notin C \big\rangle$
where $C$ runs over all concurrent vertex maps.
\end{theorem}

\begin{theorem}\label{thm:complex}
The Stanley-Reisner simplicial complexes $\Delta_{\mathcal{Q},\mm,\uu}$ are vertex decomposable (so are shellable), and homeomorphic to balls. Consequently,  $I_{\mathcal{Q},\mm,\uu}$ is Cohen-Macaulay.
\end{theorem}

As consequences of Theorems \ref{thm:prime decomposition of I} and  \ref{thm:complex},
 we obtain combinatorial formulas for the Hilbert series and multiplicity of the quotient ring ${\bf k}[X]/I_{\mathcal{Q},\mm,\uu}$ (see Corollary \ref{cor:Hilbert series}), and recover some results on $t$-secant varieties by Conca, De Negri, and Stojanac \cite{CNS} (see \S\ref{subsection:Secant varieties}).
We also obtain a combinatorially manifestly positive formula for the $h$-polynomial of $\Delta _{\mathcal {Q},{\mathbf {m}},{\mathbf {u}}}$ in Corollary \ref{cor:h-poly}.

The paper is organized as follows. 
In \S2 we introduce the notion of road maps and concurrent vertex maps.
In \S3 we study the combinatorial property of the simplicial complex $\Delta _{\mathcal {Q},{\mathbf {m}},{\mathbf {u}}}$.
In \S4 we provide a formula for the Hilbert series and multiplicity of the quotient ring ${\bf k}[X]/I_{\mathcal{Q},\mm,\uu}$, and a combinatorially manifestly positive formula for the $h$-polynomial of $\Delta _{\mathcal {Q},{\mathbf {m}},{\mathbf {u}}}$. 
We conclude the paper with examples, algorithms, and discussions in \S5. 
Because of their close relation with classical determinantal varieties, Nakajima's graded quiver varieties, cluster algebras, and nonintersecting paths, we  believe that the bipartite determinantal ideals and the corresponding combinatorial model will garner increased attention for further study.  

We would like to point out some of the results are obtained by N.~Fieldsteel and P.~Klein in \cite{FK} and their unpublished preprint.

{\bf Acknowledgments.} The author gained valuable insights from the collaborated work with Alex Yong on Kazhdan-Lusztig ideals and Schubert varieties,  upon which the current work is based. He also expresses gratitude for discussions with Patricia Klein and Allen Knutson, and for valuable comments from Aldo Conca and an anonymous reader (in particular, his/her suggestion of \cite[Corollary 5.1.14]{BH} which leads to the discovery of Corollary \ref{cor:h-poly}). 
He is very grateful to the referee for thoroughly reviewing the paper and offering numerous valuable comments and suggestions.
Computer calculations were performed using Macaulay2 \cite{M2} and SageMath \cite{Sage}.

Data sharing not applicable to this article as no datasets were generated or analysed during the current study.

\section{Concurrent vertex maps}\label{section:concurrent maps}

\subsection{Notations}
%In the whole paper should assume $1/r\le {u^-}/{v^-}\le r$. 

First we introduce some notations:
\begin{itemize}
\item For $i,j\in\mathbb{Z}$, let $[i,j]_\mathbb{R}$ be the closed interval between $i,j$, and let $[i,j]_\mathbb{Z}=[i,j]_\mathbb{R}\cap\mathbb{Z}$; we also use similar notations for open intervals and half-open half-closed intervals.

%\item For a minor $D$, we denote the corresponding square submatrix $\tilde{D}$.

\item We think the $r$ matrices $X^{(1)},\dots, X^{(r)}$ as ``pages'' in a ``book". The $k$-th page contains $X^{(k)}$, and we sometimes view it as a rectangle region ${\rm Page}_k:=[1,m_{\target(h_k)}]_\mathbb{R}\times[1,m_{\source(h_k)}]_\mathbb{R}$; a lattice point $(i,j)$ on ${\rm Page}_k$ corresponds to the variable $x_{ij}^{(k)}$. 

\noindent (Warning!) In this paper, we use the following (unusual) coordinate system to be consistent with the labelling of entries in a matrix:
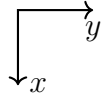
\begin{center}
\begin{tikzpicture}[scale=1]
    % Draw axes
    \draw [<->,thick] (0,-1) node (yaxis) [right] {$x$}
        |- (1,0) node (xaxis) [below] {$y$};
\end{tikzpicture}
\end{center}
Given a matrix $A$ and a point $P=(i,j)$ in the lattice of integer points corresponding to $A$, we denote the $x$ and $y$ coordinates of $P$ as $(P)^A_x=i$, $(P)^A_y=j$.  If $A$ is clear from the context, we denote $(P)^A_x$ as $P_x$ and $(P)^A_y$ as $P_y$.

\item  Define 
$$L:=\{(i,j,k)\in\mathbb{Z}^3 \ | \  1\le k\le r, (i,j)\in [1,m_{\target(h_k)}]_\mathbb{Z}\times [1,m_{\source(h_k)}]_\mathbb{Z}\}. $$
For $\gamma\in {\rm V}_\mathcal{Q}$, 
define an injective map
$$\phi_{\gamma}: [1,a_\gamma]_\mathbb{Z}\times[1,b_\gamma]_\mathbb{Z}\to L, \quad 
(p,q)\mapsto (i,j,k) \textrm{ if the $(p,q)$-entry of $A_\gamma$ is $x_{ij}^{(k)}$}.$$ 
Note that the following two induced maps are both bijective:
$$\phi_{\rm target}: 
\bigcup_{\alpha\in {\rm V}_{\rm target}}
\Big(\alpha,[1,a_\alpha]_\mathbb{Z}\times[1,b_\alpha]_\mathbb{Z}\Big)\to L,
\quad 
(\alpha,i,j)\mapsto \phi_{\alpha}(i,j)
$$
$$\phi_{\rm source}:
\bigcup_{\beta\in {\rm V}_{\rm source}}
\Big(\beta, [1,a_\beta]_\mathbb{Z}\times[1,b_\beta]_\mathbb{Z}\Big)\to L,
\quad
(\beta,i,j)\mapsto\phi_{\beta}(i,j)
$$
So if $h_k$ is an arrow $\beta\to \alpha$, then both $\phi^{-1}_\alpha(i,j,k)$ and $\phi^{-1}_\beta(i,j,k)$ are uniquely defined.

\end{itemize}

\subsection{Straight road maps}
Give a lattice point $P\in\mathbb{Z}^2$ (note that we are not using the usual coordinate system), we say that a point $Q=P+(i,j)$ is:  

-- {\bf east} of $P$ if $i=0, j>0$ (and similarly for {\bf west}, {\bf south}, {\bf north}); 

-- {\bf weakly east} of $P$ if $i=0, j\ge 0$ (and similarly for {\bf weakly west/south/north}); 

-- {\bf NE} of $P$ if $i<0, j>0$ (and similarly for {\bf NW}, {\bf SE}, {\bf SW}); 

-- {\bf weakly NE} of $P$ if $i\le 0, j\ge 0$ (and similarly for  {\bf weakly NW/SE/SW}). 

We define a {\bf SW-NE path} $P\rightsquigarrow Q$ to be a lattice path from its SW endpoint $P$ to its NE endpoint $Q$, with steps either to east along vector $(0,1)$ or to north along vector $(-1,0)$. If we reverse the orientation then we get a {\bf NE-SW path} $Q\rightsquigarrow P$. By a  path we mean either a SW-NE path or a NE-SW path. The length of a path refers to the number of length 1 edges in the path; a path has length 0 if its two endpoints coincide. 

A {\bf corner} of a lattice path $H$ is a vertex $R$ in the path that is adjacent to both a horizontal and a vertical edge of $H$. 
A corner $R$ is called a {\bf NW corner} of $H$ if both points $R+(1,0)$ and $R+(0,1)$ are in $H$; otherwise both $R-(1,0)$ and $R-(0,1)$ must be in $H$, in which case $R$ is called a {\bf SE corner} of $H$.

%Note that a path is uniquely determined by the set of vertices it contains. 
%We call the set of lattice points of a path the {\bf dotted path}. 
A collection $\{H_i\}_{i=1}^u$ of $u$ paths is said to be {\bf nonintersecting} if no vertex is contained in two paths; it is said to be connecting $(X_i)_{i=1}^u$ and $(Y_i)_{i=1}^u$ if $H_i$ has endpoints $X_i$ and $Y_i$ for each $i$.

The restriction of a path $H$ to ${\rm Page}_k$ is denoted $H|_k$. The restriction of a point $P$ to ${\rm Page}_k$ is denoted $P|_k$.  Analogously, for $C\subseteq L$, denote $C|_k$ the restriction of $C$ to ${\rm Page}_k$.

\begin{definition}\label{definition:main}
A {\bf road map} $(\{H^\alpha_{p}\}, \{V^\beta_{q}\})=\big(\{H^\alpha_{p}\}_{\alpha\in {\rm V}_{\rm target},  1\le p\le u_\alpha}, \{V^\beta_{q}\}_{\beta \in {\rm V}_{\rm source}, 1\le q\le u_\beta}\big)$ of rank $\uu$ consists of:

\noindent -- a collection of $u_\alpha$ nonintersecting paths $\{H^\alpha_{p}\}_{p=1}^{u_\alpha}$ (called horizontal paths) of $A_\alpha$ connecting $\big((a_\alpha-u_\alpha+p,1)\big)_{p=1}^{u_\alpha}$ and $\big((p,b_\alpha)\big)_{p=1}^{u_\alpha}$ for $\alpha\in {\rm V}_{\rm target}$, and 

\noindent -- a collection of $u_\beta$ nonintersecting paths $\{V^\beta_{q}\}_{q=1}^{u_\beta}$ (called vertical paths) of $A_\beta$ connecting $\big((1,b_\beta-u_\beta+q)\big)_{q=1}^{u_\beta}$ and $\big((a_\beta,q)\big)_{q=1}^{u_\beta}$ for $\beta\in {\rm V}_{\rm source}$. 

A road map is {\bf straight}  if each corner of a horizontal path lies in a vertical path, and each corner of a vertical path lies in a horizontal path. 
%More precisely: for each $1\le k\le r$, if a point $P$ in is a corner of $H^{\target(h_k)}_p$, then $P|_k\in \bigcup_q V^{\source(h_k)}_q|_k$; if $Q$ is a corner of $V^{\source(h_k)}_q$, then $Q|_k\in \bigcup_p H^{\target(h_k)}_p|_k$.
\end{definition}

Note that a horizontal path defined above may contain vertical edges. We say that a path is \emph{straight horizontal} if it contains no vertical edge;  a path is \emph{straight vertical} if it contains no horizontal edge. 
For example, if for some $\alpha\in V_{\rm target}$, $u_\alpha=a_\alpha$, then the horizontal paths $H^\alpha_p$ are straight horizontal and thus have no corners.

%We have the following useful characterization of straight road maps. 
\begin{lemma}\label{lem:straight 2 conditions}
A road map $(\big(\{H^\alpha_{p}\}, \{V^\beta_{q}\})$ is straight if and only if all of the following conditions hold: 

\noindent {\rm(a)} For each arrow $h_k$ in $\mathcal{Q}$, the intersection $H^{\target(h_k)}_{p}|_k \cap V^{\source(h_k)}_q|_k$ is a connected path $P^k_{pq}\rightsquigarrow Q^k_{pq}$ where $P^k_{pq}$ (resp.~$Q^k_{pq}$) is its SW (resp.~NE) endpoint.

\noindent {\rm(b)} 
Let $r_1<\cdots<r_s$ be the indices such that $h_{r_1},\dots,h_{r_s}$ are all the arrows with target $\alpha$. 
Then $H^\alpha_{p}$ is the unique path obtained by connecting the following vertices and paths from SW to NE by straight horizontal paths: 
$(a_\alpha-u_\alpha+p,1)$, 
$P^{r_1}_{p,1} \rightsquigarrow Q^{r_1}_{p,1}$, 
$\dots$,
 $P^{r_1}_{p,u_{\source(h_{r_1})}} \rightsquigarrow Q^{r_1}_{p,u_{\source(h_{r_1})}}$,
$P^{r_2}_{p,1} \rightsquigarrow Q^{r_2}_{p,1}$, 
$\dots$,
 $P^{r_2}_{p,u_{\source(h_{r_2})}} \rightsquigarrow Q^{r_2}_{p,u_{\source(h_{r_2})}}$,
$\dots$,
$P^{r_s}_{p,1} \rightsquigarrow Q^{r_s}_{p,1}$, 
$\dots$,
 $P^{r_s}_{p,u_{\source(h_{r_s})}} \rightsquigarrow Q^{r_s}_{p,u_{\source(h_{r_s})}}$,
 $(p,b_\alpha)$.
(See Figure \ref{fig:straight}.)
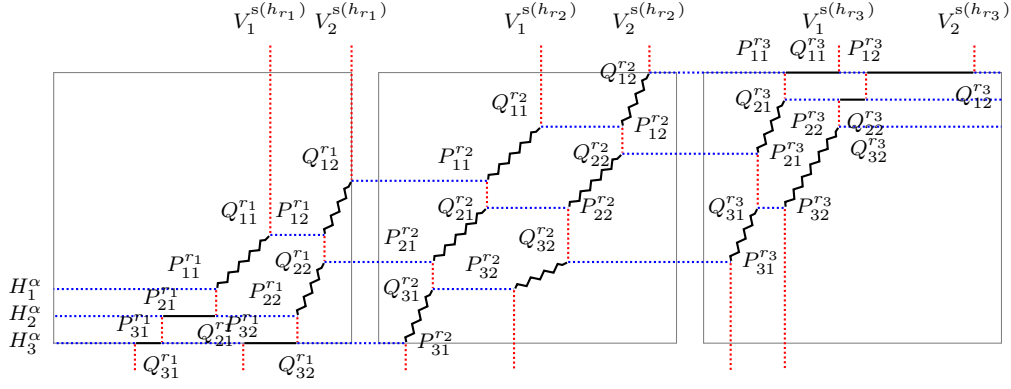
\begin{figure}[ht]
\begin{center}
\begin{tikzpicture}[scale=.72]

\draw  [gray] (0,0.5) rectangle (5.5,5.5);
\draw  [gray] (6,0.5) rectangle (11.5,5.5);
\draw  [gray] (12,0.5) rectangle (17.5,5.5);

\draw [blue,densely dotted,thick](1.5,0.5) node[scale=.1] (v6) {} -- (0,0.5);
\draw [blue,densely dotted,thick](2,1) node[scale=.1]  (v4) {} -- (0,1);
\draw [blue,densely dotted,thick](3,1.5) node[scale=.1]  (v2) {} -- (0,1.5);
\draw [blue,densely dotted,thick](2,0.5) node[scale=.1]  (v5) {} -- (3.5,0.5) node[scale=.1]  (v12) {};
\draw [blue,densely dotted,thick](3,1) node[scale=.1]  (v3) {} -- (4.5,1) node[scale=.1]  (v10) {};
\draw [blue,densely dotted,thick](4,2.5) node[scale=.1]  (v1) {} -- (5,2.5) node[scale=.1]  (v8) {};
\draw [blue,densely dotted,thick](5,2) node[scale=.1]  (v9) {} -- (7,2) node[scale=.1]  (v16) {};
\draw [blue,densely dotted,thick](5.5,3.5) node[scale=.1]  (v7) {} -- (8,3.5) node[scale=.1]  (v14) {}  (4.5,0.5) node[scale=.1]  (v11) {} -- (6.5,0.5) node [scale=.1] (v18) {};
\draw [blue,densely dotted,thick](7,1.5) node[scale=.1]  (v17) {} -- (8.5,1.5) node[scale=.1]  (v24) {};
\draw [blue,densely dotted,thick](9,4.5) node[scale=.1]  (v13) {} -- (10.5,4.5) node [scale=.1] (v20) {};
\draw [blue,densely dotted,thick](11,5.5) node[scale=.1]  (v19) {} -- (13.5,5.5) node[scale=.1]  (v26) {};
\draw [blue,densely dotted,thick](14.5,5.5) node[scale=.1]  (v25) {} -- (15,5.5) node[scale=.1]  (v32) {};
\draw [blue,densely dotted,thick](17,5.5) node[scale=.1]  (v31) {} -- (17.5,5.5);
\draw [blue,densely dotted,thick](8,3) node [scale=.1] (v15) {} -- (9.5,3) node[scale=.1]  (v22) {}  (10.5,4) node[scale=.1]  (v21) {} -- (13,4) node[scale=.1]  (v28) {};
\draw [blue,densely dotted,thick](13.5,5) node [scale=.1] (v27) {} -- (14.5,5) node [scale=.1] (v34) {};
\draw [blue,densely dotted,thick](15,5) node[scale=.1]  (v33) {} -- (17.5,5);
\draw [blue,densely dotted,thick](9.5,2) node[scale=.1]  (v23) {} -- (12.5,2) node [scale=.1] (v30) {}  (13,3) node[scale=.1]  (v29) {} -- (13.5,3) node[scale=.1]  (v36) {};
\draw [blue,densely dotted,thick](14.5,4.5) node[scale=.1]  (v35) {} -- (17.5, 4.5);
\draw [red,densely dotted,thick](4, 6) -- (v1)  (v2) -- (v3)  (v4) -- (v5)  (v6) -- (1.5,0);
\draw [red,densely dotted,thick](5.5, 6) -- (v7)  (v8) -- (v9)  (v10) -- (v11)  (v12) -- (3.5,0);
\draw [red,densely dotted,thick](9, 6) -- (v13)  (v14) -- (v15)  (v16) -- (v17)  (v18) -- (6.5,0);
\draw [red,densely dotted,thick](11,6) -- (v19)  (v20) -- (v21)  (v22) -- (v23)  (v24) -- (8.5,0);
\draw [red,densely dotted,thick](14.5, 6) -- (v25)  (v26) -- (v27)  (v28) -- (v29)  (v30) -- (12.5,0);
\draw [red,densely dotted,thick](17, 6) -- (v31)  (v32) -- (v33)  (v34) -- (v35)  (v36) -- (13.5,0);
\draw [thick] (v6) -- (v5);
\draw [thick] (v4) -- (v3);
\draw [thick] (v12) -- (v11);

\draw [thick, decorate,decoration={zigzag,segment length=2mm, amplitude=0.4mm}] (v2) -- (v1);
\draw [thick, decorate,decoration={zigzag,segment length=2mm, amplitude=0.4mm}] (v10) -- (v9);
\draw [thick, decorate,decoration={zigzag,segment length=2mm, amplitude=0.4mm}] (v8) -- (v7);
\draw [thick, decorate,decoration={zigzag,segment length=2mm, amplitude=0.4mm}] (v14) -- (v13);
\draw [thick, decorate,decoration={zigzag,segment length=2mm, amplitude=0.4mm}](v20) -- (v19);
\draw [thick] (v26) -- (14.5,5.5);
\draw [thick] (v32) -- (v31);
\draw [thick, decorate,decoration={zigzag,segment length=2mm, amplitude=0.4mm}] (v16) -- (v15);
\draw [thick, decorate,decoration={zigzag,segment length=2mm, amplitude=0.4mm}] (v22) -- (v21);
\draw [thick, decorate,decoration={zigzag,segment length=2mm, amplitude=0.4mm}] (v28) -- (v27);
\draw [thick] (v34) -- (v33);
\draw [thick, decorate,decoration={zigzag,segment length=2mm, amplitude=0.4mm}] (v18) -- (v17);
\draw [thick, decorate,decoration={zigzag,segment length=2mm, amplitude=0.4mm}] (v24) -- (v23);
\draw [thick, decorate,decoration={zigzag,segment length=2mm, amplitude=0.4mm}] (v30) -- (v29);
\draw [thick, decorate,decoration={zigzag,segment length=2mm, amplitude=0.4mm}]  (v36) -- (v35);

\draw (v2) node [above left] {\tiny $P^{r_1}_{11}$};
\draw (v1) node [above left] {\tiny $Q^{r_1}_{11}$};

\draw (v7) node [above left] {\tiny $Q^{r_1}_{12}$};
\draw (v8) node [above left] {\tiny $P^{r_1}_{12}$};

\draw (2,0.9) node [above] {\tiny $P^{r_1}_{21}$};
\draw (3,1.1) node [below] {\tiny $Q^{r_1}_{21}$};

\draw (1.5,.4) node [above] {\tiny $P^{r_1}_{31}$};
\draw (v5) node [below] {\tiny $Q^{r_1}_{31}$};

\draw (v9) node [left] {\tiny $Q^{r_1}_{22}$};
\draw (v10) node [above left] {\tiny $P^{r_1}_{22}$};

\draw (v11) node [below] {\tiny $Q^{r_1}_{32}$};
\draw (3.5,.4) node [above] {\tiny $P^{r_1}_{32}$};

\draw (v13) node [above left] {\tiny $Q^{r_2}_{11}$};
\draw (v14) node [above left] {\tiny $P^{r_2}_{11}$};

\draw (v15) node [left] {\tiny $Q^{r_2}_{21}$};
\draw (v16) node [above left] {\tiny $P^{r_2}_{21}$};

\draw (v17) node [left] {\tiny $Q^{r_2}_{31}$};
\draw (v18) node [right] {\tiny $P^{r_2}_{31}$};

\draw (v19) node [left] {\tiny $Q^{r_2}_{12}$};
\draw (v20) node [right] {\tiny $P^{r_2}_{12}$};

\draw (v21) node [left] {\tiny $Q^{r_2}_{22}$};
\draw (v22) node [right] {\tiny $P^{r_2}_{22}$};

\draw (v23) node [above left] {\tiny $Q^{r_2}_{32}$};
\draw (v24) node [above left] {\tiny $P^{r_2}_{32}$};

\draw (v25) node [above left] {\tiny $Q^{r_3}_{11}$};
\draw (v26) node [above left] {\tiny $P^{r_3}_{11}$};

\draw (v27) node [left] {\tiny $Q^{r_3}_{21}$};
\draw (v28) node [right] {\tiny $P^{r_3}_{21}$};

\draw (v29) node [left] {\tiny $Q^{r_3}_{31}$};
\draw (v30) node [right] {\tiny $P^{r_3}_{31}$};

\draw (v31) node [below] {\tiny $Q^{r_3}_{12}$};
\draw (v32) node [above ] {\tiny $P^{r_3}_{12}$};

\draw (v33) node [below] {\tiny $Q^{r_3}_{22}$};
\draw (v34) node [below left] {\tiny $P^{r_3}_{22}$};

\draw (v35) node [below right] {\tiny $Q^{r_3}_{32}$};
\draw (v36) node [right] {\tiny $P^{r_3}_{32}$};

\draw (0,1.5) node [left] {\tiny $H^\alpha_1$};
\draw (0,1) node [left] {\tiny $H^\alpha_2$};
\draw (0,0.5) node [left] {\tiny $H^\alpha_3$};

\draw (4,6) node [above] {\tiny $V^{\mathrm{s}(h_{r_1})}_1$};
\draw (5.5,6) node [above] {\tiny $V^{\mathrm{s}(h_{r_1})}_2$};

\draw (9,6) node [above] {\tiny $V^{\mathrm{s}(h_{r_2})}_1$};
\draw (11,6) node [above] {\tiny $V^{\mathrm{s}(h_{r_2})}_2$};

\draw (14.5,6) node [above] {\tiny $V^{\mathrm{s}(h_{r_3})}_1$};
\draw (17,6) node [above] {\tiny $V^{\mathrm{s}(h_{r_3})}_2$};

\end{tikzpicture}       
\end{center}
\caption{An example of $H^\alpha_p$'s in a straight road map}
\label{fig:straight}
\end{figure}

\noindent {\rm(c)}
Let $r'_1<\cdots<r'_t$ be the indices such that $h_{r'_1},\dots,h_{r'_t}$ are all the arrows with source $\beta$. 
Then $V^\beta_{q}$ 
 is the unique path obtained by connecting the following vertices and paths from NE to SW  by straight vertical paths: 
$(1,b_\beta-u_\beta+q)$, 
$Q^{r'_1}_{1,q} \rightsquigarrow P^{r'_1}_{1,q}$, 
$\dots$,
 $Q^{r'_1}_{u_{\target(h_{r'_1})},q} \rightsquigarrow P^{r'_1}_{u_{\target(h_{r'_1})},q}$,
$Q^{r'_2}_{1,q} \rightsquigarrow P^{r'_2}_{1,q}$, 
$\dots$,
 $Q^{r'_2}_{u_{\target(h_{r'_2})},q} \rightsquigarrow P^{r'_2}_{u_{\target(h_{r'_2})},q}$,
$\dots$,
$Q^{r'_t}_{1,q} \rightsquigarrow P^{r'_t}_{1,q}$, 
$\dots$,
 $Q^{r'_t}_{u_{\target(h_{r'_t})},q} \rightsquigarrow P^{r'_t}_{u_{\target(h_{r'_t})},q}$,
 $(a_\beta,q)$. 
\end{lemma}
\begin{proof}
``$\Leftarrow$''. Assume (a)(b)(c) are satisfied. By (b), any corner of $H^\alpha_p$ lies in $P^{r_k}_{p\ell} \rightsquigarrow Q^{r_k}_{p\ell}$ for some $k,\ell$, therefore lies in a vertical path. Similarly by (c), any corner of $V^\beta_q$ lies in a horizontal path. So the road map is straight. 

``$\Rightarrow$''. Assume the road map is straight.  
We first show that (a) holds. Let $(p,q)$ be the minimum pair in the dictionary order such that $H^\alpha_{p}|_k \cap V^\beta_q|_k$ is not connected.  
Let $P\rightsquigarrow Q$ and $P'\rightsquigarrow Q'$ be the first two SW-most connected components of  this  intersection. 
Denote by $Q\stackrel{H}{\rightsquigarrow} P'$ the subpath of $H^\alpha_p|_k$ from $Q$ to $P'$; by $P'\stackrel{V}{\rightsquigarrow} Q$ the subpath of $V^\beta_q|_k$ from $P'$ to $Q$.
We consider two cases:
Case 1. $P'\stackrel{V}{\rightsquigarrow} Q$ lies to the SE of $Q\stackrel{H}{\rightsquigarrow} P'$ (see Figure \ref{fig:proof of straight} Left).
\begin{figure}[ht]
\begin{center}
\begin{tikzpicture}[scale=.4]
\draw  (0,0) rectangle (10,6);
\draw (3,2) node (v1) {} -- (2.5,2) -- (2.5,1.5) -- (1.5,1.5)  ;
\draw (6.5,4) node (v2) {} -- (6.5,4.5) -- (8,4.5) -- (8,5) -- (8.5,5) -- (8.5,5.5) -- (9,5.5);
\draw [red,dotted, thick] (v1) -- (4.5,2) -- (4.5,2.5) -- (6,2.5) -- (6,3.5) -- (6.5,3.5) -- (v2);
\draw [blue, dotted, thick] (v1) -- (3,3) -- (4,3) -- (4,3.5) -- (5,3.5) -- (5,4) -- (v2);
\draw (1.5,1.5) node [below] {\tiny $P$};
\draw (v1) node [below] {\tiny $Q$};
\draw (v2) node [right] {\tiny $P'$};
\draw (9,5.5) node [below] {\tiny $Q'$};
\draw (3,3) node [below right] {\tiny $R$};
\draw [red, dotted, thick] (3.5,5) -- (3.5,3.5) -- (3,3.5) -- (3,3) -- (2,3) -- (2,2.5) -- (0.5,2.5)  ;
\draw (3.5,4.5) node [red,left] {\tiny $V^\beta_{q'}|_k$};
\draw (7,3) node [red, below] {\tiny $P'\stackrel{V}{\rightsquigarrow} Q$};
\draw (5,4) node [blue,above] {\tiny $Q\stackrel{H}{\rightsquigarrow} P'$};
\begin{scope}[shift={(12,0)}]
\draw  (0,0) rectangle (10,6);
\draw (3,2) node (v1) {} -- (2.5,2) -- (2.5,1.5) -- (1.5,1.5)  ;
\draw (6.5,4) node (v2) {} -- (6.5,4.5) -- (8,4.5) -- (8,5) -- (8.5,5) -- (8.5,5.5) -- (9,5.5);
\draw [blue,dotted, thick] (v1) -- (4.5,2) -- (4.5,2.5) -- (6,2.5) -- (6,3.5) -- (6.5,3.5) -- (v2);
\draw [red, dotted, thick] (v1) -- (3,3) -- (4,3) -- (4,3.5) -- (5,3.5) -- (5,4) -- (v2);
\draw (1.5,1.5) node [below] {\tiny $P$};
\draw (v1) node [below] {\tiny $Q$};
\draw (v2) node [right] {\tiny $P'$};
\draw (9,5.5) node [below] {\tiny $Q'$};
\draw (3,3) node [below right] {\tiny $R$};
\draw [blue, dotted, thick] (3.5,5) -- (3.5,3.5) -- (3,3.5) -- (3,3) -- (2,3) -- (2,2.5) -- (0.5,2.5)  ;
\draw (3.5,4.5) node [red,left] {\tiny $H^\alpha_{q}|_k$};
\draw  (5,4) node [red,above] {\tiny $P'\stackrel{V}{\rightsquigarrow} Q$};
\draw  (7,3) node [blue, below] {\tiny $Q\stackrel{H}{\rightsquigarrow} P'$};
\end{scope}
\end{tikzpicture}       
\end{center}
\caption{Proof of Lemma \ref{lem:straight 2 conditions}. Left: Case 1; Right: Case 2.}
\label{fig:proof of straight}
\end{figure}
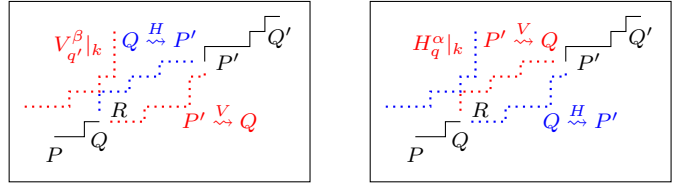
Let $R$ be the NW corner of $H^\alpha_p|_k$ in the interior of $Q\stackrel{H}{\rightsquigarrow} P'$ that is closest to $Q$. Then $R$ lies in $V^{\beta}_{q'}|_k$ for some $q'<q$,
and $H^\alpha_p|_k\cap V^\beta_{q'}|_k$ is not connected,
contradicting the choice of $(p,q)$.  
  %indeed,  the intersection contains $R$ and also a point that is weakly SW of $P$ (because $H^\alpha_p|_k$ reaches the left boundary and $V^\beta_{q'}|_k$ reaches the bottom boundary), but not $P$. This contradicts the choice of $(p,q)$.
%
Case 2.  $P'\stackrel{V}{\rightsquigarrow} Q$ lies to the NW of $Q\stackrel{H}{\rightsquigarrow} P'$ (see Figure \ref{fig:proof of straight} Right). 
The proof is similar to Case 1 so we skip.
%Let $R$ be the NW corner of $V^\beta_q|_k$ in the interior of $P'\stackrel{V}{\rightsquigarrow} Q$ that is closest to $Q$. Then $R$ lies in some $H^{\alpha}_{p'}|_k$ with $p'<p$, and $H^\alpha_{p'}|_k\cap V^\beta_{q}|_k$ is not connected, contradicting the choice of $(p,q)$.
%
This proves (a). 
%
%Next we show (b). The path $H^\alpha_p$ starts at $(a_\alpha-u_\alpha+p,1)$. It cannot make a turn before arriving at $P^{r_1}_{p,1}$ by the definition of straight road map. So the subpath from $(a_\alpha-u_\alpha+p,1)$ to  $P^{r_1}_{p,1}$ is either straight horizontal or straight vertical. The former case is what we expected. In the latter case,  $V^{\source(h_{r_1})}_1|_k$ passes through $P^{r_1}_{p,1}$, so it stays in the left-most column after it going below $P^{r_1}_{p,1}$, therefore passes the point $(a_\alpha-u_\alpha+p,1)$. This forces $P^{r_1}_{p,1}=(a_\alpha-u_\alpha+p,1)$, so it is still what we expected. 
%The path $H^\alpha_p$ then goes along the path $P^{r_1}_{p,1} \rightsquigarrow Q^{r_1}_{p,1}$ and arrives at  $Q^{r_1}_{p,1}$.
%
%After that, consider $Q^{r_1}_{p,1}\stackrel{H}{\rightsquigarrow} P^{r_1}_{p,2}$. It cannot make any turns, so is either straight horizontal or straight vertical. The former case is what we expected. So we assume the latter case, that is, $P^{r_1}_{p,2}$ is to the north of  $Q^{r_1}_{p,1}$.  Denote $\beta=\source(h_{r_1})$; then $P^{r_1}_{p,2}$ lies in $V^\beta_2|_k$, $Q^{r_1}_{p,1}$ lies in $V^\beta_1|_k$.  The path  $V^\beta_2|_k$ must pass through a point weakly left of $Q^{r_1}_{p,1}$, contradicting the fact that  $V^\beta_2|_k$ is on the right side of $V^\beta_1|_k$. 
%
%The rest of (b) is proved in the same way.  
%(c) is proved similarly to (b), so we skip. 
The parts (b) and (c) follow easily. 
\end{proof}

\subsection{Concurrent vertex maps and $\uu$-Compatible sets}\label{subsection:u-compatible sets}
\begin{definition}\label{def:concurrent vertex map}
A subset $C\subseteq L$   is called a {\bf concurrent vertex map} if there exists a straight road map  $(\{H^i_{p}\}, \{V^j_{q}\})$ such that 
\begin{equation}\label{eq:Ck}
 C|_k= \Big( \bigcup_{p=1}^{u_{\target(h_k)}} H^{\target(h_k)}_p|_k\Big)\
\bigcap\
\Big( \bigcup_{q=1}^{u_{\source(h_k)}} V^{\source(h_k)}_q|_k\Big) 
\bigcap\mathbb{Z}^2, \quad \text{ for } k=1,\dots,r.
\end{equation}
\end{definition}

\begin{definition}\label{def:long}
We say that a finite subset $S\subset\mathbb{Z}^2$ is a {\bf diagonal chain} of size $s$ if $S=\{(i_1,j_1),\dots,(i_s,j_s)\}$ satisfies $i_1<i_2<\cdots<i_s$ and $j_1<j_2<\cdots<j_s$. 

In the rest of the definition, fix $C \subseteq L$. 

For each $\alpha\in {\rm V}_{\rm target}$, define a subset $C^\alpha = \phi_\alpha^{-1}(C) \subseteq [1,a_\alpha]_\mathbb{Z}\times [1,b_\alpha]_\mathbb{Z}$.
For $(i,j,k)\in\mathbb{Z}^3$, let $\alpha={\rm t}(h_k)$, $(i,j')=\phi_\alpha^{-1}(i,j,k)$, and define the following constants $r_{ijk}$, $\bar{r}_{ijk}$, ... (later we may 
use $r_{ijk}(C)$, $\bar{r}_{ijk}(C)$, $\dots$ in place of  $r_{ijk}$, $\bar{r}_{ijk}$, $\dots$ to indicate their dependence on $C$; we also use $r_P$ to denote $r_{ijk}$ if $P=(i,j,k)$, and similar for others.)

$r_{ijk} = r^\alpha_{ij'}:=$ max size of diagonal chains in $C^\alpha$ contained in $(-\infty,i-1]_\mathbb{Z}\times (-\infty,j'-1]_\mathbb{Z}$;

$\bar{r}_{ijk} =\bar{r}^\alpha_{ij'} :=\max(r_{ijk},\min(i-1,u_{\alpha}-1-\min(a_\alpha-i,b_\alpha-j')))=$ max size of diagonal chains of $\bar{C}^\alpha:=C^\alpha\cup \{(a_\alpha-u_\alpha+t,-u_\alpha+t) | 1\le t\le u_\alpha\}\cup\{(t,b_\alpha-u_\alpha+t)|1\le t\le u_\alpha\}$ contained in $(-\infty,i-1]_\mathbb{Z}\times(-\infty,j'-1]_\mathbb{Z}$;

$s_{ijk} = s^\alpha_{ij'} :=$ max size of diagonal chains of $C^\alpha$ contained in $[i+1,\infty)_\mathbb{Z}\times [j'+1,\infty)_\mathbb{Z}$;

$\bar{s}_{ijk} = \bar{s}^\alpha_{ij'} :=\max(s_{ijk},\min(a_\alpha-i,u_{\alpha}-\min(i,j')))=$   max size of diagonal chains of $C^\alpha\cup \{(a_\alpha-u_\alpha+t,t) | 1\le t\le u_\alpha\}\cup\{(t,b_\alpha+t)|1\le t\le u_\alpha\}$ contained in $[i+1,\infty)_\mathbb{Z}\times [j'+1,\infty)_\mathbb{Z}$. 

\smallskip

Similarly, for $\beta\in {\rm V}_{\rm source}$,  define a subset $C^\beta = \phi_\beta^{-1}(C) \subseteq [1,a_\beta]_\mathbb{Z}\times [1,b_\beta]_\mathbb{Z}$.
For $(i,j,k)\in\mathbb{Z}^3$, let $\beta={\rm s}(h_k)$, $(i',j)=\phi_\beta^{-1}(i,j,k)$, and 

$r'_{ijk} = r^\beta_{i'j}  :=$ max size of diagonal chains in $C^\beta$ contained in $(-\infty,i'-1]_\mathbb{Z}\times (-\infty,j-1]_\mathbb{Z}$;

$\bar{r}'_{ijk} = \bar{r}^\beta_{i'j} :=\max(r'_{ijk},\min(j-1,u_{\beta}-1-\min(b_\beta-j,a_\beta-i')))$; %=max size of diagonal chains of $C^\beta\cup \{(-u_\beta+t,b_\beta-u_\beta+t) | 1\le t\le u_\beta\}\cup\{(a_\beta-u_\beta+t,t)|1\le t\le u_\beta\}$ contained in $(-\infty,i'-1]_\mathbb{Z}\times(-\infty,j-1]_\mathbb{Z}$;

$s'_{ijk} = s^\beta_{i'j} :=$ max size of diagonal chains of $C^\beta$ contained in $[i'+1,\infty)_\mathbb{Z}\times [j+1,\infty)_\mathbb{Z}$;

$\bar{s}'_{ijk} = \bar{s}^\beta_{i'j} :=\max(s'_{ijk},\min(b_\beta-j,u_{\beta}-\min(i',j)))$.% =max size of diagonal chains of $C^\beta\cup \{(t,b_\beta-u_\beta+t) | 1\le t\le u_\beta\}\cup\{(a_\beta+t,t)|1\le t\le u_\beta\}$ contained in $[i'+1,\infty)_\mathbb{Z}\times [j+1,\infty)_\mathbb{Z}$.

We say that a subset $C\subseteq L$ is {\bf $\uu$-compatible} if for each $\gamma\in {\rm V}_\mathcal{Q}$, $C^\gamma$ does not contain any diagonal chain of size $(u_\gamma+1)$. 
We say that $C\subseteq L$ is {\bf maximal $\uu$-compatible} if $C$ is $\uu$-compatible and there is no $\uu$-compatible $C'\subseteq L$ satisfying $C\subsetneq C'$. 
\end{definition}

\begin{proposition}\label{prop:concurrent equiv}
For $C\subseteq L$, the following are equivalent:

{\rm(i)} $C$ is a concurrent vertex map;

{\rm(ii)} $C$ is maximal $\uu$-compatible; 

{\rm(iii)} $C$ is $\uu$-compatible and 
\begin{equation}\label{eq:cardinality of max C}
|C|= N_{\mathcal{Q},\mm,\uu}:= \sum_k u_{\source(h_k)}u_{\target(h_k)}+\sum_{\alpha\in {\rm V}_{\rm target}}u_\alpha(a_\alpha-u_\alpha)+\sum_{\beta\in {\rm V}_{\rm source}}u_\beta(b_\beta-u_\beta)
\end{equation}

{\rm(iv)} $C$ satisfies the condition that a point $(i,j,k)\in L$ is in $C$ if and only if $r_{ijk}+s_{ijk}<u_{\target(h_k)}$ and $r'_{ijk}+s'_{ijk}<u_{\source(h_k)}$.

{\rm(v)} $C$ satisfies the condition that $\bar{r}_{ijk}+\bar{s}_{ijk}\in\{u_{\target(h_k)}-1,u_{\target(h_k)}\}$ and $\bar{r}'_{ijk}+\bar{s}'_{ijk}\in\{u_{\source(h_k)}-1,u_{\source(h_k)}\}$, and $(i,j,k)\in L$ is in $C$ if and only if $\bar{r}_{ijk}+\bar{s}_{ijk}=u_{\target(h_k)}-1$ and $\bar{r}'_{ijk}+\bar{s}'_{ijk}=u_{\source(h_k)}-1$.

Moreover, if $C$ is a concurrent vertex map, then the corresponding road map is uniquely determined by the following conditions:

 $(i,j)\in H^{\target(h_k)}_p|_k$ if and only if $\bar{r}_{ijk}=p-1$ and $\bar{s}_{ijk}=u_{\target(h_k)}-p$; 
 
 $(i,j) \in V^{\source(h_k)}_q|_k$ if and only if $\bar{r}'_{ijk}=q-1$ and  $\bar{s}'_{ijk}=u_{\source(h_k)}-q$.

%[Need to add the following and its proof: a point $(i,j)$ (may or may not be in $C$) is strictly below $H_p$ and weakly above $H_{p+1}$ if and only if $\bar{a}_{ij}=p$, is strictly to the right of $V_q$ and weakly to the left of $V_{q+1}$ if and only if $\bar{b}_{ij}=q$.]
\end{proposition}

Before proving the proposition, we first prove Lemmas \ref{contained in path}, \ref{lemma:u-compatible = existence of road map}, \ref{lem:cd}.

\begin{lemma}\label{contained in path}
 Let $a,b,u$ be positive integers. (See Figure \ref{fig:lemma contained in path}.)
 
{\rm(i)}  Assume $u\le \min(a,b)$. 
Let $X_p=(a-u+p,p)$, $Y_p=(p,b-u+p)$ for $1\le p\le u$.
Let $R_{\rm hex}\subseteq\mathbb{R}^2$ be the closed hexagon region with vertices 
$(1,1)$, $X_1$, $X_u$, $(a,b)$, $Y_u$, $Y_1$, and let $C\subseteq R_{\rm hex}\cap\mathbb{Z}^2$. Then 
``there exits a collection of $u$ nonintersecting paths connecting $(X_p)_{p=1}^u$ and $(Y_p)_{p=1}^u$ whose union contains $C$''
$\Leftrightarrow$ 
``all diagonal chains of $C$ have sizes $\le u$''.

\begin{figure}[ht]
\begin{center}
\begin{tikzpicture}[scale=.3]
\fill[blue!10](0,3)--(3,0)--(12,0)--(12,7)--(9,10)--(0,10)--(0,3);
\draw(1,9) node[]{\tiny$R$};
\draw(0,10) node[left]{\tiny$(1,1)$};
\draw(12,0) node[right]{\tiny$(a,b)$};
\draw  (0,0) rectangle (12,10);
\draw[fill] (0,3) circle(3pt); \draw(0,3) node[left]{\tiny$X_1=(a-u+1,1)$};
\draw[fill] (1,2) circle(3pt); \draw(1,2) node[left]{\tiny$X_2=(a-u+2,2)$};
\draw[fill] (1.7,1.3) circle(2pt);
\draw[fill] (2,1) circle(2pt);
\draw[fill] (2.3,0.7) circle(2pt);
\draw[fill] (3,0) circle(3pt); \draw(3,0) node[below ]{\tiny$X_u=(a,u)$};
\draw[fill] (9,10) circle(3pt); \draw(9,10) node[above]{\tiny$Y_1=(1,b-u+1)$};
\draw[fill] (10,9) circle(3pt); \draw(10,9) node[right]{\tiny$Y_2=(2,b-u+2)$};
\draw[fill] (10.7,8.3) circle(2pt);
\draw[fill] (11,8) circle(2pt);
\draw[fill] (11.3,7.7) circle(2pt);
\draw[fill] (12,7) circle(3pt); \draw(12,7) node[right]{\tiny$Y_u=(u,b)$};
\draw[red,thick] (0,3)-- (1,3)--(1,5) -- (2,5)--(2,7)--(6,7)--(6,10)--(9,10)  ;
\draw[red,thick] (1,2)-- (3,2) -- (3,4)--(5,4)--(5,5)--(8,5)--(8,8)--(10,8)--(10,9)  ;
\draw[red,thick] (3,0)-- (6,0) -- (6,2)--(10,2)--(10,6)--(11,6)--(11,7)--(12,7)  ;
\begin{scope}[shift={(20,0)}]
\fill[blue!10](0,0)--(12,0)--(12,10)--(0,10);--(0,0);
%\draw(1,9) node[]{\tiny$R$};
\draw(0,10) node[left]{\tiny$(1,1)$};
\draw(12,0) node[right]{\tiny$(a,b)$};
\draw  (0,0) rectangle (12,10);
\draw[fill] (0,3) circle(3pt);% \draw(0,3) node[left]{\tiny$X_1=(m-w+1,1)$};
\draw[fill] (0,2) circle(3pt);% \draw(1,2) node[left]{\tiny$X_2=(m-w+2,2)$};
\draw[fill] (0,1.3) circle(2pt);
\draw[fill] (0,1) circle(2pt);
\draw[fill] (0,0.7) circle(2pt);
\draw[fill] (0,0) circle(3pt);% \draw(3,0) node[below ]{\tiny$X_w=(m,w)$};
\draw[fill] (12,10) circle(3pt);% \draw(9,10) node[above]{\tiny$Y_1=(1,n-w+1)$};
\draw[fill] (12,9) circle(3pt);% \draw(10,9) node[right]{\tiny$Y_2=(2,n-w+1)$};
\draw[fill] (12,8.3) circle(2pt);
\draw[fill] (12,8) circle(2pt);
\draw[fill] (12,7.7) circle(2pt);
\draw[fill] (12,7) circle(3pt);% \draw(12,7) node[right]{\tiny$Y_w=(w,n)$};
\draw[fill] (0,3) circle(3pt);
\draw[fill] (1,2) circle(3pt);
\draw[fill] (1.7,1.3) circle(2pt);
\draw[fill] (2,1) circle(2pt);
\draw[fill] (2.3,0.7) circle(2pt);
\draw[fill] (3,0) circle(3pt);=
\draw[fill] (9,10) circle(3pt);
\draw[fill] (10,9) circle(3pt);
\draw[fill] (10.7,8.3) circle(2pt);
\draw[fill] (11,8) circle(2pt);
\draw[fill] (11.3,7.7) circle(2pt);
\draw[fill] (12,7) circle(3pt);
\draw[red,thick] (0,3)-- (1,3)--(1,5) -- (2,5)--(2,7)--(6,7)--(6,10)--(12,10)  ;
\draw[red,thick] (0,2)-- (3,2) -- (3,4)--(5,4)--(5,5)--(8,5)--(8,8)--(10,8)--(10,9)--(12,9)  ;
\draw[red,thick] (0,0)-- (6,0) -- (6,2)--(10,2)--(10,6)--(11,6)--(11,7)--(12,7)  ;
\end{scope}
\end{tikzpicture}       
\end{center}
\caption{Lemma \ref{contained in path}. Left is (i), the shaded hexagon region is $R_{\rm hex}$; Right is (ii).}
\label{fig:lemma contained in path}
\end{figure}
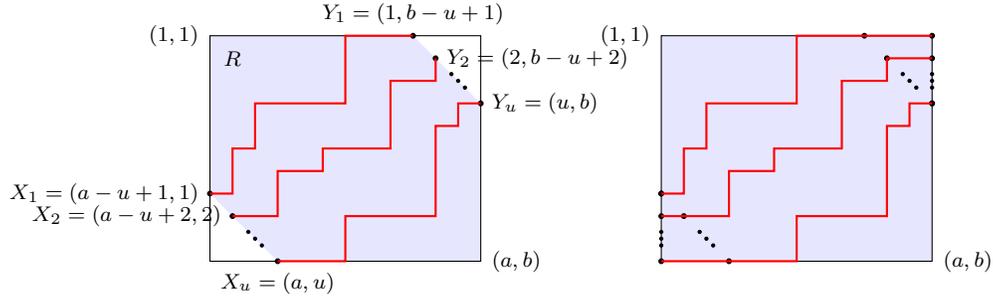

For {\rm(ii)} and {\rm(ii')},  assume $C\subseteq [1,a]_\mathbb{Z}\times[1,b]_\mathbb{Z}$.

{\rm(ii)} If $u < a$, then ``there exists a collection of $u$ nonintersecting paths connecting $\big((a-u+p,1)\big)_{p=1}^u$ and $\big((p,b)\big)_{p=1}^u$ whose union contains $C$'' $\Leftrightarrow$ `` $u\le b$ and all diagonal chains of $C$ have sizes $\le u$''. 
If $u=a$, then (regardless of $b$) there exists a unique collection of $u$ nonintersecting paths connecting $\big((a-u+p,1)\big)_{p=1}^u$ and $\big((p,b)\big)_{p=1}^u$ whose union contains $C$, where each path is straight horizontal.

{\rm(ii')} If $u<b$, then ``there exists a collection of $u$ nonintersecting paths connecting $\big((1,b-u+q)\big)_{q=1}^u$ and $\big((a,q)\big)_{q=1}^u$ whose union contains $C$'' $\Leftrightarrow$ `` $u\le a$ and all diagonal chains of $C$ have sizes $\le u$''. 
If $u=b$, then (regardless of $a$) there exists a unique collection of $u$ nonintersecting paths connecting $\big((1,b-u+q)\big)_{q=1}^u$ and $\big((a,q)\big)_{q=1}^u$ whose union contains $C$, where each path is straight vertical.
\end{lemma}
\begin{proof}
(ii) and (ii') follow easily from (i). To prove (i):
``$\Rightarrow$'': %assume that there exits a collection of $u$ nonintersecting paths connecting $(X_p)_{p=1}^u$ and $(Y_p)_{p=1}^u$ whose union contains $C$. 
For any diagonal chain $D\subseteq C$, different points in $D$ lie in different paths, thus $|D|\le  u$; 
``$\Leftarrow$'': assume $C\subseteq R_{\rm hex}$ such that all diagonal chains of $C$ have sizes $\le u$. Without loss of generality, we can further assume that $C$ contains all $X_p$, $Y_p$ for $1\le p\le u$. Indeed, if a diagonal chain $D$  in $C\cup\{X_p\}$ has size  $> u$, then it cannot contain the point $X_p$, otherwise $D$ contains at most $p-1$ points to the NW to $X_p$ and at most $u-p$ points to the SE of $X_p$ which implies $|D|\le (p-1)+(u-p)+1=u$.  

Let $r_{ij}:=$ max size of diagonal chains in $C\cap \big([1,i-1]\times [1,j-1]\big)$ for $i,j\in\mathbb{Z}$.
Then we get the desired nonintersecting paths by the following:

%\smallskip

\noindent  {\bf SE-most construction:} 
for $p=1,\dots,u$, define the path $H_p$ to be the SE boundary of the set $\{(i,j)\in R_{\rm hex} : r_{ij}=p-1\}$, that is, a point $(i,j)\in R_{\rm hex}$ is in $H_p$ if and only if ``$r_{ij}=p-1$ and $r_{i+1,j+1}=p$''.

Intuitively, $H_p$ is the unique path satisfying the following:
(i) it joins $X_p$ and $Y_p$ such that the maximal size of a diagonal chain in $C$ to the  NW of $H_p$ is $p-1$; 
(ii) any other path satisfying (i) does not contain any point to the SE of $H_p$. 
\end{proof}

\begin{remark}
%(a) The SE-most construction can be equivalently described as follows: 

%First construct $H_1$: starting from $X_1$ and always walk east unless there is a point of $C$ in the north, in which case walk to that point; keep walking like this until it reaches $Y_1$;

%Then construct $H_2$: starting from $X_2$ and always walk east unless there is a point of $C\setminus H_1$ in the north, in which case walk to that point; keep walking like this until it reaches $Y_2$;

%$\vdots$

%Finally construct $H_u$: starting from $X_u$ and always walk east unless there is a point of $C\setminus (H_1\cup \cdots\cup H_{u-1})$ in the north, in which case walk to that point; keep walking like this until it reaches $Y_u$.

%(b) 
Let $s_{ij}=$ max size of diagonal chains in $C\cap ([i+1,a] \times [j+1,b])$ for $i,j\in\mathbb{Z}$. 
There is also a
{\bf NW-most construction:} for $p=1,\dots,u$, define the path $H_p$ to the NW boundary of $\{(i,j)\in R_{\rm hex} : s_{ij}=p-1\}$, that is, a point $(i,j)\in R_{\rm hex} $ is in $H_p$ if and only if ``$s_{ij}=p-1$ and $s_{i-1,j-1}=p$''.

\end{remark}

\begin{lemma} \label{lemma:u-compatible = existence of road map}
For a subset $C \subseteq L$, the following are equivalent:

{\rm(i)} $C$ is $\uu$-compatible;

{\rm(ii)} there exists a (not necessarily straight) road map $(\{H^\alpha_{p}\}, \{V^\beta_{q}\})$ such that
\begin{equation}\label{eq:C in HV}
C|_k\subseteq\Big( \bigcup_{p=1}^{u_{\target(h_k)}} H^{\target(h_k)}_p|_k\Big)\
\bigcap\
\Big( \bigcup_{q=1}^{u_{\source(h_k)}} V^{\source(h_k)}_q|_k\Big), \quad \textrm{ for all } k.
\end{equation}
%In this case, we can use the {\bf SE-most construction} to create such a road map as follows: 
 
%$H^\alpha_p\cap\mathbb{Z}^2=\{(i,j')\in [1,a_\alpha]_\mathbb{Z}\times [1,b_\alpha]_\mathbb{Z} \ | \  \bar{r}_{ij'}=p-1, \bar{r}_{i+1,j'+1}=p\}$ where $\bar{r}_{ij'}:=\bar{r}_{ijk}$ with $(i,j,k)=\phi_\alpha(i,j')$; 

%$V^\beta_p\cap\mathbb{Z}^2=\{(i',j)\in [1,a_\alpha]_\mathbb{Z}\times [1,b_\alpha]_\mathbb{Z} \ | \  \bar{r}_{i'j}=q-1, \bar{r}_{i'+1,j+1}=q\}$ where $\bar{r}_{i'j}:=\bar{r}_{ijk}$ with $(i,j,k)=\phi_\beta(i',j)$.

\end{lemma}

\begin{proof}
(ii)$\Rightarrow$(i): assume such a road map $(\{H^\alpha_{p}\}, \{V^\beta_{q}\})$ exists. Then for $\alpha\in {\rm V}_{\rm target}$, $C^\alpha$ is contained in a collection of nonintersecting paths $\{H^\alpha_p\}_{p=1}^{u_\alpha}$, so Lemma \ref{contained in path} (ii) implies that $C^\alpha$ does not contain any diagonal chain of size $(u_\alpha+1)$; similarly,  for $\beta\in {\rm V}_{\rm source}$, Lemma \ref{contained in path} (ii') implies that $C^\beta$ does not contain any diagonal chain of size $(u_\beta+1)$. Thus $C$ is ${\bf u}$-compatible. 
(i)$\Rightarrow$(ii): 
 assume $C$ is  ${\bf u}$-compatible. 
 By Lemma \ref{contained in path},  
for $\alpha\in {\rm V}_{\rm target}$,  $C^\alpha$ is contained in a collection of nonintersecting paths $\{H^\alpha_p\}_{p=1}^{u_\alpha}$; 
for $\beta\in {\rm V}_{\rm source}$,  $C^\beta$ is contained in a collection of nonintersecting paths $\{V^\beta_q\}_{q=1}^{u_\beta}$. 
This gives a desired road map $(\{H^\alpha_{p}\}, \{V^\beta_{q}\})$.
 \end{proof}

\begin{lemma}[Straightening Lemma]\label{lem:cd}
Let $u,a,b$ be positive integers, $u\le \min(a-1,b)$. Assume  $c_{pq}, d_{pq}\in [1,a]_\mathbb{Z}$  (for $1\le p\le u$ and $1\le q\le b$) satisfy the following inequalities for all $p,q$ (whenever the inequalities make sense):
$d_{pq}\ge c_{pq}, \quad c_{pq}\ge d_{p,q+1}, \quad c_{pq}>d_{p-1,q}.$ 
Then there exist integers $c'_{pq}$, $d'_{pq}$  ($1\le p\le u$ and $1\le q\le b$) satisfying 

{\rm(i)} $d'_{pq}\ge c'_{pq}, \quad c'_{pq}>d'_{p-1,q}$ for all $p,q$ (whenever the inequalities make sense);

{\rm(ii)} $c'_{pq}=d'_{p,q+1}$ for 
$1\le p\le u$ and $1\le q\le b-1$;

{\rm(iii)} $d'_{p1}=a-u+p$ and $c'_{pb}=p$ for $1\le p\le u$;

{\rm(iv)}   
$\bigcup_p[c_{pq},d_{pq}]_\mathbb{Z} \subseteq \bigcup_p[c'_{pq},d'_{pq}]_\mathbb{Z}$
for $q=1,\dots,b$;
moreover, if the `` $\subseteq$'' is actually `` $=$'' for all $q$, then $c_{pq}=c'_{pq}$, $d_{pq}=d'_{pq}$ for all $p,q$.
\end{lemma}

\begin{proof}
Define
$C :=\bigcup_{p=1}^a\bigcup_{q=1}^b  \{(i,q) : i\in [c_{pq},d_{pq}]_\mathbb{Z} \}\subseteq [1,a]_\mathbb{Z}\times[1,b]_\mathbb{Z}.$
We claim that there exists a collection of $u$ nonintersecting paths $\{H_p\}_{p=1}^u$ connecting $\big((a-u+p,1)\big)_{p=1}^u$ and $\big((p,b)\big)_{p=1}^u$ whose union contains $C$. 
Indeed,  Let $D$ be a diagonal chain of $C$. For every $p\in[1,u]$, let $L_p$ be the broken line connecting $(d_{p1},1)\to (c_{p1},1) \to 
(d_{p2},2)\to (c_{p2},2)\to \cdots \to (d_{pb},b)\to (c_{pb},b)$.
Each $L_p$ contains at most one point in $D$. Therefore $D$ has at most $u$ points. Then the claim follows from Lemma \ref{contained in path} (ii). 
Now let $[c'_{pq},d'_{pq}]$ be the closed interval (possibly degenerated to a point) that is the intersection of $H_p$ and the vertical line $y=q$. Then obviously $c'_{pq},d'_{pq}$ satisfy (i)--(iv) (except the ``moreover'' part). 

Next, we prove the ``moreover'' part in (iv). Consider the sum:
$$\aligned
\sum_{p,q}(d_{pq}-c_{pq})
&=\sum_p\Big(d_{p1}-c_{pb}-\sum_{q=1}^{b-1}(c_{pq}-d_{p,q+1})\Big)
\le\sum_p\Big(d_{p1}-c_{pb}\Big)\\
&
\le \Big((a-u+1)+\cdots+a\Big)-\Big(1+\cdots+u)\big)=(a-u)u,
\endaligned
$$
where the two equalities hold if and only if for all $p,q$, $c_{pq}=d_{p,q+1}, \;  d_{p1}=a-u+p, \; c_{pb}=p. $
As a consequence,
$
\sum_{p,q}(d'_{pq}-c'_{pq})
=(a-u)u.
$
Suppose `` $\subseteq$'' in (iv)  is actually `` $=$'' for all $q$, that is,
\begin{equation}\label{eq:subset is equal}
S_q:= \bigcup_p[c_{pq},d_{pq}]_\mathbb{Z} = \bigcup_p[c'_{pq},d'_{pq}]_\mathbb{Z}.
\end{equation}
So 
$\sum_{p=1}^u(d_{pq}-c_{pq}+1)|=\sum_{p=1}^u(d'_{pq}-c'_{pq}+1) \text{ for every }q, $
thus
%$$\sum_{p=1}^u(d_{pq}-c_{pq})=\sum_{p=1}^u(d'_{pq}-c'_{pq}),  \text{ for every }q.$$ 
$\sum_{p,q}(d_{pq}-c_{pq})=\sum_{p,q}(d'_{pq}-c'_{pq})=(a-u)u.$
Then $(\{c_{pq}\},\{d_{pq}\})$ also satisfies (ii) and (iii). 
Now we prove $c_{pq}=c'_{pq}$ and $d_{pq}=d'_{pq}$ by induction on $q$: 
for the base case $q=1$, we have $d_{p1}=a-u+p=d'_{p1}$ for all $p$; $c_{p1}=c'_{p1}=a-u+p$ for $p>1$; taking the minimum integer on the two sides of \eqref{eq:subset is equal} for $q=1$ gives $c_{11}=c'_{11}$. 
for the inductive step, let $q\ge2$, assume $c_{pq'}=c'_{pq'}$ and $d_{pq'}=d'_{pq'}$ for $q'< q$ and all $p$, we want to show that 
$c_{pq}=c'_{pq}$ and $d_{pq}=d'_{pq}$ for all $p$. Indeed, $d_{pq}=c_{p,q-1}=c'_{p,q-1}=d'_{pq}$; 
taking the minimum integer on the two sides of \eqref{eq:subset is equal} gives $c_{1q}=c'_{1q}$; for $p>1$,  
$c_{pq}=\min(S_q\cap[d_{p-1,q}+1,a]) 
=\min(S_q\cap[c_{p-1,q-1}+1,a])
=\min(S_q\cap[c'_{p-1,q-1}+1,a])
=\min(S_q\cap[d'_{p-1,q}+1,a]) 
=c'_{pq}$.
This completes the induction.
\end{proof}

Now we are ready to prove the main proposition in this section. 
\begin{proof}[Proof of Proposition \ref{prop:concurrent equiv}]
We shall prove the statements in the following order: (i)$\Rightarrow$ (v) and ``Moreover part",  (v) $\Rightarrow$ (iv) $\Rightarrow$ (ii) $\Rightarrow$(i)$\Rightarrow$(iii)$\Rightarrow$(ii).  

\noindent(i)$\Rightarrow$(v) and ``Moreover part": the idea of the proof is to use the geometry of a straight road map given in Lemma \ref{lem:straight 2 conditions}, where the rectangle has been divided into curly hexagon shapes $R_{pq}$ as illustrated in Figure \ref{fig:straight relabel}. The proof studies two cases depending on whether the point $(i,j')=\phi_\alpha^{-1}(i,j,k)$ lies on a horizontal path; in each case the values of $\bar{r}_{ijk} (=\bar{r}^\alpha_{ij'}), \bar{s}_{ijk} (=\bar{s}^\alpha_{ij'}), \dots$ are explicitly determined and the proof follows. 

Let $C$ be a concurrent vertex map.
By definition, there exists a straight road map $(\{H^\alpha_{p}\}, \{V^\beta_{q}\})$ satisfying \eqref{eq:Ck}.
Fix $\alpha\in {\rm V}_{\rm target}$. For convenience, denote
\begin{equation}\label{eq:convenient notation for H and V}
\begin{aligned}
&\textrm{$V^{\source(h_{r_1})}_1$, $V^{\source(h_{r_1})}_2$,$\dots$, $V^{\source(h_{r_1})}_{u_{\source(h_{r_i})}}$, $V^{\source(h_{r_2})}_1$, $V^{\source(h_{r_2})}_2$, $\dots$,$V^{\source(h_{r_2})}_{u_{\source(h_{r_2})}}$,
$V^{\source(h_{r_s})}_1$, $V^{\source(h_{r_s})}_2$, $\dots$,$V^{\source(h_{r_s})}_{u_{\source(h_{r_s})}}$}\\
&\textrm{as
$V^\alpha_1,V^\alpha_2,\dots,V^\alpha_{v_\alpha}$}
\end{aligned}
\end{equation}
 in that order, where  $r_1<\cdots<r_s$ are all indices such that $h_{r_1},\dots,h_{r_s}$ are all arrows with target $\alpha$.
 We also denote $V^\alpha_q$ to be the vertical line $y=q$ for $q\le0$. 
\begin{equation}\label{eq:P to Q}
\textrm{Denote $P_{pq}\rightsquigarrow Q_{pq}$ to be the subpath $H^\alpha_p\cap V^\alpha_q$.}
\end{equation} 
(See Figure \ref{fig:straight relabel} and compare the labeling therein with the original labeling in Figure \ref{fig:straight}.) Note that the region $[1,a_\alpha]_\mathbb{R}\times(-\infty,b_\alpha]_\mathbb{R}$ has been divided into regions $R_{pq}$ bouded by $H^\alpha_p$, $H^\alpha_{p+1}$, $V^\alpha_q$, $V^\alpha_{q+1}$ for $0\le p\le u_\alpha$ and $q\le v_\alpha$. So in general $R_{pq}$ has a ``curly hexagon'' shape with boundary $P_{pq}Q_{pq}P_{p,q+1}Q_{p+1,q+1}P_{p+1,q+1}Q_{p+1,q}$, and has degenerate shape in the obvious sense for $p=0$ or $p=u_\alpha$,  for $q\le 0$ or $q=v_\alpha$. 
\begin{figure}[h]
\begin{center}
\begin{tikzpicture}[scale=.8]

\draw  [gray] (0,0.5) rectangle (5.5,5.5);
\draw  [gray] (6,0.5) rectangle (11.5,5.5);
\draw  [gray] (12,0.5) rectangle (17.5,5.5);

\draw [blue,densely dotted,thick](1.5,0.5) node[scale=.1] (v6) {} -- (-1.6,0.5);
\draw [blue,densely dotted,thick](2,1) node[scale=.1]  (v4) {} -- (-1.6,1);
\draw [blue,densely dotted,thick](3,1.5) node[scale=.1]  (v2) {} -- (-1.6,1.5);
\draw [blue,densely dotted,thick](2,0.5) node[scale=.1]  (v5) {} -- (3.5,0.5) node[scale=.1]  (v12) {};
\draw [blue,densely dotted,thick](3,1) node[scale=.1]  (v3) {} -- (4.5,1) node[scale=.1]  (v10) {};
\draw [blue,densely dotted,thick](4,2.5) node[scale=.1]  (v1) {} -- (5,2.5) node[scale=.1]  (v8) {};
\draw [blue,densely dotted,thick](5,2) node[scale=.1]  (v9) {} -- (7,2) node[scale=.1]  (v16) {};
\draw [blue,densely dotted,thick](5.5,3.5) node[scale=.1]  (v7) {} -- (8,3.5) node[scale=.1]  (v14) {}  (4.5,0.5) node[scale=.1]  (v11) {} -- (6.5,0.5) node [scale=.1] (v18) {};
\draw [blue,densely dotted,thick](7,1.5) node[scale=.1]  (v17) {} -- (8.5,1.5) node[scale=.1]  (v24) {};
\draw [blue,densely dotted,thick](9,4.5) node[scale=.1]  (v13) {} -- (10.5,4.5) node [scale=.1] (v20) {};
\draw [blue,densely dotted,thick](11,5.5) node[scale=.1]  (v19) {} -- (13.5,5.5) node[scale=.1]  (v26) {};
\draw [blue,densely dotted,thick](14.5,5.5) node[scale=.1]  (v25) {} -- (15,5.5) node[scale=.1]  (v32) {};
\draw [blue,densely dotted,thick](17,5.5) node[scale=.1]  (v31) {} -- (17.5,5.5);
\draw [blue,densely dotted,thick](8,3) node [scale=.1] (v15) {} -- (9.5,3) node[scale=.1]  (v22) {}  (10.5,4) node[scale=.1]  (v21) {} -- (13,4) node[scale=.1]  (v28) {};
\draw [blue,densely dotted,thick](13.5,5) node [scale=.1] (v27) {} -- (14.5,5) node [scale=.1] (v34) {};
\draw [blue,densely dotted,thick](15,5) node[scale=.1]  (v33) {} -- (17.5,5);
\draw [blue,densely dotted,thick](9.5,2) node[scale=.1]  (v23) {} -- (12.5,2) node [scale=.1] (v30) {}  (13,3) node[scale=.1]  (v29) {} -- (13.5,3) node[scale=.1]  (v36) {};
\draw [blue,densely dotted,thick](14.5,4.5) node[scale=.1]  (v35) {} -- (17.5, 4.5);
\draw [red,densely dotted,thick](4, 6) -- (v1)  (v2) -- (v3)  (v4) -- (v5)  (v6) -- (1.5,0);
\draw [red,densely dotted,thick](5.5, 6) -- (v7)  (v8) -- (v9)  (v10) -- (v11)  (v12) -- (3.5,0);
\draw [red,densely dotted,thick](9, 6) -- (v13)  (v14) -- (v15)  (v16) -- (v17)  (v18) -- (6.5,0);
\draw [red,densely dotted,thick](11,6) -- (v19)  (v20) -- (v21)  (v22) -- (v23)  (v24) -- (8.5,0);
\draw [red,densely dotted,thick](14.5, 6) -- (v25)  (v26) -- (v27)  (v28) -- (v29)  (v30) -- (12.5,0);
\draw [red,densely dotted,thick](17, 6) -- (v31)  (v32) -- (v33)  (v34) -- (v35)  (v36) -- (13.5,0);
\draw [thick] (v6) -- (v5);
\draw [thick] (v4) -- (v3);
\draw [thick] (v12) -- (v11);

\draw [thick, decorate,decoration={zigzag,segment length=2mm, amplitude=0.4mm}] (v2) -- (v1);
\draw [thick, decorate,decoration={zigzag,segment length=2mm, amplitude=0.4mm}] (v10) -- (v9);
\draw [thick, decorate,decoration={zigzag,segment length=2mm, amplitude=0.4mm}] (v8) -- (v7);
\draw [thick, decorate,decoration={zigzag,segment length=2mm, amplitude=0.4mm}] (v14) -- (v13);
\draw [thick, decorate,decoration={zigzag,segment length=2mm, amplitude=0.4mm}](v20) -- (v19);
\draw [thick] (v26) -- (14.5,5.5);
\draw [thick] (v32) -- (v31);
\draw [thick, decorate,decoration={zigzag,segment length=2mm, amplitude=0.4mm}] (v16) -- (v15);
\draw [thick, decorate,decoration={zigzag,segment length=2mm, amplitude=0.4mm}] (v22) -- (v21);
\draw [thick, decorate,decoration={zigzag,segment length=2mm, amplitude=0.4mm}] (v28) -- (v27);
\draw [thick] (v34) -- (v33);
\draw [thick, decorate,decoration={zigzag,segment length=2mm, amplitude=0.4mm}] (v18) -- (v17);
\draw [thick, decorate,decoration={zigzag,segment length=2mm, amplitude=0.4mm}] (v24) -- (v23);
\draw [thick, decorate,decoration={zigzag,segment length=2mm, amplitude=0.4mm}] (v30) -- (v29);
\draw [thick, decorate,decoration={zigzag,segment length=2mm, amplitude=0.4mm}]  (v36) -- (v35);

\draw (v2) node [above left] {\tiny $P_{11}$};
\draw (v1) node [above left] {\tiny $Q_{11}$};

\draw (v7) node [above left] {\tiny $Q_{12}$};
\draw (v8) node [above left] {\tiny $P_{12}$};

\draw (2,0.9) node [above] {\tiny $P_{21}$};
\draw (3,1.1) node [below] {\tiny $Q_{21}$};

\draw (1.5,.4) node [above] {\tiny $P_{31}$};
\draw (v5) node [below] {\tiny $Q_{31}$};

\draw (v9) node [left] {\tiny $Q_{22}$};
\draw (v10) node [above left] {\tiny $P_{22}$};

\draw (v11) node [below] {\tiny $Q_{32}$};
\draw (3.5,.6) node [below] {\tiny $P_{32}$};

\draw (v13) node [above left] {\tiny $Q_{13}$};
\draw (v14) node [above left] {\tiny $P_{13}$};

\draw (v15) node [left] {\tiny $Q_{23}$};
\draw (v16) node [above left] {\tiny $P_{23}$};

\draw (v17) node [left] {\tiny $Q_{33}$};
\draw (v18) node [right] {\tiny $P_{33}$};

\draw (v19) node [left] {\tiny $Q_{14}$};
\draw (v20) node [right] {\tiny $P_{14}$};

\draw (v21) node [left] {\tiny $Q_{24}$};
\draw (v22) node [right] {\tiny $P_{24}$};

\draw (v23) node [above left] {\tiny $Q_{34}$};
\draw (v24) node [above left] {\tiny $P_{34}$};

\draw (v25) node [above left] {\tiny $Q_{15}$};
\draw (v26) node [above left] {\tiny $P_{15}$};

\draw (v27) node [left] {\tiny $Q_{25}$};
\draw (v28) node [right] {\tiny $P_{25}$};

\draw (v29) node [left] {\tiny $Q_{35}$};
\draw (v30) node [right] {\tiny $P_{35}$};

\draw (v31) node [below] {\tiny $Q_{16}$};
\draw (v32) node [above ] {\tiny $P_{16}$};

\draw (v33) node [below] {\tiny $Q_{26}$};
\draw (v34) node [below left] {\tiny $P_{26}$};

\draw (v35) node [below right] {\tiny $Q_{36}$};
\draw (v36) node [right] {\tiny $P_{36}$};

\draw (-1.5,1.5) node [left] {\tiny $H^\alpha_1$};
\draw (-1.5,1) node [left] {\tiny $H^\alpha_2$};
\draw (-1.5,0.5) node [left] {\tiny $H^\alpha_3$};

\draw (4,6) node [above] {\tiny $V^\alpha_1$};
\draw (5.5,6) node [above] {\tiny $V^\alpha_2$};

\draw (9,6) node [above] {\tiny $V^\alpha_3$};
\draw (11,6) node [above] {\tiny $V^\alpha_4$};

\draw (14.5,6) node [above] {\tiny $V^\alpha_5$};
\draw (17,6) node [above] {\tiny $V^\alpha_6$};

\draw (.7,4.5) node [above] {\color{red}\tiny $R_{00}$};
\draw (.7,1) node [above] {\color{red}\tiny $R_{10}$};
\draw (.7,0.5) node [above] {\color{red}\tiny $R_{20}$};

\draw (4.5,4.5) node [above] {\color{red}\tiny $R_{01}$};
\draw (4,1.5) node [above] {\color{red}\tiny $R_{11}$};
\draw (3.7,0.5) node [above] {\color{red}\tiny $R_{21}$};

\draw (7,4.5) node [above] {\color{red}\tiny $R_{02}$};
\draw (6.5,2.6) node [above] {\color{red}\tiny $R_{12}$};
\draw (5.5,1) node [above] {\color{red}\tiny $R_{22}$};

\draw (10,4.7) node [above] {\color{red}\tiny $R_{03}$};
\draw (9.2,3.6) node [above] {\color{red}\tiny $R_{13}$};
\draw (8.2,2) node [above] {\color{red}\tiny $R_{23}$};
\draw (7.5,0.8) node [above] {\color{red}\tiny $R_{33}$};

\draw (12,4.5) node [above] {\color{red}\tiny $R_{14}$};
\draw (11,2.5) node [above] {\color{red}\tiny $R_{24}$};
\draw (10,0.8) node [above] {\color{red}\tiny $R_{34}$};

\draw (14,5) node [above] {\color{red}\tiny $R_{15}$};
\draw (13.7,4) node [above] {\color{red}\tiny $R_{25}$};
\draw (13,0.8) node [above] {\color{red}\tiny $R_{35}$};

\draw (16,5) node [above] {\color{red}\tiny $R_{16}$};
\draw (16,4.5) node [above] {\color{red}\tiny $R_{26}$};
\draw (16,0.8) node [above] {\color{red}\tiny $R_{36}$};

%extra vertical lines
\draw [red,densely dotted,thick](-0.5,0)--(-0.5,6) (-1,0)--(-1,6) (-1.5,0)--(-1.5,6);
\draw (-.3,6) node [above] {\tiny $V^\alpha_0$};
\draw (-1,6) node [above] {\tiny $V^\alpha_{-1}$};
\draw (-1.7,6) node [above] {\tiny $V^\alpha_{-2}$};
\draw (0.3,3) node [above right] {\tiny $P_{10}(=Q_{10})$};\draw [-stealth](0.3,3) -- (-0.5,1.5);
\draw (0.3,2.5) node [above right] {\tiny $P_{20}(=Q_{20})$};\draw [-stealth](0.3,2.5) -- (-0.5,1);
\draw (0.3,2) node [above right] {\tiny $P_{30}(=Q_{30})$};\draw [-stealth](0.3,2) -- (-0.5,0.5);
\draw (-.2,5.7) node [red, right] {\tiny $R_{0,-1}$};\draw [-stealth](-.15,5.7) -- (-0.75,5.7);

\end{tikzpicture}       
\end{center}
\caption{An example of the notation in \eqref{eq:convenient notation for H and V}.}
\label{fig:straight relabel}
\end{figure}
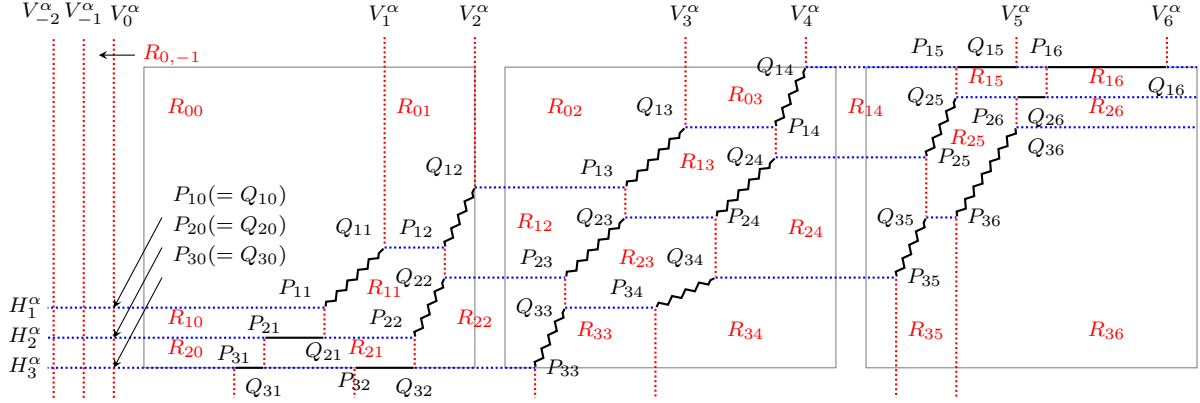
Any point $(i,j')\in [1,a_\alpha]_\mathbb{Z}\times[1,b_\alpha]_\mathbb{Z}$ must fall into one of the following two cases:

Case 1. $(i,j')$ lies strictly between $H^\alpha_p$ and $H^\alpha_{p+1}$ for some $p$. Assume it lies in the closed region $R_{pq}$ for some $q$. Recall the definition of $\bar{C}^\alpha$ in Definition \ref{def:long}.
It is clear (see Figure \ref{fig:straight relabel}) that $\bar{r}^\alpha_{ij'}= p$ 
%because we can find a diagonal chain in $\bar{C}^\alpha \cap \big( (-\infty,i-1]_\mathbb{Z}\times(-\infty,j'-1]_\mathbb{Z}\big)$ of size $p$ (by choosing a suitable point from each $P_{p-t,q-t}\rightsquigarrow Q_{p-t,q-t}$ for $t=1,\dots,p$), and there does not exist longer diagonal chains in $\bar{C}$, because any such a chain contains at most one point in each of $H^\alpha_1,\dots,H^\alpha_p$. Similarly, 
and $\bar{s}^\alpha_{ij'}=u_\alpha-p$. Thus $\bar{r}^\alpha_{ij'}+\bar{s}^\alpha_{ij'}=u_\alpha$.

Case 2. If $(i,j')$ lies on $H^\alpha_p$ for some $p$. Then $\bar{r}^\alpha_{ij'}= p-1$, $\bar{s}^\alpha_{ij'}=u_\alpha-p$, thus $\bar{r}^\alpha_{ij'}+\bar{s}^\alpha_{ij'}=u_\alpha-1$.

This proves that 
$\bar{r}_{ijk}+\bar{s}_{ijk}
=\bar{r}^\alpha_{ij'}+\bar{s}^\alpha_{ij'}
=u_{\target(h_k)}-1$ or $u_{\target(h_k)}$, and the former equality holds if and only if $(i,j')$ (determined by $\phi_\alpha(i,j')=(i,j,k)$) is on a horizontal path, that is, $(i,j)$ is in $\bigcup_p H^{\target(h_k)}_p|_k$. 

The condition for $V^{\source(h_k)}_q|_k$ can be proved similarly. Therefore (i) implies (v) and ``Moreover part". 

\smallskip

\noindent (v)$\Rightarrow$(iv): assume (v) holds. The idea of the proof is to use the fact that $\bar{r}_{ijk}\ge r_{ijk}$, and that the if $\bar{r}_{ijk}\neq r_{ijk}$, then we have an explicit expression of $\bar{r}_{ijk}$; so we can prove (iv) on a case-by-case basis. 

For the ``only if'' part: 
if $(i,j,k)\in C$, then  
 $r_{ijk}+s_{ijk}\le \bar{r}_{ijk}+\bar{s}_{ijk} =u_{\target(h_k)}-1<u_{\target(h_k)}$; similarly, $r'_{ijk}+s'_{ijk}<u_{\source(h_k)}$.
 
For the ``if'' part: we prove the contrapositive. Assuming that $(i,j,k)\notin C$, we need to show that it is impossible to have both $r_{ijk}+s_{ijk}<u_{\target(h_k)}$ and $r'_{ijk}+s'_{ijk}<u_{\source(h_k)}$. Suppose both inequalities hold. 
By (v), we can assume $\bar{r}_{ijk}+\bar{s}_{ijk}=u_\alpha$ (where $\alpha=\target(h_k)$) without loss of generality. Then either $r_{ijk}\neq \bar{r}_{ijk}$, or $s_{ijk}\neq \bar{s}_{ijk}$. Consider 3 cases:

Case 1. If $r_{ijk}\neq \bar{r}_{ijk}$ and $s_{ijk}=\bar{s}_{ijk}$. Then 
$\bar{r}_{ijk}=\min(i-1,u_{\alpha}-1-\min(a_\alpha-i,b_\alpha-j'))\le u_{\alpha}-1-\min(a_\alpha-i,b_\alpha-j')$. On the other hand, the distance from $(i,j')$ to the bottom (resp.~right)  side of the rectangle $[1,a_\alpha]_\mathbb{Z}\times[1,b_\alpha]_\mathbb{Z}$ is $a_\alpha-i$ (resp.~$b_\alpha-j'$), so $s_{ijk}\le\min(a_\alpha-i,b_\alpha-j')$. Then
$\bar{r}_{ijk}+\bar{s}_{ijk}=\bar{r}_{ijk}+s_{ijk}\le u_{\alpha}-1$, a contradiction.

Case 2. If $r_{ijk}=\bar{r}_{ijk}$ and $s_{ijk}\neq\bar{s}_{ijk}$. The proof is similar to Case 1.

Case 3. If $r_{ijk}\neq \bar{r}_{ijk}$ and $s_{ijk}\neq\bar{s}_{ijk}$. Then 
%$\bar{r}_{ijk}=\min(i-1,u_{\alpha}-1-\min(a_\alpha-i,b_\alpha-j'))$ and
%$\bar{s}_{ijk}=\min(a_\alpha-i,u_{\alpha}-\min(i,j'))$.
%So
\begin{equation}\label{eq:bar r+s}
\bar{r}_{ijk}+\bar{s}_{ijk}=\min(i-1,u_{\alpha}-1-\min(a_\alpha-i,b_\alpha-j'))+\min(a_\alpha-i,u_{\alpha}-\min(i,j'))
\end{equation}
We prove the contradiction that \eqref{eq:bar r+s} is $\neq u_\alpha$ by considering three cases separately:

\noindent If $i\le j'$, then $\bar{r}_{ijk}+\bar{s}_{ijk}\le (i-1)+(u_{\alpha}-\min(i,j'))=(i-1)+(u_{\alpha}-i)=u_\alpha-1$;

\noindent If $a_\alpha-i\le b_\alpha-j'$, then $\bar{r}_{ijk}+\bar{s}_{ijk}\le (u_\alpha-1-(a_\alpha-i))+(a_\alpha-i)=u_\alpha-1$;

\noindent If $i>j'$ and $a_\alpha-i> b_\alpha-j'$, then
$\bar{r}_{ijk}+\bar{s}_{ijk}\le (u_\alpha-1-(b_\alpha-j'))+(u_\alpha-j')=2u_\alpha-1-b_\alpha=u_\alpha-1-(b_\alpha-u_\alpha)\le u_\alpha-1$, where the last inequality uses the assumption \eqref{eq:condition on u}.

%This contradicts our assumption. So (v)$\Rightarrow$(iv).

\smallskip

\noindent (iv)$\Rightarrow$(ii): Assume (iv) holds. We shall first show that $C$ is $\uu$-compatible, then show that it is maximal $\uu$-compatible. 

To show that $C$ is $\uu$-compatible: assume for some $\gamma\in {\rm V}_\mathcal{Q}$ that $C^\gamma$ contains a diagonal chain $D$ of size $\ge u_\gamma+1$, let $(i',j')$ be the SE-most point in $D$ and denote $(i,j,k)=\phi_\gamma(i',j')$. Thus $(i,j,k)\in C$. On the other hand, $r_{ijk}+s_{ijk}\ge r_{ijk}\ge u_\gamma$, which implies $(i,j,k)\notin C$, a contradiction. 

To show that $C$ is maximal $\uu$-compatible: if $C$ is not maximal $\uu$-compatible, then there exists a $\uu$-compatible $C'$ satisfying $C\subsetneq C'\subseteq L$. Let $(i,j,k)\in C'\setminus C$. Note that $r_{ijk}$ (resp.~$s_{ijk},r'_{ijk},s'_{ijk}$) are the same for $C$ and $C'$. But then $r_{ijk}+s_{ijk}<u_{\target(h_k)}$ and $r'_{ijk}+s'_{ijk}<u_{\source(h_k)}$, contradicting (iv).

\smallskip

\noindent (ii)$\Rightarrow$(i): Assume $C$ is maximal $\uu$-compatible. The idea to show that $C$ is a concurrent vertex map is by showing that the corresponding road map is straight; to do that, we shall verify the conditions (a)--(c) in Lemma \ref{lem:straight 2 conditions}.

By Lemma \ref{lemma:u-compatible = existence of road map}, there exists a road map $(\{H^\alpha_{p}\}, \{V^\beta_{q}\})$ such that \eqref{eq:C in HV} holds; furthermore we can assume that the road map is obtained by SE-most construction.
By the maximality of $C$, \eqref{eq:C in HV} becomes an equality:
\begin{equation}\label{eq:C = HV}
C|_k = \Big( \bigcup_{p=1}^{u_{\target(h_k)}} H^{\target(h_k)}_p|_k\Big)\
\bigcap\
\Big( \bigcup_{q=1}^{u_{\source(h_k)}} V^{\source(h_k)}_q|_k\Big), \quad \textrm{ for all } k.
\end{equation}
 It suffices to prove that $(\{H^\alpha_{p}\}, \{V^\beta_{q}\})$ is straight; or equivalently, to prove the conditions (a)--(c) in Lemma \ref{lem:straight 2 conditions}.
 
To prove condition (a): assume it does not hold.  The idea is that one can construct a diagonal chain $Z_1,\dots,Z_p$ such that the $\phi_\alpha$-image of one of these points added to $C$ gives a $\uu$-compatible set $C'\supsetneq C$, contradicting the maximality of $C$. Below is the detail. 

 Similar to the proof of Lemma \ref{lem:straight 2 conditions} (a), assume $(p,q)$ is the minimal pair in the dictionary order such that $H^\alpha_{p}|_k \cap V^\beta_q|_k$ is not connected. Let $P\rightsquigarrow Q$ and $P'\rightsquigarrow Q'$ be the SW-most two connected subpaths of  $H^\alpha_{p}|_k \cap V^\beta_q|_k$. Without loss of generality we assume $P'\stackrel{V}{\rightsquigarrow} Q$ lies to the SE of $Q\stackrel{H}{\rightsquigarrow} P'$ (this is Case 1 in the Proof of Lemma \ref{lem:straight 2 conditions}). 
Observe that no interior lattice point of $Q\stackrel{H}{\rightsquigarrow} P'$ is in $C|_k$, because if $R$ is such a point, then $R$ lies in some $V_{q'}$ with $q'<q$, and thus $(p,q')<(p,q)$ and $H^\alpha_p|_k\cap V^\beta_{q'}|_k$ is not connected, contradicting the choice of $(p,q)$.

Now focus on $A_\alpha$. Since $P'$ is to the NE of $Q$, $Q\stackrel{H}{\rightsquigarrow} P'$ has a unique highest NW corner $Y$. In the previous paragraph we observe that $Y\notin \phi_\alpha^{-1}(C)$. Since the road map is obtained by SE-most construction, we must have $Y=Y_p=(p,b_\alpha-u_\alpha+p)$, whose distance to the right side of $A_\alpha$ is $u_\alpha-p$. So there are at most $u_\alpha-p$ vertical paths on the right side of $Y$. On the other hand, the vertical paths 
 $V^\alpha_q,V^\alpha_{q+1},\dots,V^\alpha_{v_\alpha}$ are on the right side of $Y$. So 
\begin{equation}\label{eq:sq1<=up}
v_\alpha-q+1\le u_\alpha-p
\end{equation}
It follows that $q-p-1\ge v_\alpha-u_\alpha\ge0$ by \eqref{eq:condition on u}. So $q\ge p+1\ge2$.
For each $1\le p'\le p$, let $Z_{p'}$ be the rightmost point in the intersection of the horizontal line $x=p'$ with $V^\alpha_{p'+q-p-1}$ (see Figure \ref{fig:Z}). 
Since $Z_1,\dots,Z_p$ are mutually in SE-NW position and all of them are above $H^\alpha_p$, there exists an integer $p_0$ such that $1\le p_0\le p$ and $Z_{p_0}$ is not in $\cup_{p'=1}^{p-1}H^\alpha_{p'}$, so is not in $\phi_\alpha^{-1}(C)$. %, or equivalently, not in any of the $p-1$ horizontal paths $H^\alpha_1,\dots,H^\alpha_{p-1}$ (here we use the obvious fact that $H^\alpha_p, H^\alpha_{p+1},\dots$ are below those $Z_{p'}$'s). Ineed, this follows immediately from the fact that any horizontal path contains at most one of $Z_{p'}$'s, because any two $Z_{p'}$'s are in SE-NW position. 
\begin{figure}[ht]
\begin{center}
\begin{tikzpicture}[scale=.6]
%\draw  (0,0) rectangle (10,6);
\draw (3,2) node (v1) {} -- (2.5,2) -- (2.5,1.5) -- (1.5,1.5)  ;
\draw (6.5,4) node (v2) {};
\draw [red,   dotted, thick] (v1) -- (4.5,2) -- (4.5,2.5) -- (6,2.5) -- (6,3.5) -- (6.5,3.5) -- (v2);
\draw [blue, dotted, thick] (v1) -- (3,3) -- (4,3) -- (4,3.5) -- (5,3.5) -- (5,4) -- (v2);
\draw (1.5,1.5) node [below] {\tiny $P$};
\draw (v1) node [below] {\tiny $Q$};
\draw (6.5,3.8) node [right] {\tiny $P'$}; \draw[fill]  (v2) circle(2pt);
\draw (7,5) node [above] {\tiny $Y=(p,b_\alpha-u_\alpha+p)$}; \draw[fill]  (5,4) circle(2pt);\draw [-stealth](6,5) -- (5.2,4.2);

\draw (6,2.5) node [red, below] {\tiny $V^\alpha_q$};
\draw (3.5,3) node [blue,above] {\tiny $H^\alpha_p$};

\draw (2.4,4) node [right] {\tiny $Z_p$}; \draw[fill]  (2.5,4) circle(2pt);
\draw [red, dotted, thick] (2.5,5.5) -- (2.5,4) -- (2,4) -- (2,3) -- (1,3) -- (1,2.5) -- (0.5,2.5)  ;
\draw (3.5,5.8) node [red,left] {\tiny $V^\alpha_{q-1}$};

\draw (0.4,4.5) node [right] {\tiny $Z_{p-1}$}; \draw[fill]  (0.5,4.5) circle(2pt);
\draw [red, dotted, thick] (0.5,5.5) -- (0.5,4.5) -- (0,4.5) -- (0,3.5) -- (-1,3.5) -- (-1,3)  ;
\draw (1.5,5.8) node [red,left] {\tiny $V^\alpha_{q-2}$};

\draw (-0.6,5) node [right] {\tiny $Z_{p-2}$}; \draw[fill]  (-0.5,5) circle(2pt);
\draw [red, dotted, thick] (-.5,5.5) -- (-.5,5) -- (-1,5) -- (-1,4) -- (-2,4) -- (-2,3)  ;
\draw (0,5.8) node [red,left] {\tiny $V^\alpha_{q-3}$};

\end{tikzpicture}       
\end{center}
\caption{Proof of (ii)$\Rightarrow$(i), condition (a)}
\label{fig:Z}
\end{figure}
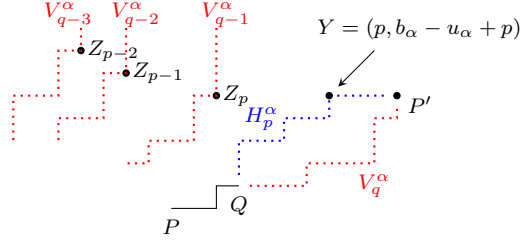
We claim that $C':=C\cup\{\phi_\alpha(Z_{p_0})\}$ is ${\bf u}$-compatible thus contradicting the maximality of $C$. Indeed, 
for $\beta\in {\rm V}_{\rm source}$: because ${C'}^\beta$ is contained in $\bigcup_{q=1}^{u_\beta} V^\beta_q$, it does not contain any diagonal chain of size $(u_\beta+1)$;
for $\alpha\in {\rm V}_{\rm target}$: it suffices to show that if $D$ is a diagonal chain in ${C'}^{\alpha}$ containing 
$\phi_\alpha(Z_{p_0})$, then the size of $D$ is $\le u_\alpha$. 
Since $\phi_\alpha(Z_{p_0})$ has $x$-coordinate $p_0$, there are at most $p_0-1$ points of $D$ that are NW of $\phi_\alpha(Z_{p_0})$;
meanwhile, since  $\phi_\alpha(Z_{p_0})$ is on $V^\alpha_{p_0+q-p-1}$, 
% any point of $D$ that is weakly SE of $\phi_\alpha(Z_{p_0})$ must be in $\bigcup_{q'=p_0+q-p-1}^{v_\alpha} V^\alpha_{q'}$, and no two points of $D$ lie in the same $V^\alpha_{q'}$, so 
there are at most $v_\alpha-(p_0+q-p-1)$ points of $D$ that are SE of $\phi_\alpha(Z_{p_0})$. Thus the size of $D$ is $\le (p_0-1)+1+(v_\alpha-(p_0+q-p-1))=v_\alpha-(q-p-1)\le u_\alpha$.
This proves (a). 

 To prove condition (b): we shall give an explicit construction of $H_p^\alpha$ using Lemma  \ref{lem:cd}.
More precisely, the straightening process is described as follows. Denote the $x$-coordinate of $P_{pq}$ by $d_{pq}$, the $x$-coordinate of $Q_{pq}$ by $c_{pq}$. 
 Let $c'_{pq}, d'_{pq}$ be obtained from the Straightening Lemma \ref{lem:cd}. Let $P'_{pq}$ be the leftmost point on $V^\alpha_q$ whose $x$-coordinate is $d'_{pq}$, let $Q'_{pq}$ be the rightmost point on $V^\alpha_q$ whose $x$-coordinate is $c'_{pq}$. Construct ${H'_p}^\alpha$ to be the path 
$$(a_\alpha-u_\alpha+p,1)\to P'_{p1}\stackrel{V}{\rightsquigarrow} Q'_{p1}\to P'_{p2}\stackrel{V}{\rightsquigarrow} Q'_{p2}\to\cdots\to P'_{p,v_\alpha}\stackrel{V}{\rightsquigarrow} Q'_{p,v_\alpha}\to (p,b)$$
where ``$\to$'' denote straight horizontal paths. Define $C'$ to be the right side of \eqref{eq:C in HV} with $H^\alpha_p$ replaced by ${H'_p}^\alpha$ for each $\alpha$. Then $C'$ is also ${\bf u}$-compatible and $C\subseteq C'$; moreover, if $C=C'$, then Lemma \ref{lem:cd} implies  $c_{pq}=c'_{pq}, d_{pq}=d'_{pq}$ for all $p,q$. In particular, $c_{pq}=d_{p,q+1}$ for all $p,q$ whenever the equality makes sense. Then $Q_{pq}$ and $P_{p,q+1}$ have the same $x$-coordinate. 
This proves (b). 

The proof of condition (c) is similar, so we skip.
\smallskip

(i)$\Rightarrow$(iii). The idea to prove \eqref{eq:cardinality of max C} is to give an explicit expression for the number of lattice points in each connected path $P^k_{pq}\rightsquigarrow Q^k_{pq}$, and add them up for all $(p,q)$. 
More precisely, given a concurrent vertex map $C$ induced from a straight road map $(\{H^\alpha_p\},\{V^\beta_q\})$, the number of lattice points in $P^k_{pq}\rightsquigarrow Q^k_{pq}$ is: 
%equal to 1 plus the number of edges therein:
\begin{equation}\label{eq:1+diff+diff}
1+\big| (P^k_{pq})_x - (Q^k_{pq})_x \big|
+\big| (P^k_{pq})_y - (Q^k_{pq})_y \big|
\end{equation}
It is easy to see that, the sums of the above three terms for all $(p,q)$ are 
$\sum_k u_{\source(h_k)}u_{\target(h_k)}$, 
$\sum_{\alpha\in {\rm V}_{\rm target}} u_\alpha(a_\alpha-u_\alpha)$,
$\sum_{\beta\in {\rm V}_{\rm source}} u_\beta(b_\beta-u_\beta)$, respectively. 
This proves \eqref{eq:cardinality of max C}. The fact that $C$ is $\uu$-compatible follows from Lemma \ref{lemma:u-compatible = existence of road map}.

(iii)$\Rightarrow$(ii).  Assume $C$ is ${\bf u}$-compatible and $|C|=N_{\mathcal{Q},\mm,\uu}$. Let $C'$ be maximal ${\bf u}$-compatible set that contains $C$. Since it is already proved that (ii)$\Rightarrow$(i)$\Rightarrow$(iii),  we must have $|C'|=N_{\mathcal{Q},\mm,\uu}=|C|$, which implies $C'=C$ and thus $C$ is maximal. 
 \end{proof}

\begin{proof}[Proof of Theorem \ref{thm:prime decomposition of I}]
The natural generators of ${\rm init }(I_{\mathcal{Q},\mm,\uu})$ are $(u_\gamma+1)$-diagonal chains in $A_\gamma$ for all $\gamma$. So $C\subseteq L$ is in $\Delta_{\mathcal{Q},\mm,\uu}$ exactly when $C$ does not contain any
$(u_\gamma+1)$-diagonal chain in $A_\gamma$ for any $\gamma$, or in other words, when $C$ is ${\bf u}$-compatible. Then $C$ is a facet of  $\Delta_{\mathcal{Q},\mm,\uu}$ if and only if $C$ is maximal ${\bf u}$-compatible. So the theorem follows from the statement (i)$\Leftrightarrow$(ii) of Proposition \ref{prop:concurrent equiv}. 
\end{proof}

\section{Properties of the simplicial complex $\Delta_{\mathcal{Q},\mm,\uu}$}
\subsection{$C_{\max}(S)$}
In this subsection we will define $C_{\max}(S)$ and study its properties. 
In (a) of the definition below, we use notation $H^\alpha_p, V^\alpha_q, P_{pq},Q_{pq}$ introduced in \eqref{eq:convenient notation for H and V} and \eqref{eq:P to Q}. 
In (b), we use the mirrored version of  \eqref{eq:convenient notation for H and V} and \eqref{eq:P to Q}, in particular,
$r'_1<\cdots<r'_t$ are all the indices such that $h_{r'_1},\dots,h_{r'_t}$ are all the arrows with source $\beta$.

\begin{definition}
Let $C$ be a concurrent vertex map.

(a) A point $(i,j,k)\in C$ is called a {\bf horizontal NW corner} if $(i,j,k)=\phi_\alpha(P)$ where $P$ is a NW corner of a path $H^\alpha_p$ for some $\alpha\in {\rm V}_{\rm target}$ and $1\le p\le u_\alpha$.
Such a corner is said to be {\bf non-essential}  if  the corresponding $P$ satisfies $P=Q_{p,p+v_\alpha-u_\alpha}$ for some $1\le p\le u_\alpha$; otherwise the corner is said to be {\bf essential}.

\smallskip

(b) A point $(i,j,k)\in C$ is called a {\bf vertical NW corner} if 
$(i,j,k)=\phi_\beta(P)$ where $P$ is a NW corner of a path $V^\beta_q$ for some $\beta\in {\rm V}_{\rm source}$ and $1\le q\le u_\beta$.
Such a corner is said to be {\bf non-essential}  if  the corresponding $P$ satisfies $P=P_{q+v_\beta-u_\beta,q}$ for some $1\le q\le u_\beta$; otherwise the corner is said to be {\bf essential}.
\end{definition}
Note the points $Q_{p,p+v_\alpha-u_\alpha}$ for $p=u_\alpha,u_{\alpha-1},\dots$ are just the rightmost $Q$ on the lowest horizontal path,  the 2nd rightmost $Q$ on the 2nd lowest horizontal path, $\dots$, respectively.  
Also note that if $(i,j,k)$ is both a horizontal and a vertical NW corner, then it is an essential horizontal NW corner and also an essential vertical NW corner. 
%Indeed, assume the statement is false, then without loss of generality we assume it is a non-essential horizontal NW corner, and denote $(i,j,k)=\phi_\alpha(P)=\phi_\beta(P')$ where $\alpha,\beta$ are the target and source of $h_k$, respectively. Then $P=Q_{p,p+v_\alpha-u_\alpha}$ which is on the vertical path $V^\alpha_{p+v_\alpha-u_\alpha}$ but is not a NW corner of that path. This contradicts our assumption. 

\begin{example}
In Figure \ref{fig:essential corner} we give an example to illustrate essential (denoted by $\bullet$) and non-essential horizontal NW corners (denoted by $\circ$). Here $v_\alpha=6$, $u_\alpha=3$, so by definition, a non-essential horizontal NW corner is a NW corner on $H^\alpha_p$ that is not $Q_{p,p+3}$ for $1\le p\le 3$.  In particular, $Q_{14}$ and $Q_{36}$ are non-essential horizontal NW corner (and that $Q_{25}$ is not a NW corner), and all other horizontal NW corners are essential. 
 
\begin{figure}[h]
\begin{center}
\begin{tikzpicture}[scale=.8]

\draw  [gray] (0,0.5) rectangle (5.5,5.5);
\draw  [gray] (6,0.5) rectangle (11.5,5.5);
\draw  [gray] (12,0.5) rectangle (17.5,5.5);

\draw [blue,densely dotted,thick](1.5,0.5) node[scale=.1] (v6) {} -- (-1.6,0.5);
\draw [blue,densely dotted,thick](2,1) node[scale=.1]  (v4) {} -- (-1.6,1);
\draw [blue,densely dotted,thick](3,1.5) node[scale=.1]  (v2) {} -- (-1.6,1.5);
\draw [blue,densely dotted,thick](2,0.5) node[scale=.1]  (v5) {} -- (3.5,0.5) node[scale=.1]  (v12) {};
\draw [blue,densely dotted,thick](3,1) node[scale=.1]  (v3) {} -- (4.5,1) node[scale=.1]  (v10) {};
\draw [blue,densely dotted,thick](4,2.5) node[scale=.1]  (v1) {} -- (5,2.5) node[scale=.1]  (v8) {};
\draw [blue,densely dotted,thick](5,2) node[scale=.1]  (v9) {} -- (7,2) node[scale=.1]  (v16) {};
\draw [blue,densely dotted,thick](5.5,3.5) node[scale=.1]  (v7) {} -- (8,3.5) node[scale=.1]  (v14) {}  (4.5,0.5) node[scale=.1]  (v11) {} -- (6.5,0.5) node [scale=.1] (v18) {};
\draw [blue,densely dotted,thick](7,1.5) node[scale=.1]  (v17) {} -- (8.5,1.5) node[scale=.1]  (v24) {};
\draw [blue,densely dotted,thick](9,4.5) node[scale=.1]  (v13) {} -- (10.5,4.5) node [scale=.1] (v20) {};
\draw [blue,densely dotted,thick](11,5.5) node[scale=.1]  (v19) {} -- (13.5,5.5) node[scale=.1]  (v26) {};
\draw [blue,densely dotted,thick](14.5,5.5) node[scale=.1]  (v25) {} -- (15,5.5) node[scale=.1]  (v32) {};
\draw [blue,densely dotted,thick](17,5.5) node[scale=.1]  (v31) {} -- (17.5,5.5);
\draw [blue,densely dotted,thick](8,3) node [scale=.1] (v15) {} -- (9.5,3) node[scale=.1]  (v22) {}  (10.5,4) node[scale=.1]  (v21) {} -- (13,4) node[scale=.1]  (v28) {};
\draw [blue,densely dotted,thick](13.5,5) node [scale=.1] (v27) {} -- (14.5,5) node [scale=.1] (v34) {};
\draw [blue,densely dotted,thick](15,5) node[scale=.1]  (v33) {} -- (17.5,5);
\draw [blue,densely dotted,thick](9.5,2) node[scale=.1]  (v23) {} -- (12.5,2) node [scale=.1] (v30) {}  (13,3) node[scale=.1]  (v29) {} -- (13.5,3) node[scale=.1]  (v36) {};
\draw [blue,densely dotted,thick](14.5,4.5) node[scale=.1]  (v35) {} -- (17.5, 4.5);
\draw [red,densely dotted,thick](4, 6) -- (v1)  (v2) -- (v3)  (v4) -- (v5)  (v6) -- (1.5,0);
\draw [red,densely dotted,thick](5.5, 6) -- (v7)  (v8) -- (v9)  (v10) -- (v11)  (v12) -- (3.5,0);
\draw [red,densely dotted,thick](9, 6) -- (v13)  (v14) -- (v15)  (v16) -- (v17)  (v18) -- (6.5,0);
\draw [red,densely dotted,thick](11,6) -- (v19)  (v20) -- (v21)  (v22) -- (v23)  (v24) -- (8.5,0);
\draw [red,densely dotted,thick](14.5, 6) -- (v25)  (v26) -- (v27)  (v28) -- (v29)  (v30) -- (12.5,0);
\draw [red,densely dotted,thick](17, 6) -- (v31)  (v32) -- (v33)  (v34) -- (v35)  (v36) -- (13.5,0);
\draw [thick] (v6) -- (v5);
\draw [thick] (v4) -- (v3);
\draw [thick] (v12) -- (v11);

\draw [thick] (v2) --(3.5,1.5)--(3.5,2.5)-- (v1);
\draw[fill=black] (3.5,2.5) circle (.6ex);

\draw [thick] (v10) --(4.5,2)-- (v9);
\draw[fill=black] (4.5,2) circle (.6ex);

\draw [thick] (v8) --(5,3.5)-- (v7);
\draw[fill=black] (5,3.5) circle (.6ex);

\draw [thick] (v14)--(8,4)--(9,4) -- (v13);
\draw[fill=black] (9,4.5) circle (.6ex);
\draw[fill=black] (8,4) circle (.6ex);

\draw[thick](v20)--(10.5,5)--(11,5)--(v19);
\draw[fill=black] (10.5,5) circle (.6ex);

\draw [thick] (v26) -- (14.5,5.5);

\draw [thick] (v32) -- (v31);

\draw [thick] (v16) --(8,2)-- (v15);
\draw[fill=black] (8,3) circle (.6ex);

\draw [thick] (v22) --(9.5,3.5)--(10.5,3.5)-- (v21);
\draw[fill=black] (10.5,4) circle (.6ex);
\draw[fill=black] (9.5,3.5) circle (.6ex);

\draw[thick] (v28)--(13,5)--(v27);
\draw[fill=black] (13,5) circle (.6ex);

\draw [thick] (v34) -- (v33);

\draw [thick] (v18) -- (7,0.5)-- (v17);
\draw[fill=black] (7,1.5) circle (.6ex);

\draw [thick] (v24) --(9,1.5)--(9,2)-- (v23);
\draw[fill=black] (9,2) circle (.6ex);

\draw [thick] (v30) -- (12.5,3)-- (v29);
\draw[fill=black] (12.5,3) circle (.6ex);

\draw[thick] (v35)--(14.5,4)--(14.5,3.5)--(14,3.5)--(14,3)--(v36); 
\draw[fill=black] (14,3.5) circle (.6ex);

\draw [fill=white](v19) circle (.6ex);
\draw (v19) node [above left] {\tiny $Q_{14}$};

\draw[black] (v27) circle (.2ex);
\draw (13.75,5.1) node [below] {\tiny $Q_{25}$};

\draw[fill=white] (v35) circle (.6ex);
\draw (v35) node [below right] {\tiny $Q_{36}$};

\draw (-1.5,1.5) node [left] {\tiny $H^\alpha_1$};
\draw (-1.5,1) node [left] {\tiny $H^\alpha_2$};
\draw (-1.5,0.5) node [left] {\tiny $H^\alpha_3$};

\draw (4,6) node [above] {\tiny $V^\alpha_1$};
\draw (5.5,6) node [above] {\tiny $V^\alpha_2$};

\draw (9,6) node [above] {\tiny $V^\alpha_3$};
\draw (11,6) node [above] {\tiny $V^\alpha_4$};

\draw (14.5,6) node [above] {\tiny $V^\alpha_5$};
\draw (17,6) node [above] {\tiny $V^\alpha_6$};

\end{tikzpicture}       
\end{center}
\caption{An example of essential horizontal NW corners.}
\label{fig:essential corner}
\end{figure}
\end{example}

\begin{lemma}\label{lem:no corner}
{\rm(a)} Assume $\alpha \in{\rm V}_{\rm target}$, $1\le p\le u_\alpha$. Then $(P_{pq'})^{A_\alpha}_x=p$ for $q'\ge p+1+v_\alpha-u_\alpha$,  
$(Q_{pq'})^{A_\alpha}_x=p$ for $q'\ge p+v_\alpha-u_\alpha$.
Consequently, for $q'\ge p+1+v_\alpha-u_\alpha$,
 the subpath $P_{pq'}{\rightsquigarrow}Q_{pq'}$ contains no horizontal NW corner, 
and $P_{pq'}+(1,0)= Q_{p+1,q'}$. 

{\rm(b)}  Assume $\beta \in{\rm V}_{\rm source}$, $1\le q\le u_\beta$. Then 
 $(Q_{p'q})^{A_\beta}_y=q$ for $p'\ge q+1+v_\beta-u_\beta$, 
 $(P_{p'q})^{A_\beta}_y=q$ for $p'\ge q+v_\beta-u_\beta$.
Consequently, for $p'\ge q+1+v_\beta-u_\beta$, 
the subpath $Q_{p'q}{\rightsquigarrow}P_{p'q}$ contains no vertical NW corner;
and $Q_{p'q}+(0,1)=P_{p',q+1}$.
\end{lemma}
\begin{proof} 
(a) We prove the first statement by backward induction on $p$. For the base case $p=u_\alpha$, the assertion 
$(Q_{u_\alpha,v_\alpha})^{A_\alpha}_x= u_\alpha$ follows from Lemma \ref{lem:straight 2 conditions} (b).  
For the inductive step, assume the statement is true for $p+1$. Then for $q'\ge p+1+v_\alpha-u_\alpha$,  $(Q_{p+1,q'})^{A_\alpha}_x=p+1$, thus $(P_{pq'})^{A_\alpha}_x\le (Q_{p+1,q'})^{A_\alpha}_x-1= p$. On the other hand, $(P_{pq'})^{A_\alpha}_x\ge p$ because it is on the path $H^\alpha_p$. So $(P_{pq'})^{A_\alpha}_x=p$ for $q'\ge p+1+v_\alpha-u_\alpha$. 
Next, $(Q_{pq'})^{A_\alpha}_x = (P_{p,q'+1})^{A_\alpha}_x=p$ for $q'+1\ge p+1+v_\alpha-u_\alpha$, that is, for $q'\ge p+v_\alpha-u_\alpha$.
Consequently for $q'\ge p+1+v_\alpha-u_\alpha$, $(P_{pq'})^{A_\alpha}_x = (Q_{p,q'-1})^{A_\alpha}_x=p$, thus the subpath of $H^\alpha_p$ from $Q_{p,q'-1}$ to the right endpoint $(p,v_\alpha)$ must be contained in the horizontal line $x=p$, therefore the subpath $P_{pq'}{\rightsquigarrow}Q_{pq'}$ contains no horizontal NW corner; meanwhile, since $(P_{pq'})^{A_\alpha}_x=p$ and $(Q_{p+1,q'})^{A_\alpha}_x=p+1$,  we have $P_{pq'}+(1,0)= Q_{p+1,q'}$. 

(b) is just the transpose of (a) thus can be proved similarly.
\end{proof}

\begin{definition}\label{df:order}
(i) Define a total ordering $<_T$ on $L$ as follows: $(i,j,k)<_T (i',j',k')$ if and only if
($k<k'$), or  ($k=k'$ and $i<i'$), or ($k=k'$ and $i=i'$ and $j<j'$). 

(ii) Given two concurrent vertex maps $C=(P_1,\dots,P_N)$ and $C'=(P'_1,\dots,P'_N)$, where $P_1<_T P_2<_T\cdots<_T P_N$ and  $P'_1<_T P'_2<_T\cdots<_T P'_N$, define $C<_TC'$ if there is an integer $1\le t\le N$ such that $P_t<_T P'_t$ and $P_s=P'_s$ for all $s>t$. 

(iii) For any $\uu$-compatible subset $S\subseteq L$, define $C_{\max}(S)$ as follows: 
first initialize $S':=S$, 
then for $P$ running from the largest vertex in $L$ to the smallest under the order $<_T$, redefine $S':=S'\cup\{P\}$ if $S'\cup \{P\}$ is $\uu$-compatible.
The final $S'$ is $C_{\max}(S)$.
Similarly, to define $C_{\min}(S)$, we first initialize $S':=S$, 
then for $P$ running from the smallest vertex in $L$ to the largest, redefine $S':=S'\cup\{P\}$ if $S'\cup \{P\}$ is $\uu$-compatible.
The final $S'$ is $C_{\min}(S)$.
 \end{definition}

Note that the order $<_T$ is ``opposite'' to the order defined in \eqref{eq: favorite order} in the following sense: 
$$\textrm{$x^{(k)}_{ij}>x^{(k')}_{i'j'}$ if and only if $(i,j,k)<_T (i',j',k')$.}$$
Also note that $C_{\max}(S)$ is maximal $\uu$-compatible, therefore is a concurrent vertex map. 
%Indeed, if $C_{\max}(S)\cup\{P\}$ is $\uu$-compatible, then $P$ should be added to $S'$ in the construction of $C_{\max}(S)$.

%In most part of the paper we only discuss properties of $C_{\max}(S)$. The properties of $C_{\min}(S)$ are similar so we omit.

\begin{lemma}\label{lem:larger C'}
Let $(i,j,k)\in C$ be an essential (horizontal or vertical) NW corner. 
Then there is a unique triple $(i',j',k')\notin C$, such that $(i,j,k)<_T (i',j',k')$, and $C'=(C\setminus\{(i,j,k)\})\cup\{(i',j',k')\}$ is also a concurrent vertex map and  $C <_T C'$. 
As a consequence, for a $\uu$-compatible subset $S\subseteq L$,  all essential NW corners of $C_{\max}(S)$ are contained in $S$.
\end{lemma}
\begin{proof}
Without loss of generality, assume $(i,j,k)$ is an essential horizontal NW corner. Let $\alpha=\target(h_k)$. We consider two cases.

Case 1: if $(i,j,k)$ is a vertical NW corner. Let $t$ be the smallest nonnegative integer such that $(i',j',k'):=(i+t+1,j+t+1,k)$ is either not a horizontal NW corner or not a vertical NW corner (see Fig \ref{fig:essentialAB} left). It is easy to see that $(i+t+1,j+t+1)$ is a valid point in ${\rm Page}_k$, and $(i',j',k')\notin C$.
%Note that $(i+t+1,j+t+1)$ is a valid point in ${\rm Page}_k$ because a point in the bottom row of ${\rm Page}_k$ cannot be a horizontal NW corner, and a point in the rightmost column of  ${\rm Page}_k$ cannot be a vertical NW corner.
%Here we used the fact that $(i+t+1,j+t+1)$ is a valid point in ${\rm Page}_k$ (that is, $i+t+1\le m_{\target(h_k)}$ and  $j+t+1\le m_{\source(h_k)}$) because a point in the bottom row of ${\rm Page}_k$ cannot be a horizontal NW corner, and a point in the rightmost column of  ${\rm Page}_k$ cannot be a vertical NW corner.
%We claim that $(i',j',k')\notin C$. Indeed, let $p,q$ be the integers determined by the condition that $\phi_\alpha^{-1}(i,j,k) \in H^\alpha_p\cap V^\alpha_q$. Note note that for each $0\le t'\le t$, $(i+t',j+t',k)$ is both a horizontal and a vertical NW corner; thus the three points $\phi_\alpha^{-1}(i+t'+1,j+t',k)$, $\phi_\alpha^{-1}(i+t',j+t',k)$, $\phi_\alpha^{-1}(i+t',j+t'+1,k)$ must lie in the same horizontal path $H^\alpha_{p+t'}$ and the same vertical path $V^\alpha_{q+t'}$. If $(i',j',k')\in C$, then $\phi_\alpha^{-1}(i',j',k')$ must lie in both $H^\alpha_{p+t'+1}$ and $V^\alpha_{q+t'+1}$, but then it must be a NW corner of both, contradicting the choice of $t$. 
Define $C'=(C\setminus\{(i,j,k)\})\cup\{(i',j',k')\}$, and modify each $H^\alpha_{p+t'+1}$ into $H'_{p+t'+1}$ (resp., modify each $V^\alpha_{q+t'+1}$ into $V'_{q+t'+1}$) by replacing its NW corner $\phi_\alpha^{-1}(i+t',j+t',k)$  by a SE corner $\phi_\alpha^{-1}(i+t'+1,j+t'+1,k)$, as indicated in Figure \ref{fig:essentialAB} right. This new road map corresponds to $C'$, thus $C'$ is a concurrent vertex maps. It is obvious that  $(i,j,k)<_T (i',j',k')$ and $C<_T C'$. 
%Since $C'$ is contained in the road map obtained from the above modification,  $C'$ must be $\uu$-compatible. Furthermore $C'$ is maximal $\uu$-compatible since $|C'|=|C|$. So $C'$ is a concurrent vertex map. 

\begin{figure}[ht]
\begin{center}
\begin{tikzpicture}[scale=.4]
      \node at (0,0){$\bullet$};
      \node at (1,0) {$\bullet$}; 
      \node at (0,1) {$\bullet$}; \node at (-.1,1.5)[anchor=east]{\tiny$\phi_\alpha^{-1}(i,j,k)$};
      \node at (1,1) {$\bullet$};
      \draw (0, 0) -- (0,1) (0,1)--(1,1);
      \draw [decorate, decoration={  zigzag}]  (0,0) -- (-2,-2);
      \draw [decorate, decoration={  zigzag}]  (3,3) -- (1,1);
      \node at (-2,-2)[anchor=east]{\tiny$H^\alpha_p\cap V^\alpha_q$};
    \begin{scope}[shift={(1,-1)}]
      \node at (0,0){$\bullet$};
      \node at (1,0) {$\ddots$};
      \node at (1,1) {$\bullet$};
      \draw (0, 0) -- (0,1) (0,1)--(1,1);
      \draw [decorate, decoration={  zigzag}]  (0,0) -- (-2,-2);
      \draw [decorate, decoration={  zigzag}]  (3,3) -- (1,1);
      \node at (-2,-2)[anchor=east]{\tiny$H^\alpha_{p+1}\cap V^\alpha_{q+1}$};
      \end{scope}
    \begin{scope}[shift={(3,-3)}]
      \node at (0,0){$\bullet$};
      \node at (1,0) {+}; \node at (1,0)[anchor=north west]{\tiny$\phi_\alpha^{-1}(i',j',k')$};
      \node at (0,1) {$\bullet$};
      \node at (1,1) {$\bullet$}; %\node at (-.1,1)[anchor = east]{\tiny$\phi_\alpha^{-1}(i+t,j+t,k)$};
      \draw (0, 0) -- (0,1) (0,1)--(1,1);
      \draw [decorate, decoration={  zigzag}]  (0,0) -- (-2,-2);
      \draw [decorate, decoration={  zigzag}]  (3,3) -- (1,1);
      \node at (-2,-2)[anchor=east]{\tiny$H^\alpha_{p+t}\cap V^\alpha_{q+t}$};
      \end{scope}
      \node at (9, -1) {$\Longrightarrow$};
     \begin{scope}[shift={(16,0)}]
      \node at (0,0){$\bullet$};
      \node at (1,0) {$\bullet$}; \node at (-.1,1.5)[anchor=east]{\tiny$\phi_\alpha^{-1}(i,j,k)$};
      \node at (0,1) {+};
      \node at (1,1) {$\bullet$};
      \draw (0,0) -- (1,0) (1,0)--(1,1);
      \draw [decorate, decoration={  zigzag}]  (0,0) -- (-2,-2);
      \draw [decorate, decoration={  zigzag}]  (3,3) -- (1,1);
      \node at (-2,-2)[anchor=east]{\tiny$H'_{p}\cap V'_{q}$};
    \begin{scope}[shift={(1,-1)}]
      \node at (0,0){$\bullet$};
      \node at (1,0) {$\bullet$};
      \node at (1,1) {$\bullet$};
      \draw (0,0) -- (1,0) (1,0)--(1,1);
      \draw [decorate, decoration={  zigzag}]  (0,0) -- (-2,-2);
      \draw [decorate, decoration={  zigzag}]  (3,3) -- (1,1);
      \node at (-2,-2)[anchor=east]{\tiny$H'_{p+1}\cap V'_{q+1}$};
      \end{scope}
    \begin{scope}[shift={(3,-3)}]
      \node at (0,0){$\bullet$};
      \node at (1,0) {$\bullet$};
      \node at (0,1) {$\ddots$}; \node at (1,0)[anchor=north west]{\tiny$\phi_\alpha^{-1}(i',j',k')$};
      \node at (1,1) {$\bullet$};
      \draw (0,0) -- (1,0) (1,0)--(1,1);
      \draw [decorate, decoration={  zigzag}]  (0,0) -- (-2,-2);
      \draw [decorate, decoration={  zigzag}]  (3,3) -- (1,1);
      \node at (-2,-2)[anchor=east]{\tiny$H'_{p+t}\cap V'_{q+t}$};
      \end{scope}
    \end{scope}  
 \end{tikzpicture}
 \caption{Lemma \ref{lem:larger C'} Case 1. Left: before modification; Right: after modification}\label{fig:essentialAB}
\end{center}
\end{figure}

Next, we show the uniqueness of $(i',j',k')$. Assume $(i'',j'',k'')$ is a point in $C''$ 
satisfying the conditions that  $(i,j,k)<_T (i'',j'',k'')$ and $C''=(C\setminus\{(i,j,k)\})\cup\{(i'',j'',k'')\}$ is a concurrent vertex map.
The first condition implies that $\phi_\alpha^{-1}(i'',j'',k'')$ does not lie weakly NW of $\phi_\alpha^{-1}(i,j,k)$. 
Let $\{H''_p\}_{p=1}^{u_\alpha}$ and $\{V''_q\}_{q=1}^{v_\alpha}$ be the horizontal and vertical paths in $A_\alpha$ corresponding to $C''$. 
By the ``Moreover'' part of Proposition \ref{prop:concurrent equiv}, since  $\phi_\alpha^{-1}(i+1,j,k)$ is on $H^\alpha_p$, we have $\bar{r}_{i+1,j,k}(C)=p-1$; since $\phi_\alpha^{-1}(i'',j'',k'')$ does not lie weakly NW of $\phi_\alpha^{-1}(i,j,k)$, $C$ and $C''$ agree at every point weakly NW of $\phi_\alpha^{-1}(i,j,k)$ except the point $\phi_\alpha^{-1}(i,j,k)$ itself, and 
$\bar{r}_{i+1,j,k}(C'')=\bar{r}_{i+1,j,k}(C)=p-1$. 
%Again by the ``Moreover'' part of Proposition \ref{prop:concurrent equiv}, 
Since  $(i+1,j,k)$ is in $C$, it is also in $C''$, thus $\phi_\alpha^{-1}(i+1,j,k)$ is on $H''_p$ because $\bar{r}_{i+1,j,k}(C'')=p-1$. 
Similarly, $\phi_\alpha^{-1}(i,j+1,k)$ is on $H''_p$. If $\phi_\alpha^{-1}(i,j,k)$ is also on $H''_p$, then $(i,j,k)$ must be a horizontal NW corner, and therefore must be in $C''$, a contradiction. 
So $\phi_\alpha^{-1}(i,j,k)$ is not on $H''_p$.
This implies $\phi_\alpha^{-1}(i+1,j+1,k)$ is on $H''_p$.
Using the same argument and by induction on $t'$, we see that $\phi_\alpha^{-1}(i+t'+1,j+t',k)$, $\phi_\alpha^{-1}(i+t'+1,j+t'+1,k)$, $\phi_\alpha^{-1}(i+t',j+t'+1,k)$ are on the same horizontal path $H''_{p+t'}$ for every $0\le t'\le t$. In particular, $\phi_\alpha^{-1}(i',j',k')$ on $H''_{p+t}$. Similarly, $\phi_\alpha^{-1}(i',j',k')$ is on $V''_{q+t}$. Thus $(i',j',k')$ is in $C''$, which implies $(i'',j'',k'')=(i',j',k')$.

Case 2: if $(i,j,k)$ is not a vertical NW corner. 
This is possible only when $\phi_\alpha^{-1}(i,j,k)=Q_{pq}$ for some $p,q$ and the vertical edge connecting the vertices $\phi_\alpha^{-1}(i,j,k)$ and $\phi_\alpha^{-1}(i+1,j,k)$ is in $H^\alpha_p$. 
By Lemma \ref{lem:no corner}, $q\le p+v_\alpha-u_\alpha$. If the equality holds, then $(i,j,k)$ is a non-essential horizontal NW corner, contradicting our assumption. So $q\le p+v_\alpha-u_\alpha-1$, thus for all $t\le u_\alpha-p$,  $q+t+1\le (p+v_\alpha-u_\alpha-1)+(u_\alpha-p)+1= v_\alpha$. 
Let $t\le u_\alpha-p$ be the smallest nonnegative integer such that $P_{p+t,q+t+1}+(1,0)\notin \phi_\alpha^{-1}(C)$. 
Then for $0\le t'<t$, we must have $P_{p+t',q+t'+1}+(1,0)=Q_{p+t'+1,q+t'+1}$. 
For $0\le t'\le t$, 
denote $R_{p+t',q+t'}$ to be the leftmost vertex on $P_{p+t'}\rightsquigarrow Q_{p+t'}$ with $x$-coordinate $p+t'$.  % (that is, of the same height as $Q_{p+t'}$).
The fact that $Q_{pq}$ is a horizontal NW corner forces $R_{p+t',q+t'}$ to also be a horizontal NW corner for each $0\le t'\le t$. 
In particular, $R_{p+t,q+t}$ is a horizontal NW corner, and therefore the $x$-coordinate of $R_{p+t,q+t}$ is $<a_\alpha$.  
\begin{figure}[ht]
\begin{center}
\begin{tikzpicture}[scale=.45]
      \node at (0,0){$\bullet$};
      \node at (0,1) {$\bullet$};      \node at (4.5,2)[anchor=south]{\tiny$\phi_\alpha^{-1}(i,j,k)=Q_{pq}=R_{pq}$};\draw [-stealth](1,2) -- (0.2,1.2);
      \node at (4,1) {$\bullet$};  \node at (4.2,1)[anchor=south]{\tiny$P_{p,q+1}$};
      \draw (0, 0) -- (0,1);
      \draw[blue, thick, dotted](0,1)--(4,1); 
      \draw[red, thick, dotted](0,1)--(0,2.5);
       \draw [decorate, decoration={  zigzag}]  (0,0) -- (-1,-1);
   \node at (-.5,-.7)[anchor=north]{\tiny$H^\alpha_p$};
    \begin{scope}[shift={(2,-1)}]
      \node at (0,0){$\bullet$};
      \node at (0,1){$\bullet$}; \node at (0.2,1.1)[anchor=north west]{\tiny$R_{p+1,q+1}$};\draw [-stealth](0.5,0.5) -- (0.2,0.8);
      \node at (1,1){$\bullet$};
      \node at (2,1){$\bullet$}; \node at (3.5,1.8)[anchor=south west]{\tiny$Q_{p+1,q+1}$};\draw [-stealth](3.6,2.1) -- (2.2,1.2);
      \node at (4,1) {$\bullet$};   \node at (3.9,.9)[anchor=south west]{\tiny$P_{p+1,q+2}$};
      \draw (0, 0) -- (0,1)--(2,1);
      \draw[blue, thick, dotted](2,1)--(4,1); 
      \draw[red, thick, dotted](2,1)--(2,2);
       \draw [decorate, decoration={  zigzag}]  (0,0) -- (-2,-2);
   \node at (-1.3,-1.7)[anchor=north]{\tiny$H^\alpha_{p+1}$};
      \end{scope}
    \begin{scope}[shift={(5,-2)}]
      \node at (0,0){$\bullet$};
      \node at (0,1){$\bullet$};\node at (0.2,1)[anchor=north west]{\tiny$R_{p+2,q+2}$};\draw [-stealth](0.5,0.5) -- (0.2,0.8);
      \node at (1,1){$\bullet$};\node at (1,.8)[anchor=south west]{\tiny$Q_{p+2,q+2}$};
      \node at (2.5,-.5){$\ddots$};
      \draw (0, 0) -- (0,1)--(1,1);
      \draw[blue, thick, dotted](1,1)--(2,1); 
      \draw[red, thick, dotted](1,1)--(1,2);
       \draw [decorate, decoration={  zigzag}]  (0,0) -- (-2,-2);
   \node at (-1.3,-1.7)[anchor=north]{\tiny$H^\alpha_{p+2}$};
    \end{scope}
    \begin{scope}[shift={(9,-4)}]
      \node at (0,0){$\bullet$};
      \node at (0,1){$\bullet$};\node at (0.2,1)[anchor=north west]{\tiny$R_{p+t,q+t}$};\draw [-stealth](0.5,0.5) -- (0.2,0.8);
      \node at (1,1){$\bullet$};\node at (.75,1)[anchor=south west]{\tiny$Q_{p+t,q+t}$};
      \node at (4,1) {$\bullet$};   \node at (3.5,1)[anchor=south west]{\tiny$P_{p+t,q+t+1}$};
      \node at (4,0){$+$}; \node at (3.5,0)[anchor=north west]{\tiny$P'=\phi_\alpha^{-1}(i',j',k')$};
      \draw (0, 0) -- (0,1)--(1,1);
      \draw[blue, thick, dotted](1,1)--(4,1); 
      \draw[red, thick, dotted](1,1)--(1,2);
      \draw[red, thick, dotted](4,0)--(4,1);
       \draw [decorate, decoration={  zigzag}]  (0,0) -- (-2,-2);
   \node at (-1.3,-1.7)[anchor=north]{\tiny$H^\alpha_{p+t}$};
      \end{scope}
\node at (15, -1) {$\Longrightarrow$};
\begin{scope}[shift={(18,0)}]
      \node at (0,0){$\bullet$};  \node at (0,1)[anchor=south west]{\tiny$Q_{pq}$};
      \node at (0,1) {$+$};     %\node at (0,1)[anchor=south east]{\tiny$(i,j,k)$};
      \node at (4,1) {$\bullet$};
       \draw[blue, thick, dotted] (0,0)--(2,0);       
      \draw[red, thick, dotted] (0,0)--(0,2.5);
       \draw [decorate, decoration={  zigzag}]  (0,0) -- (-1,-1);
   \node at (-.5,-.7)[anchor=north]{\tiny$H'_p$};
    \begin{scope}[shift={(2,-1)}]
      \node at (0,0){$\bullet$}; %\node at (0,0)[anchor=north west]{\tiny$Q'_{p+1,q+1}$};
      \node at (0,1){$\bullet$};% \node at (0,1)[anchor=north west]{\tiny$P'_{p,q+1}$};
      \node at (1,1){$\bullet$};
      \node at (2,1){$\bullet$}; 
      \node at (4,1) {$\bullet$};
      \draw  (0,1)--(2,1)  (2,1)--(2,2);
       \draw[blue, thick, dotted](0,0)--(3,0); 
      \draw[red, thick, dotted](0,0)--(0,1);
       \draw [decorate, decoration={  zigzag}]  (0,0) -- (-2,-2);
   \node at (-1.3,-1.7)[anchor=north]{\tiny$H'_{p+1}$};
      \end{scope}
    \begin{scope}[shift={(5,-2)}]
      \node at (0,0){$\bullet$}; %\node at (0,0)[anchor=north west]{\tiny$Q'_{p+2,q+2}$};
      \node at (0,1){$\bullet$}; %\node at (0,1)[anchor=north west]{\tiny$P'_{p+1,q+2}$};
      \node at (1,1){$\bullet$};
      \node at (2.5,-.5){$\ddots$};
      \draw (0,1)--(1,1)--(1,2);
       \draw[blue, thick, dotted](0,0)--(2,0); 
      \draw[red, thick, dotted](0, 0) -- (0,1);
       \draw [decorate, decoration={  zigzag}]  (0,0) -- (-2,-2);
   \node at (-1.3,-1.7)[anchor=north]{\tiny$H'_{p+2}$};
      \end{scope}
    \begin{scope}[shift={(9,-4)}]
      \node at (0,0){$\bullet$}; %\node at (0,0)[anchor=north west]{\tiny$Q'_{p+t,q+t}$};
      \node at (0,1){$\bullet$}; %\node at (0,1)[anchor=north west]{\tiny$P'_{p+t-1,q+t}$};
      \node at (1,1){$\bullet$};
      \node at (4,1) {$\bullet$}; 
      \node at (4,0){$\bullet$}; \node at (3.5,0)[anchor=north west]{\tiny$P'$};
      \draw  (0,1)--(1,1) (4,0)--(4,1);
      \draw[blue, thick, dotted](4,0)--(0,0); 
      \draw[red, thick, dotted](0,0)--(0,1);
      \draw[red, thick, dotted](1,1)--(1,2);
       \draw [decorate, decoration={  zigzag}]  (0,0) -- (-2,-2);
     \node at (-1.3,-1.7)[anchor=north]{\tiny$H'_{p+t}$};
      \end{scope}
\end{scope}
\end{tikzpicture}
 \caption{Lemma \ref{lem:larger C'} Case 2. Left: before modification; right: after modification.}\label{fig:essentialA}
\end{center}
\end{figure}
Now, let $P'=P_{p+t,q+t+1}+(1,0)$, $(i',j',k')=\phi_\alpha(P')$, $C'=(C\setminus\{(i,j,k)\})\cup\{(i',j',k')\}$, keep all vertical paths unchanged, and modify each horizontal path $H^\alpha_{p+t'+1}$ into $H'_{p+t'+1}$ by replacing its subpath 
$R_{p+t',q+t'}+(1,0)\to R_{p+t',q+t'}\to P_{p+t',q+t'+1}$  
by a subpath
$R_{p+t',q+t'}+(1,0)\to P_{p+t',q+t'+1}+(1,0)\to P_{p+t',q+t'+1}$, as shown in Figure \ref{fig:essentialA}. It is easy to see that $C'$ is a concurrent vertex map corresponding to the modified road map, and $C<_T C'$.
%Since $C'$ is contained in the road map obtained after above modification, we see that $C'$ is $\uu$-compatible. Furthermore $C'$ is maximal $\uu$-compatible since $|C'|=|C|$. So $C'$ is a concurrent vertex map. It is obvious that  $(i,j,k)<_T (i',j',k')$ and $C<_T C'$. 

Next, we show the uniqueness of $(i',j',k')$. Assume $(i'',j'',k'')$ is a point in $C''$ satisfying the conditions $(i,j,k)<_T (i'',j'',k'')$ and $C''=(C\setminus\{(i,j,k)\})\cup\{(i'',j'',k'')\}$ is a concurrent vertex map. 
Then $P'':=\phi_\alpha^{-1}(i'',j'',k'')$ does not lie weakly NW to $Q_{pq}=\phi_\alpha^{-1}(i,j,k)$. 
Let $\{H''_p\}_{p=1}^{u_\alpha}$ and $\{V''_q\}_{q=1}^{v_\alpha}$ be the horizontal and vertical paths in $A_\alpha$ corresponding to $C''$. 
We now prove the following claim by induction.

Claim: for $0\le t'\le t$, $H''_{p+t'}$ contains the subpath $R_{p+t',q+t'}+(1,0)\to P_{p+t',q+t'+1}+(1,0)\to P_{p+t',q+t'+1}.$

%For $t'=0$, we already noted that the vertex $R_{pq}+(1,0)$ is in $H''_{p}$. The path $H''_{p}$ goes east from $R_{pq}+(1,0)$ until it reaches a SE corner $(x_0,y_0)$. 
%Since the value of $\bar{r}_{\phi_\alpha(P_{p,q+1})}$ does not decrease when replacing $C$ by $C''$, $H''_p$ must pass through a point weakly west of $P_{p,q+1}$; in other words, $y_0 \le (P_{p,q+1})_y$, the $y$-coordinate of $P_{p,q+1}$. We consider three cases: 
For $t'=0$, let $x_0:=(R_{pq})_0$. Then $H''_p$ must pass $R_{pq}+(1,0)$, go east to a SE corner $(x_0,y_0)$, turn north to $(x_0-1,y_0)$ with $(R_{pq})_y<y_0\le (P_{p,q+1})_y$. 
If $y_0= (P_{p,q+1})_y$ then we are done. 
If $(R_{pq})_y< y_0 < (R_{p+1,q+1})_y $, then $(x_0,y_0)$ is a SE corner and $(x_0-1,y_0)$ is a NW corner of $H''_p$, so both are in $\phi_\alpha^{-1}(C''\setminus C$), a contradiction.
So we can assume $(R_{p+1,q+1})_y \le y_0 < (P_{p,q+1})_y$. Then $(x_0,y_0)\in \phi_\alpha^{-1}(C\cap C'')$. 
Since $(x_0,y_0)$ is in $H''_p$ and $H_{p+1}$, we have $\bar{s}_{\phi_\alpha(x_0,y_0)}(C'')
=
\bar{s}_{\phi_\alpha(x_0,y_0)}(C)
+1
$, thus $P''$ lies SE of $(x_0,y_0)$. But then $\bar{r}_{\phi_\alpha(P_{p,q+1})}(C'')=\bar{r}_{\phi_\alpha(P_{p,q+1})}(C)=p-1$, thus $H''_p$ contains $P_{p,q+1}$, the vertex $(x_0-1,y_0)$ is a NW corner of $H''_p$. So both $P''$ and $(x_0-1,y_0)$ are in $\phi_\alpha^{-1}(C''\setminus C)$, a contradiction. 

%Otherwise consider two cases: 
%(i) if $(R_{p+1,q+1})_y \le y_0 < (P_{p,q+1})_y$: then $(x_0,y_0)\in \phi_\alpha^{-1}(C\cap C'')$. 
%Since $(x_0,y_0)$ is in $H''_p$ and $H_{p+1}$, we have $\bar{s}_{\phi_\alpha(x_0,y_0)}(C'')
%=
%\bar{s}_{\phi_\alpha(x_0,y_0)}(C)
%+1
%$, thus $P''$ lies SE of $(x_0,y_0)$. But then $\bar{r}_{\phi_\alpha(P_{p,q+1})}(C'')=\bar{r}_{\phi_\alpha(P_{p,q+1})}(C)=p-1$, thus $H''_p$ contains $P_{p,q+1}$, the vertex $(x_0-1,y_0)$ is a NW corner of $H''_p$, which is in $C''\setminus C=\{P''\}$. This implies $P''=(x_0-1,y_0)$, contradicting the fact that $P''$ lies SE of $(x_0,y_0)$.  %See Figure \ref{fig:t'} Left.

%(ii) if $(R_{pq})_y< y_0 < (R_{p+1,q+1})_y $: then $\phi_\alpha(x_0,y_0)\in C''\setminus C$, so is $P''=(x_0,y_0)$. Meanwhile, $H''_p$ must  pass through $P_{p,q+1}$, so must have a NW corner $(x_0-1,y_0)$. But then $(x_0-1,y_0)\neq P''$ is also in $C''\setminus C$, a contradiction. 

For $t'>0$, the proof is similar: let $x_0=(R_{p+t',q+t'})_x$.  Then $H''_{p+t'}$ must pass $R_{p+t',q+t'}+(1,0)$, go east to a SE corner $(x_0,y_0)$, turn north to $(x_0-1,y_0)$ with $(R_{p+t',q+t'})_y<y_0\le (P_{p+t',q+t'+1})_y$. 
If $y_0= (P_{p+t',q+t'+1})_y$ then we are done. 
If $(R_{p+t',q+t'})_y< y_0 < (R_{p+t'+1,q+t'+1})_y $, then $(x_0,y_0)$ is a SE corner and $(x_0-1,y_0)$ is a NW corner of $H''_p$, so the former is in $\phi_\alpha^{-1}(C''\setminus C)$ and the latter is in $\phi_\alpha^{-1}(C)$ thus is in the horizontal segment $R_{p+t',q+t'}\to Q_{p+t',q+t'}$, which is a subpath of $H''_{p+t'-1}$ by inductive hypothesis. So $(x_0-1,y_0)$ is in both $H''_{p+t'}$ and $H''_{p+t'-1}$, a contradiction. 
So we can assume $(R_{p+t'+1,q+t'+1})_y \le y_0 < (P_{p+t',q+t'+1})_y$. Then $(x_0,y_0)\in \phi_\alpha^{-1}(C\cap C'')$. 
Since $(x_0,y_0)$ is in $H''_p$ and $H_{p+1}$, we have $\bar{s}_{\phi_\alpha(x_0,y_0)}(C'')
=
\bar{s}_{\phi_\alpha(x_0,y_0)}(C)
+1
$, thus $P''$ lies SE of $(x_0,y_0)$. But then $\bar{r}_{\phi_\alpha(P_{p,q+1})}(C'')=\bar{r}_{\phi_\alpha(P_{p,q+1})}(C)=p-1$, thus $H''_p$ contains $P_{p,q+1}$, the vertex $(x_0-1,y_0)$ is a NW corner of $H''_p$. So $(x_0-1,y_0)$ lies in $\phi_\alpha^{-1}(C)$, thus is contained in the horizontal segment $R_{p+t',q+t'}Q_{p+t',q+t'}$, which is a subpath of $H''_{p+t'-1}$ by the inductive hypothesis. So $(x_0-1,y_0)$ has to lie in both $H''_{p+t'}$ and $H''_{p+t'-1}$, a contradiction. 

Now the claim is proved. In particular, when $t'=t$, the claim says that $H''_{p+t}$ contains $P_{p+t,q+t+1}+(1,0)=P'$. So $(i',j',k')$ is in a horizontal path induced from $C''$. Similarly, it is in a vertical path induced from $C''$. So it is in $C''$. This implies the uniqueness. 
\end{proof}

\begin{proposition}\label{prop:C(S)>=C}
Let $C$ be a concurrent vertex map and $S$ be a subset of $C$. Then $C_{\max}(S)\ge_T C$, and the equality holds if and only $S$ contains all the essential horizontal and vertical NW corners of $C$. 
Similarly, $C_{\min}(S)\le_T C$, and the equality holds if and only $S$ contains all the essential horizontal and vertical SE corners of $C$. 
\end{proposition}
\begin{proof}
By symmetry, we only prove the first statement. Note that
$C_{\max}(S)\ge_T C$ is obvious, so we only need to prove the condition for the equality to hold. 
%``$\Rightarrow$'': we prove the contrapositive. Without lost of generality assume that $S$ does not contain an essential horizontal NW corner $(i,j,k)$ of $C$, so $S\subseteq C\setminus\{(i,j,k)\}$. By Lemma \ref{lem:larger C'}, there exists $C'=(C\setminus\{(i,j,k)\})\cup\{(i',j',k')\}>_TC$, so $C'$ contains $S$ and thus $C_{\max}(S)\ge_T C'>_T C$. 
``$\Rightarrow$'' follows immediately from Lemma \ref{lem:larger C'}. For ``$\Leftarrow$'', assume that $S$ contains all the essential horizontal and vertical NW corners of $C$, and assume the contrary that $C':=C_{\max}(S)>_T C$. 
Let
$$\textrm{$(i,j,k)$ be the largest point in $C'\setminus C$ under the order ``$<_T$''.}$$
Note that a point $(i',j',k')$ satisfying $(i',j',k')>_T (i,j,k)$ lies in $C'$ if and only if it lies in $C$. 
Then the choice of $(i,j,k)$ guarantees 
\begin{equation}\label{eq:s=s}
\bar{s}_{ijk}(C')=\bar{s}_{ijk}(C).
\end{equation}

%Denote $p_0=\bar{r}_{ijk}(C)$, $q_0=\bar{s}_{ijk}(C)$. 
Let $\alpha={\rm t}(h_k)$, $\beta={\rm s}(h_k)$, $p_0=\bar{r}_{ijk}(C)$. By Proposition \ref{prop:concurrent equiv},
in $A_\alpha$,  $\phi_\alpha^{-1}(i,j,k)$ is strictly below $H_{p_0}^\alpha$ (if $p_0>0$) and weakly above $H_{p_0+1}^\alpha$ (if $p_0<u_\alpha-1$). Let $q_0$ be the integer in the range $0\le q_0\le v_\alpha-1$ such that $\phi_\alpha^{-1}(i,j,k)$ is strictly to the right of $V_{q_0}^\alpha$ (if $q_0>0$) and weakly to the left of $V_{q_0+1}^\alpha$ (if $q_0<v_\alpha-1$). 
We shall construct s sequence of points 
\begin{equation}\label{eq:sequence S}
S^\alpha_{p_0},T^\alpha_{p_0},S^\alpha_{p_0-1},T^\alpha_{p_0-1},\dots,S^\alpha_{1}, T^\alpha_1
\textrm{ in }A_\alpha
\end{equation}
 recursively as follows.
Start from $\phi_\alpha^{-1}(i,j,k)$, walk west (for at least 1 unit) until reaching path $H^\alpha_{p_0}$ at point $S^\alpha_{p_0}$, then walk north along $H^\alpha_{p_0}$ (for at least one unit) until reaching a NW corner $T^\alpha_{p_0}$; then for $p=p_0, p_0-1,\dots,2$, once $T^\alpha_p$ is constructed, walk west from it (for at least 1 unit) until reaching path $H^\alpha_{p-1}$ at point $S^\alpha_{p-1}$,  then walk north along $H^\alpha_{p-1}$ (for at least 1 unit) until reaching a NW corner which we denote $T^\alpha_{p-1}$. (See Figure \ref{fig:sequence of points}).
\begin{figure}[h]
\begin{center}
\begin{tikzpicture}[scale=.8]

\draw  [gray] (0,0.5) rectangle (5.5,5.5);
\draw  [gray] (6,0.5) rectangle (11.5,5.5);
\draw  [gray] (12,0.5) rectangle (17.5,5.5);

\draw [blue,densely dotted,thick](1.5,0.5) node[scale=.1] (v6) {} -- (0,0.5);
\draw [blue,densely dotted,thick](2,1) node[scale=.1]  (v4) {} -- (0,1);
\draw [blue,densely dotted,thick](3,1.5) node[scale=.1]  (v2) {} -- (0,1.5);
\draw [blue,densely dotted,thick](2,0.5) node[scale=.1]  (v5) {} -- (3.5,0.5) node[scale=.1]  (v12) {};
\draw [blue,densely dotted,thick](3,1) node[scale=.1]  (v3) {} -- (4.5,1) node[scale=.1]  (v10) {};
\draw [blue,densely dotted,thick](4,2.5) node[scale=.1]  (v1) {} -- (5,2.5) node[scale=.1]  (v8) {};
\draw [blue,densely dotted,thick](5,2) node[scale=.1]  (v9) {} -- (7,2) node[scale=.1]  (v16) {};
\draw [blue,densely dotted,thick](5.5,3.5) node[scale=.1]  (v7) {} -- (8,3.5) node[scale=.1]  (v14) {}  (4.5,0.5) node[scale=.1]  (v11) {} -- (6.5,0.5) node [scale=.1] (v18) {};
\draw [blue,densely dotted,thick](7,1.5) node[scale=.1]  (v17) {} -- (8.5,1.5) node[scale=.1]  (v24) {};
\draw [blue,densely dotted,thick](9,4.5) node[scale=.1]  (v13) {} -- (10.5,4.5) node [scale=.1] (v20) {};
\draw [blue,densely dotted,thick](11,5.5) node[scale=.1]  (v19) {} -- (13.5,5.5) node[scale=.1]  (v26) {};
\draw [blue,densely dotted,thick](14.5,5.5) node[scale=.1]  (v25) {} -- (15,5.5) node[scale=.1]  (v32) {};
\draw [blue,densely dotted,thick](17,5.5) node[scale=.1]  (v31) {} -- (17.5,5.5);
\draw [blue,densely dotted,thick](8,3) node [scale=.1] (v15) {} -- (9.5,3) node[scale=.1]  (v22) {}  (10.5,4) node[scale=.1]  (v21) {} -- (13,4) node[scale=.1]  (v28) {};
\draw [blue,densely dotted,thick](13.5,5) node [scale=.1] (v27) {} -- (14.5,5) node [scale=.1] (v34) {};
\draw [blue,densely dotted,thick](15,5) node[scale=.1]  (v33) {} -- (17.5,5);
\draw [blue,densely dotted,thick](9.5,2) node[scale=.1]  (v23) {} -- (12.5,2) node [scale=.1] (v30) {}  (13,3) node[scale=.1]  (v29) {} -- (13.5,3) node[scale=.1]  (v36) {};
\draw [blue,densely dotted,thick](14.5,4.5) node[scale=.1]  (v35) {} -- (17.5, 4.5);
\draw [red,densely dotted,thick](4, 6) -- (v1)  (v2) -- (v3)  (v4) -- (v5)  (v6) -- (1.5,0);
\draw [red,densely dotted,thick](5.5, 6) -- (v7)  (v8) -- (v9)  (v10) -- (v11)  (v12) -- (3.5,0);
\draw [red,densely dotted,thick](9, 6) -- (v13)  (v14) -- (v15)  (v16) -- (v17)  (v18) -- (6.5,0);
\draw [red,densely dotted,thick](11,6) -- (v19)  (v20) -- (v21)  (v22) -- (v23)  (v24) -- (8.5,0);
\draw [red,densely dotted,thick](14.5, 6) -- (v25)  (v26) -- (v27)  (v28) -- (v29)  (v30) -- (12.5,0);
\draw [red,densely dotted,thick](17, 6) -- (v31)  (v32) -- (v33)  (v34) -- (v35)  (v36) -- (13.5,0);
\draw [thick] (v6) -- (v5);
\draw [thick] (v4) -- (v3);
\draw [thick] (v12) -- (v11);

\draw [thick, decorate,decoration={zigzag,segment length=2mm, amplitude=0.4mm}] (v2) -- (v1);
\draw [thick, decorate,decoration={zigzag,segment length=2mm, amplitude=0.4mm}] (v10) -- (v9);
%\draw [thick, decorate,decoration={zigzag,segment length=2mm, amplitude=0.4mm}] (v8) -- (v7);
\draw [thick] (v8) -- (5.1,2.5)--(5.1,2.7)--(5.3,2.7)--(5.3, 3.2)--(5.5,3.2)-- (v7);
\draw [thick, decorate,decoration={zigzag,segment length=2mm, amplitude=0.4mm}] (v14) -- (v13);
\draw [thick, decorate,decoration={zigzag,segment length=2mm, amplitude=0.4mm}](v20) -- (v19);
\draw [thick] (v26) -- (14.5,5.5);
\draw [thick] (v32) -- (v31);
%\draw [thick, decorate,decoration={zigzag,segment length=2mm, amplitude=0.4mm}] (v16) -- (v15);
\draw [thick] (v16) -- (7,2.2)--(7.3,2.2)--(7.3, 2.7)--(7.7,2.7)--(7.7,2.8)--(8,2.8)-- (v15);
\draw [thick, decorate,decoration={zigzag,segment length=2mm, amplitude=0.4mm}] (v22) -- (v21);
\draw [thick, decorate,decoration={zigzag,segment length=2mm, amplitude=0.4mm}] (v28) -- (v27);
\draw [thick] (v34) -- (v33);
\draw [thick, decorate,decoration={zigzag,segment length=2mm, amplitude=0.4mm}] (v18) -- (v17);
\draw [thick, decorate,decoration={zigzag,segment length=2mm, amplitude=0.4mm}] (v24) -- (v23);
\draw [thick, decorate,decoration={zigzag,segment length=2mm, amplitude=0.4mm}] (v30) -- (v29);
\draw [thick, decorate,decoration={zigzag,segment length=2mm, amplitude=0.4mm}]  (v36) -- (v35);

\draw (v2) node [above left] {\color{gray}\tiny $P_{11}$};
\draw (v1) node [above left] {\color{gray}\tiny $Q_{11}$};

\draw (v7) node [above] {\color{gray}\tiny $Q_{12}$};
\draw (v8) node [above left] {\color{gray}\tiny $P_{12}$};

\draw (2,0.9) node [above] {\color{gray}\tiny $P_{21}$};
\draw (3,1.1) node [below] {\color{gray}\tiny $Q_{21}$};

\draw (1.5,.4) node [above] {\color{gray}\tiny $P_{31}$};
\draw (v5) node [below] {\color{gray}\tiny $Q_{31}$};

\draw (v9) node [left] {\color{gray}\tiny $Q_{22}$};
\draw (v10) node [above left] {\color{gray}\tiny $P_{22}$};

\draw (v11) node [below] {\color{gray}\tiny $Q_{32}$};
\draw (3.5,.6) node [below] {\color{gray}\tiny $P_{32}$};

\draw (v13) node [above left] {\color{gray}\tiny $Q_{13}$};
\draw (v14) node [above left] {\color{gray}\tiny $P_{13}$};

\draw (v15) node [right] {\color{gray}\tiny $Q_{23}$};
\draw (v16) node [above left] {\color{gray}\tiny $P_{23}$};

\draw (v17) node [left] {\color{gray}\tiny $Q_{33}$};
\draw (v18) node [right] {\color{gray}\tiny $P_{33}$};

\draw (v19) node [left] {\color{gray}\tiny $Q_{14}$};
\draw (v20) node [right] {\color{gray}\tiny $P_{14}$};

\draw (v21) node [left] {\color{gray}\tiny $Q_{24}$};
\draw (v22) node [right] {\color{gray}\tiny $P_{24}$};

\draw (v23) node [below right] {\color{gray}\tiny $Q_{34}$};
\draw (v24) node [below right] {\color{gray}\tiny $P_{34}$};

\draw (v25) node [above left] {\color{gray}\tiny $Q_{15}$};
\draw (v26) node [above left] {\color{gray}\tiny $P_{15}$};

\draw (v27) node [left] {\color{gray}\tiny $Q_{25}$};
\draw (v28) node [right] {\color{gray}\tiny $P_{25}$};

\draw (v29) node [left] {\color{gray}\tiny $Q_{35}$};
\draw (v30) node [right] {\color{gray}\tiny $P_{35}$};

\draw (v31) node [below] {\color{gray}\tiny $Q_{16}$};
\draw (v32) node [above ] {\color{gray}\tiny $P_{16}$};

\draw (v33) node [below] {\color{gray}\tiny $Q_{26}$};
\draw (v34) node [below left] {\color{gray}\tiny $P_{26}$};

\draw (v35) node [below right] {\color{gray}\tiny $Q_{36}$};
\draw (v36) node [right] {\color{gray}\tiny $P_{36}$};

\draw (0,1.5) node [left] {\tiny $H^\alpha_1$};
\draw (0,1) node [left] {\tiny $H^\alpha_2$};
\draw (0,0.5) node [left] {\tiny $H^\alpha_3$};

\draw (4,6) node [above] {\tiny $V^\alpha_1$};
\draw (5.5,6) node [above] {\tiny $V^\alpha_2$};

\draw (9,6) node [above] {\tiny $V^\alpha_3$};
\draw (11,6) node [above] {\tiny $V^\alpha_4$};

\draw (14.5,6) node [above] {\tiny $V^\alpha_5$};
\draw (17,6) node [above] {\tiny $V^\alpha_6$};

%draw the sequence of points
\fill[red] (11,2.2) circle(1.5pt) (7.3,2.2) circle(1.5pt)  (7.3,2.7) circle(1.5pt) (5.3,2.7) circle (1.5pt) (5.3,3.2) circle (1.5pt);
\node at (11,2.2)[anchor=south]{\color{red}\small$\phi_\alpha^{-1}(i,j,k)$};
\draw[red, thick, dotted](11,2.2)--(7.3,2.2) (7.3,2.7)--(5.3,2.7); 
\draw (7.1,2.3) node [below right] {\color{red}\small $S^\alpha_{2}$};
\draw (7.5,3) node [left] {\color{red}\small $T^\alpha_{2}$};
\draw (5.1,2.4) node [right] {\color{red}\small $S^\alpha_{1}$};
\draw (5.4,2.9) node [above left] {\color{red}\small $T^\alpha_{1}$};

%draw the sequence of points
\fill[blue] (10.5,5.5) circle(1.5pt) (10.5,4.5) circle(1.5pt) (13,4) circle(1.5pt)  (13.5,3) circle(1.5pt) (17.5,3) circle(1.5pt);
\draw[blue, very thick](10.5,5.5)--(10.5,4.5) -- (13,4)--(13.5,3)--(17.5,3); 
\draw (10.5,5.5) node [above] {\color{blue}\tiny $(1,(P_{14})_y)$};
\draw (17.5,3) node [above left] {\color{blue}\tiny $((P_{36)})_x,b_\alpha)$};
\draw (17.5,3) node [below left] {\color{blue}\small ${\rm Bound}_1$};

\end{tikzpicture}       
\end{center}
\caption{An example of the sequence in \eqref{eq:sequence S} with $p_0=2$. The blue broken line is ${\rm Bound}_1$.}
\label{fig:sequence of points}
\end{figure}
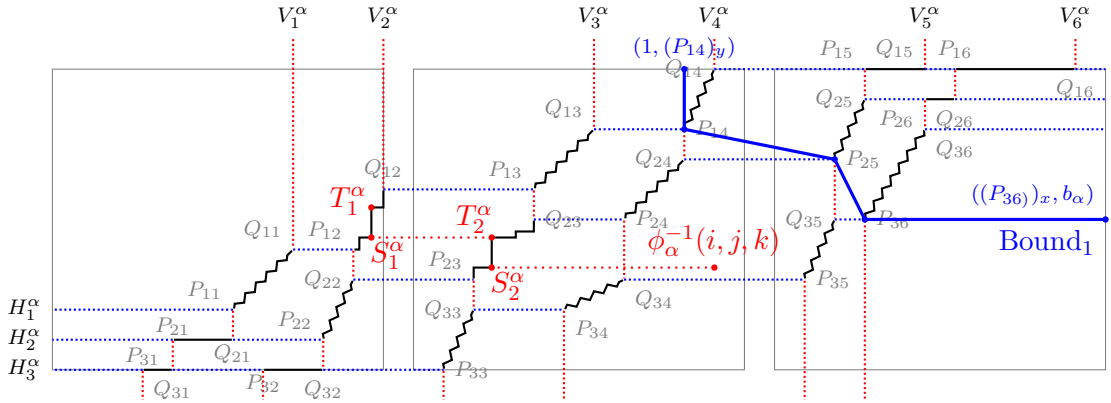
If some point in the sequence \eqref{eq:sequence S} can be constructed, we call it constructable; otherwise we call it unconstructable. 
%Obviously, if $S^\alpha_p$ is constructable, so is $T^\alpha_p$; if a point is unconstructable, so are all the points after it.  In the following claim, we will show that either all points are constructable, or none of them are.

\noindent{\bf Claim.} (i) ``Some point in \eqref{eq:sequence S} is unconstructable'' $\Leftrightarrow$  (ii) ``$S^\alpha_{p_0}$ is unconstructable'' $\Leftrightarrow$ (iii) ``$i=a_\alpha-u_\alpha+p_0+1$''.

\noindent Proof of Claim. For (ii)$\Rightarrow$(iii), note that  $S^\alpha_{p_0}$ is unconstructable if walking west from $\phi_\alpha^{-1}(i,j,k)$ will never reach $H^\alpha_{p_0}$. 
In this case, the $x$-coordinate of $\phi_\alpha^{-1}(i,j,k)$ is greater than the $x$-coordinate of the left endpoint of $H^\alpha_{p_0}$, that is, $i>a_\alpha-u_\alpha+p_0$. 
Since $\phi_\alpha^{-1}(i,j,k)$ is weakly above $H_{p_0+1}^\alpha$, 
the $x$-coordinate of $\phi_\alpha^{-1}(i,j,k)$ is no larger than the $x$-coordinate of the left endpoint of $H^\alpha_{p_0+1}$, that is, $i\le a_\alpha-u_\alpha+p_0+1$.  
So we must have $i= a_\alpha-u_\alpha+p_0+1$. 
It is obvious that (iii)$\Rightarrow$(ii)$\Rightarrow$(i). 
For (i)$\Rightarrow$(ii), consider the case  $S^\alpha_{p}$ is constructable but $S^\alpha_{p-1}$ is unconstructable for some $2\le p\le p_0$. 
Then the $x$-coordinate of $S_p^\alpha$ is $> a_\alpha-u_\alpha+p$, 
the $x$-coordinate of $S_{p+1}^\alpha$ is $> a_\alpha-u_\alpha+(p+1)$,
$\dots$, 
the $x$-coordinate of $S_{p_0}^\alpha$ is $i> a_\alpha-u_\alpha+p_0$. But then $S_{p_0}^\alpha$ is unconstructable, a contradiction to the assumption that $S^\alpha_{p}$ is constructable. So this case will never happen. This completes the proof of the Claim.

Assume the sequence of points in \eqref{eq:sequence S} are constructable. Then $T^\alpha_{p_0},\dots,T^\alpha_1$ are horizontal NW corners of $C$ and form a diagonal chain in $A_\alpha$.   
It is easy to see that:
\begin{equation}\label{eq:TonHV}
	\textrm{$S^\alpha_p$ and $T^\alpha_p$ are on %the subpath 
	$H^\alpha_p\cap V^\alpha_{e_p}$ for $1\le p\le p_0$, where $1\le e_1<e_2<\cdots<e_{p_0}\le q_0$.} 
\end{equation}

Similarly, let $q_1=\bar{r}'_{ijk}(C)$. The vertex $\phi_\beta^{-1}(i,j,k)$ is strictly to the right of $V_{q_1}^\beta$, weakly to the left of $V_{q_1+1}^\beta$, strictly below $V_{p_1}^\beta$, weakly above  $V_{p_1+1}^\beta$, for some $p_1$.
We construct points $S^\beta_{q_1}, T^\beta_{q_1},S^\beta_{q_1-1}, T^\beta_{q_1-1},\dots,S^\beta_{1},T^\beta_{1}$ in $A_\beta$ similarly to \eqref{eq:sequence S}  by starting from $\phi_\beta^{-1}(i,j,k)$ and walking north and west iteratively.
Like the above Claim, we have 
 ``some of those points is unconstructable'' $\Leftrightarrow$  ``$S^\beta_{q_1}$ is unconstructable'' $\Leftrightarrow$ $j=b_\beta-u_\beta+q_1+1$.
Assume those points are constructable. Then $T^\beta_{q_1},\dots,T^\beta_1$ are vertical NW corners of $C$ and form a diagonal chain in $A_\beta$, and 
\begin{equation}\label{eq:TonHV2}
\textrm{$S^\beta_q$ and $T^\beta_q$ are on %the subpath 
$H^\beta_{f_q}\cap V^\beta_{q}$ for $1\le q\le q_1$, where $1\le f_1<f_2<\cdots<f_{q_1}\le p_1$.} 
\end{equation}

We make the following easy observation whose proof is skipped. (See Figure \ref{fig:sequence of points} for an example of ${\rm Bound}_1$.) 
\smallskip

\noindent {\bf Observation} (a): $T^\alpha_{p_0},\dots,T^\alpha_1$ are essential horizontal NW corners of $C$ if $\phi_\alpha^{-1}(i,j,k)$ lies on the SW side of the following broken line in $A_\alpha$: 
$${\rm Bound}_1 := (1,(P_{1,1+v_\alpha-u_\alpha})_y)\to P_{1,1+v_\alpha-u_\alpha}\to P_{2,2+v_\alpha-u_\alpha}\to \cdots \to P_{u_\alpha v_\alpha}\to ((P_{u_\alpha v_\alpha})_x,b_\alpha).$$ 
\noindent {\bf Observation} (b): $T^\beta_{q_1},\dots,T^\beta_1$ are all essential vertical NW corners of $C$ if $\phi_\beta^{-1}(i,j,k)$ is to the NE of the following broken line in $A_\beta$:
$${\rm Bound}_2=((Q'_{1+v_\beta-u_\beta})_x,1)\to Q'_{1+v_\beta-u_\beta,1}\to Q'_{2+v_\beta-u_\beta,2}\to \cdots \to Q'_{v_\beta u_\beta}\to (a_\beta,(Q'_{v_\beta u_\beta})_y).$$

Also observe that 
$$\textrm{``$T^\alpha_{p_0},\dots, T^\alpha_{1}$ are essential or unconstructable'' $\Rightarrow$ $\bar{r}_{ijk}(C')\ge p_0$.}$$ 
Indeed, 
if they are essential, then by the assumption that $S$ contains all the essential NW corners of $C$, $\phi_\alpha^{-1}(S)$ must contain $T^\alpha_1,\dots, T^\alpha_{p_0}$, thus $\bar{r}_{ijk}(C')\ge p_0$;
if they are unconstructable, then $i=a_\alpha-u_\alpha+p_0+1$, and by definition,
$\bar{r}_{ijk}(C')=\max(r_{ijk},\min(i-1,u_{\alpha}-1-\min(a_\alpha-i,b_\alpha-j')))\ge
\min(i-1,u_{\alpha}-1-\min(a_\alpha-i,b_\alpha-j'))
\ge \min(i-1,u_{\alpha}-1-(a_\alpha-i))
= \min(a_\alpha-u_\alpha+p_0,p_0)=p_0$.
Similarly we have
$$\textrm{``$T^\beta_1,\dots, T^\beta_{q_1}$ are essential or unconstructable'' $\Rightarrow$ $\bar{r}'_{ijk}(C')\ge q_1$.}$$
To finish the proof of Proposition \ref{prop:C(S)>=C}, we discuss the following 4 cases separately. 

\smallskip

\noindent Case 1: every point of $T^\alpha_1,\dots, T^\alpha_{p_0}$ and $T^\beta_1,\dots,T^\beta_{q_1}$ is either essential or unconstructable. 
As we have seen,  $\bar{r}_{ijk}(C')\ge p_0=\bar{r}_{ijk}(C)$.
Together with \eqref{eq:s=s},  
we have $\bar{r}_{ijk}(C)+\bar{s}_{ijk}(C)\le \bar{r}_{ijk}(C')+\bar{s}_{ijk}(C')=u_\alpha-1$. 
Thus the equality $\bar{r}_{ijk}(C)+\bar{s}_{ijk}(C)=u_\alpha-1$ holds by Proposition \ref{prop:concurrent equiv}, and $\bar{r}_{ijk}(C')=p_0$. 
%and $\bar{r}_{ijk}(C') = p_0$. 
Similarly, $\bar{r}'_{ijk}(C)+\bar{s}'_{ijk}(C)= u_\beta-1$ and $\bar{r}'_{ijk}(C')=q_1$. Thus $(i,j,k)\in C$ by Proposition \ref{prop:concurrent equiv}, contradicting the assumption $(i,j,k)\notin C$. 

\smallskip

\noindent Case 2: $T^\beta_1,\dots,T^\beta_{q_1}$ are essential or unconstructable, but for some $1\le a\le p_0$, $T^\alpha_a$ is non-essential (so $T^\alpha_a=Q_{a,a+v_\alpha-u_\alpha}$). Then $\bar{r}'_{ijk}(C)+\bar{s}'_{ijk}(C)= u_\beta-1$ and $\bar{r}'_{ijk}(C')=q_1$ as we have shown in Case 1. If $\bar{r}_{ijk}(C)+\bar{s}_{ijk}(C)=u_\alpha-1$ then $(i,j,k)\in C$, contradicting our assumption. So $\bar{r}_{ijk}(C)+\bar{s}_{ijk}(C)=u_\alpha$. 
By Proposition \ref{prop:concurrent equiv}, 
 $(i,j) \in V^{\beta}_{q_1+1}|_k$, but $(i,j)\notin H^\alpha_p|_k$ for any $p$. So there are $p,q$ such that $\phi_\alpha^{-1}(i,j,k)$ is on a path $V^\alpha_q$, and strictly between the paths $H^\alpha_p$ and $H^\alpha_{p+1}$; therefore $\phi_\alpha^{-1}(i,j,k)$ lies in the interior of a vertical line segment $P_{p,q}Q_{p+1,q}$. 
By Observation (a), $\phi^{-1}_\alpha(i,j,k)$ is weakly on the NE side of ${\rm Bound}_1$. So the whole segment $P_{p,q}Q_{p+1,q}$ is weakly on the NE side of  ${\rm Bound}_1$. By the definition of  ${\rm Bound}_1$, we have $q\ge p+1+v_\alpha-u_\alpha$. By Lemma \ref{lem:no corner}, the segment $P_{p,q}Q_{p+1,q}$ has length 1, which contradicts the fact that $\phi_\alpha^{-1}(i,j,k)$ lies in the interior of $P_{p,q}Q_{p+1,q}$.

\smallskip

\noindent{Case 3}: $T^\alpha_1,\dots,T^\alpha_{p_0}$ are essential or unconstructable, but for some $1\le b\le q_1$, $T^\beta_b$ is non-essential. The proof is similar to Case 2.

\smallskip

\noindent{Case 4}: $T^\alpha_a=Q_{a,a+v_\alpha-u_\alpha}$ is non-essential for some $a$ such that $1\le a\le p_0$, and $T^\beta_b=P_{b+v_\beta-u_\beta,b}$ is non-essential for some $b$ such that $1\le b\le q_1$. Without loss of generality, assume $a$ and $b$ are the maximal choices. 
Define $D_1$ to be the number of horizontal paths in $A_\beta$ in those bocks above ${\rm Page}_k$, 
$D_2$ to be the number of vertical paths in $A_\alpha$ in those blocks to the left of ${\rm Page}_k$, that is, 
$$D_1=p_1-p_0=\sum_{r'_i<k} u_{\target(h_{r'_i})}, \quad D_2=q_0-q_1=\sum_{r_i<k} u_{\source(h_{r_i})}.$$
We also have
$H^\alpha_p|_k=H^\beta_{p+D_1}|_k$ for $1\le p\le u_\alpha$,
and
$V^\beta_q|_k=H^\alpha_{q+D_2}|_k$ for $1\le q\le u_\beta$.
Since $\phi_\alpha^{-1}(i,j,k)$ is on the weakly NE side of ${\rm Bound}_1$, $\phi_\beta^{-1}(i,j,k)$ is on the weakly SE side of ${\rm Bound}_2$, we have
\begin{equation}\label{eq:q0>=}
q_0\ge p_0+v_\alpha-u_\alpha, \quad p_1\ge q_1+v_\beta-u_\beta.
\end{equation}
so
$\sum_{r'_i<k} u_{\target(h_{r'_i})}+\sum_{r_i<k} u_{\source(h_{r_i})} = D_1+D_2=(p_1-p_0)+(q_0-q_1)\ge v_\alpha-u_\alpha+v_\beta-u_\beta.$ 
Since $\alpha={\target(h_{k})}$, $\beta={\source(h_{k})}$, $v_\alpha=\sum u_{\source(h_{r_i})}$,  $v_\beta=\sum u_{\target(h_{r'_i})}$, we have
$\sum_{r'_i\le k} u_{\target(h_{r'_i})}+\sum_{r_i\le k} u_{\source(h_{r_i})}\ge \sum u_{\source(h_{r_i})}+\sum u_{\target(h_{r'_i})},$ 
$\sum_{r'_i> k} u_{\target(h_{r'_i})}+\sum_{r_i> k} u_{\source(h_{r_i})}\le 0.$
Since $u_\gamma>0$ for any $\gamma$, there should be no summand at all in the above two sums. This implies: 
$r_s=r'_t=k$ (that is, ${\rm Page}_k$ is the rightmost block of $A_\alpha$, and is the bottommost block of $A_\beta$), 
and
% the inequalities in \eqref{eq:q0>=}   become equalities, that is,
$q_0 = p_0+v_\alpha-u_\alpha, \quad p_1 = q_1+v_\beta-u_\beta$, 
$D_1=p_1-p_0 = v_\beta-u_\alpha$, $D_2=q_0-q_1=v_\alpha-u_\beta$. 
Then $p_0+v_\alpha-u_\alpha=q_0=q_1+v_\alpha-u_\beta$, thus 
\begin{equation}\label{eq:p0-q1}
p_0-q_1=u_\alpha-u_\beta.
\end{equation}

%Next, we claim that either $T^\alpha_a$ is on page $k$ of $A_\alpha$, or $T^\beta_b$ is on page $k$ of $A_\beta$.

Without loss of generality, assume $p_0-a\le q_1-b$. Let $b'=a+q_1-p_0$, then $b\le b'\le q_1$ and $p_0-a=q_1-b'$. 
We claim that the following holds for $d=0,\dots,p_0-a$:

(i) the subpath $P_{p_0-d,q_0-d}\rightsquigarrow Q_{p_0-d,q_0-d}$ is in the last block of $A_\alpha$, and

(ii) both $T^\alpha_{p_0-d}$ and $\phi_\alpha^{-1}\phi_\beta(T^\beta_{q_1-d})$ are on  $P_{p_0-d,q_0-d}\rightsquigarrow Q_{p_0-d,q_0-d}$, and 

(iii) the point $T^\alpha_{p_0-d}$ is weakly SW of the point $\phi_\alpha^{-1}\phi_\beta(T^\beta_{q_1-d})$. 

To show (i), note that 
$q_0-d=(p_0+v_\alpha-u_\alpha)-d\ge v_\alpha-u_\alpha+a
= v_\alpha-u_\alpha+(p_0-q_1+b')\stackrel{\eqref{eq:p0-q1}}{=}
 v_\alpha-u_\alpha+(u_\alpha-u_\beta+b')=v_\alpha-u_\beta+b'>v_\alpha-u_\beta$. %(because $b'\ge b\ge1$). 
 Then (i) must hold because $A_\alpha$ has $v_\alpha$ vertical paths, $u_\beta$ of which are in the last block.
 Therefore  the subpath $P_{p_0-d,q_0-d}\rightsquigarrow Q_{p_0-d,q_0-d}$, which is contained in the $(q_0-d)$-th vertical path $V^\alpha_{q_0-d}$, is in the last block. 

To show (ii): since $T^\alpha_a=Q_{a,a+v_\alpha-u_\alpha}$, we have $e_a=a+v_\alpha-u_\alpha$, where $e_a$ is  defined in \eqref{eq:TonHV}. Since $a+v_\alpha-u_\alpha=e_a<e_{a+1}<\cdots<e_{p_0}\le q_0=p_0+v_\alpha-u_\alpha$,  for $a\le p\le p_0$ we must have $e_p=p+v_\alpha-u_\alpha$, and $T^\alpha_{p}$ must lie on 
$P_{p,p+v_\alpha-u_\alpha}\rightsquigarrow Q_{p,q+v_\alpha-u_\alpha}$,
equivalently, 
$T^\alpha_{p_0-d}$ lies on  $P_{p_0-d,q_0-d}\rightsquigarrow Q_{p_0-d,q_0-d}$ for $d=0,\dots,p_0-a$.  %This proves half of (ii). 
Similarly, recall that $f_q$ is  defined in \eqref{eq:TonHV2}; for $b\le q\le q_1$, we must have $f_q=q+v_\beta-u_\beta$, and $T^\beta_q$ must lie on 
$P'_{q+v_\beta-u_\beta,q}\rightsquigarrow Q'_{q+v_\beta-u_\beta,q}$. 

Now consider those $q$ in the range $b'\le q\le q_1$. 
Note that $H^\beta_{q+v_\beta-u_\beta}$ is in the last block of $A_\beta$ because 
$q+v_\beta-u_\beta\ge b'+v_\beta-u_\beta = (a+q_1-p_0)+v_\beta-u_\beta \stackrel{\eqref{eq:p0-q1}}{=}
(a+u_\beta-u_\alpha)+v_\beta-u_\beta=a+v_\beta-u_\alpha>v_\beta-u_\alpha=D_1$.
So $P'_{q+v_\beta-u_\beta,q}\rightsquigarrow Q'_{q+v_\beta-u_\beta,q}$ is in the last block of $A_\beta$. 
Apply $\phi_\alpha^{-1}\phi_\beta$, we see that
$\phi_\alpha^{-1}\phi_\beta(T^\beta_q)$ is on 
$\phi_\alpha^{-1}\phi_\beta(P'_{q+v_\beta-u_\beta,q}\rightsquigarrow Q'_{q+v_\beta-u_\beta,q})
=(P_{q+v_\beta-u_\beta-D_1,q+D_2}\rightsquigarrow Q_{q+v_\beta-u_\beta-D_1,q+D_2})
$.
Let $d=q_1-q$, then the indices in the last expression become
$(q+v_\beta-u_\beta-D_1,q+D_2)
=(q_1-d+v_\beta-u_\beta-(v_\beta-u_\alpha),q_1-d+v_\alpha-u_\beta)
=(p_0-u_\alpha+u_\beta-d+v_\beta-u_\beta-(v_\beta-u_\alpha),q_0-v_\alpha+u_\beta-d+v_\alpha-u_\beta)
=(p_0-d,q_0-d)
$. Thus
 $\phi_\alpha^{-1}\phi_\beta(T^\beta_{q_1-d})$ is indeed on  $P_{p_0-d,q_0-d}\rightsquigarrow Q_{p_0-d,q_0-d}$. This proves (ii). 
 
To show (iii) by induction on $d$: for $d=0$, since $S^\alpha_{p_0}$ is to the west of $\phi_\alpha^{-1}(i,j,k)$, 
$\phi_\alpha^{-1}\phi_\beta(S^\beta_{q_1})$ is to the north of $\phi_\alpha^{-1}(i,j,k)$, we see that
$S^\alpha_{p_0}$ is to the SW of $\phi_\alpha^{-1}\phi_\beta(S^\beta_{q_1})$, and therefore 
$T^\alpha_{p_0}$ is weakly SW to $\phi_\alpha^{-1}\phi_\beta(T^\beta_{q_1})$. 
Now assume $T^\alpha_{p_0-(d-1)}$ is weakly SW of $\phi_\alpha^{-1}\phi_\beta(T^\beta_{q_1-(d-1)})$, similar to 
the above argument for $d=0$, we see that
$S^\alpha_{p_0-d}$ is SW of $\phi_\alpha^{-1}\phi_\beta(S^\beta_{q_1-d})$, and therefore 
$T^\alpha_{p_0-d}$ is weakly SW of $\phi_\alpha^{-1}\phi_\beta(T^\beta_{q_1-d})$. 

%This completes the proof of (i)--(iii).

Now using (i)--(iii) for $d=p_0-a$, we see that the four points 
$$(S,T,T',S')=(S^\alpha_a, T^\alpha_a,   \phi_\alpha^{-1}\phi_\beta(T^\beta_{a+q_1-p_0}), \phi_\alpha^{-1}\phi_\beta(T^\beta_{a+q_1-p_0}) )$$ 
are all on the path $P_{a,a+q_0-p_0}\rightsquigarrow Q_{a,a+q_0-p_0}$, and $S$ is to the south of $T$, $T$ is weakly SW of $T'$, $T'$ is to the west of $S'$. But $T=Q_{a,a+q_0-p_0}$ is the NE endpoint of the above path, so $T'=T$ and there is no place for $S'$, a contradiction. (See Figure \ref{fig:Case 4}.)
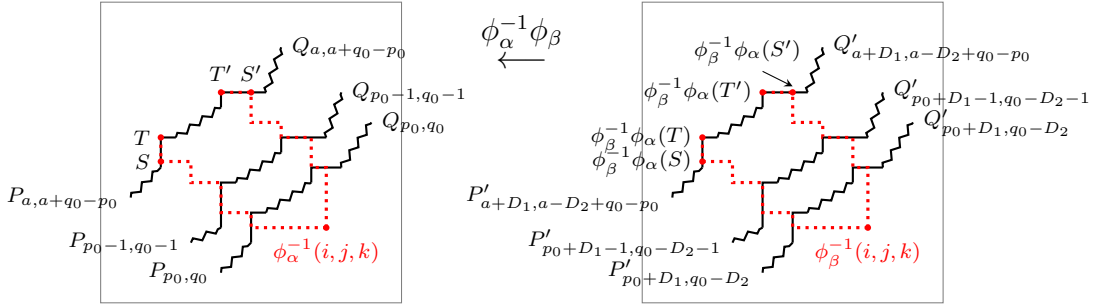
\begin{figure}[h]
\begin{center}
\begin{tikzpicture}[scale=.4]

\draw  [gray] (0,0) rectangle (10,10);

\draw (4,1) node [left] {\tiny $P_{p_0,q_0}$};
\draw [thick, decorate,decoration={zigzag,segment length=2mm, amplitude=0.4mm}] (4,1) -- (5,2);
\draw [thick] (5,2) -- (5,3)--(5.5,3);
\draw [thick, decorate,decoration={zigzag,segment length=2mm, amplitude=0.4mm}] (5.5,3) -- (7,4);
\draw [thick] (7,4) -- (7,4.5)--(8,4.5);
\draw [thick, decorate,decoration={zigzag,segment length=2mm, amplitude=0.4mm}] (8,4.5) -- (9,6);
\draw (9,6) node [right] {\tiny $Q_{p_0,q_0}$};

\draw (3,2) node [left] {\tiny $P_{p_0-1,q_0-1}$};
\draw [thick, decorate,decoration={zigzag,segment length=2mm, amplitude=0.4mm}] (3,2) -- (4,2.5);
\draw [thick] (4,2.5) -- (4,4)--(4.5,4);
\draw [thick, decorate,decoration={zigzag,segment length=2mm, amplitude=0.4mm}] (4.5,4) -- (6,5);
\draw [thick] (6,5) -- (6,5.5)--(7.5,5.5);
\draw [thick, decorate,decoration={zigzag,segment length=2mm, amplitude=0.4mm}] (7.5,5.5) -- (8,7);
\draw (8,7) node [right] {\tiny $Q_{p_0-1,q_0-1}$};

\draw (1,3.5) node [left] {\tiny $P_{a,a+q_0-p_0}$};
\draw [thick, decorate,decoration={zigzag,segment length=2mm, amplitude=0.4mm}] (1,3.5) -- (2,4.5);
\draw [thick] (2,4.5) -- (2,5.5)--(2.5,5.5);
\draw [thick, decorate,decoration={zigzag,segment length=2mm, amplitude=0.4mm}] (2.5,5.5) -- (4,6.5);
\draw [thick] (4,6.5) -- (4,7)--(5.5,7);
\draw [thick, decorate,decoration={zigzag,segment length=2mm, amplitude=0.4mm}] (5.5,7) -- (6,8.5);
\draw (6,8.5) node [right] {\tiny $Q_{a,a+q_0-p_0}$};

%draw the sequence of points
\fill[red] (7.5,2.5) circle(3pt) ;
\node at (7.5,2.5)[anchor=north]{\color{red}\tiny$\phi_\alpha^{-1}(i,j,k)$};
\draw[red, very thick, dotted](7.5,2.5)--(5,2.5)--(5,3)--(4,3)--(4,4)--(3,4)--(3,4.7)--(2,4.7)--(2,5.5); 
\draw[red, very thick, dotted](7.5,2.5)--(7.5,4.5)--(7,4.5)--(7,5.5)--(6,5.5)--(6,6)--(5,6)--(5,7)--(4,7); 

\fill[red] (2,4.7) circle(3pt)  (2,5.5) circle(3pt)  (5,7) circle(3pt)  (4,7) circle(3pt) ;
\draw (2,4.7) node [left] {\tiny $S$};
\draw (2,5.5) node [left] {\tiny $T$};
\draw (4,7) node [above] {\tiny $T'$};
\draw (5,7) node [above] {\tiny $S'$};

\node at (14, 9) {$\phi_\alpha^{-1}\phi_\beta$};
\node at (14, 8) {$\longleftarrow$};

\begin{scope}[shift={(18,0)}]
\draw  [gray] (0,0) rectangle (10,10);

\draw (4,1) node [left] {\tiny $P'_{p_0+D_1,q_0-D_2}$};
\draw [thick, decorate,decoration={zigzag,segment length=2mm, amplitude=0.4mm}] (4,1) -- (5,2);
\draw [thick] (5,2) -- (5,3)--(5.5,3);
\draw [thick, decorate,decoration={zigzag,segment length=2mm, amplitude=0.4mm}] (5.5,3) -- (7,4);
\draw [thick] (7,4) -- (7,4.5)--(8,4.5);
\draw [thick, decorate,decoration={zigzag,segment length=2mm, amplitude=0.4mm}] (8,4.5) -- (9,6);
\draw (9,6) node [right] {\tiny $Q'_{p_0+D_1,q_0-D_2}$};

\draw (3,2) node [left] {\tiny $P'_{p_0+D_1-1,q_0-D_2-1}$};
\draw [thick, decorate,decoration={zigzag,segment length=2mm, amplitude=0.4mm}] (3,2) -- (4,2.5);
\draw [thick] (4,2.5) -- (4,4)--(4.5,4);
\draw [thick, decorate,decoration={zigzag,segment length=2mm, amplitude=0.4mm}] (4.5,4) -- (6,5);
\draw [thick] (6,5) -- (6,5.5)--(7.5,5.5);
\draw [thick, decorate,decoration={zigzag,segment length=2mm, amplitude=0.4mm}] (7.5,5.5) -- (8,7);
\draw (8,7) node [right] {\tiny $Q'_{p_0+D_1-1,q_0-D_2-1}$};

\draw (1,3.5) node [left] {\tiny $P'_{a+D_1,a-D_2+q_0-p_0}$};
\draw [thick, decorate,decoration={zigzag,segment length=2mm, amplitude=0.4mm}] (1,3.5) -- (2,4.5);
\draw [thick] (2,4.5) -- (2,5.5)--(2.5,5.5);
\draw [thick, decorate,decoration={zigzag,segment length=2mm, amplitude=0.4mm}] (2.5,5.5) -- (4,6.5);
\draw [thick] (4,6.5) -- (4,7)--(5.5,7);
\draw [thick, decorate,decoration={zigzag,segment length=2mm, amplitude=0.4mm}] (5.5,7) -- (6,8.5);
\draw (6,8.5) node [right] {\tiny $Q'_{a+D_1,a-D_2+q_0-p_0}$};

%draw the sequence of points
\fill[red] (7.5,2.5) circle(3pt) ;
\node at (7.5,2.5)[anchor=north]{\color{red}\tiny$\phi_\beta^{-1}(i,j,k)$};
\draw[red, very thick, dotted](7.5,2.5)--(5,2.5)--(5,3)--(4,3)--(4,4)--(3,4)--(3,4.7)--(2,4.7)--(2,5.5); 
\draw[red, very thick, dotted](7.5,2.5)--(7.5,4.5)--(7,4.5)--(7,5.5)--(6,5.5)--(6,6)--(5,6)--(5,7)--(4,7); 

\fill[red] (2,4.7) circle(3pt)  (2,5.5) circle(3pt)  (5,7) circle(3pt)  (4,7) circle(3pt) ;
\draw (2,4.7) node [left] {\tiny $\phi_\beta^{-1}\phi_\alpha(S)$};
\draw (2,5.5) node [left] {\tiny $\phi_\beta^{-1}\phi_\alpha(T)$};
\draw (4,7) node [left] {\tiny $\phi_\beta^{-1}\phi_\alpha(T')$};
\draw (3.5,7.5) node [above] {\tiny $\phi_\beta^{-1}\phi_\alpha(S')$};\draw [-stealth](4,7.7) -- (5,7.2);

\end{scope}
\end{tikzpicture}       
\end{center}
\caption{Case 4. Left: last block of $A_\alpha$; Right: last block of in $A_\beta$. The points on the right side coincide with the points on the left side under the map $\phi_\alpha^{-1}\phi_\beta$.}
\label{fig:Case 4}
\end{figure}
So in all four cases we get the desired contradiction. This completes the proof of Proposition \ref{prop:C(S)>=C}.
\end{proof}

\subsection{Chute moves}

Recall the classical chute moves of pipe dreams (for example, see \cite[\S16.1]{MS}), where a ``$+$'' is moved from the NE corner to the SW corner, as shown in Figure \ref{fig: classical chute move}. 
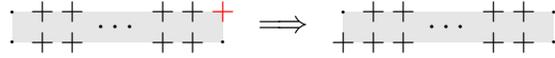
\begin{figure}[ht]
\begin{tikzpicture}[scale=.4]
    \begin{scope}
      \fill[gray!20](0,0)--(7,0)--(7,1)--(0,1)--(0,0);
      \node at (0,1) {.}; \node at (1,1){+}; \node at (2,1){+}; \node at (5,1){+}; \node at (6,1){+};\node[red]at (7,1) {+};
      \node at (0,0) {.};\node at (1,0){+};\node at (2,0){+};\node at (5,0){+};\node at (6,0){+};\node at (7,0) {.};
      \node at (3.5,0.5) {$\cdots$};
      \node at (9, 0.5) {$\Longrightarrow$};
      \end{scope}
    \begin{scope}[shift={(11,0)}]
      \fill[gray!20](0,0)--(7,0)--(7,1)--(0,1)--(0,0);
      \node at (0,1) {.}; \node at (1,1){+}; \node at (2,1){+}; \node at (5,1){+}; \node at (6,1){+};\node at (7,1) {.};
      \node at (0,0) {+};\node at (1,0){+};\node at (2,0){+};\node at (5,0){+};\node at (6,0){+};\node at (7,0) {.};
      \node at (3.5,0.5) {$\cdots$};
      \end{scope}
 \end{tikzpicture}
 \caption{Classical chute move.}
 \label{fig: classical chute move}
 \end{figure}
Chute moves can be used to generate all reduced pipe dreams.
\iffalse
\begin{example}
{\small
Pipe dreams that corresponds to facets of the Stanley-Reisner complex of  $I_2^{(3,3)}$.
 (dimension = \# of dots in a pipe dream, multiplicity = \# of pipe dreams) The left is the maximum under $>_T$; the right is the minimum. A chute move decreases an element under $>_T$}. 
\end{example}

\begin{figure}[h]\setlength{\unitlength}{1.2pt}
\begin{picture}(40,30)
\put(0,0){$\cdot$}\put(10,0){$\cdot$}\put(20,0){$\cdot$}
\put(0,10){$+$}\put(10,10){$+$}\put(20,10){$\cdot$}
\put(0,20){$+$}\put(10,20){$+$}\put(20,20){$\cdot$}
\end{picture}
\begin{picture}(40,30)
\put(0,0){$\cdot$}\put(10,0){$\cdot$}\put(20,0){$+$}
\put(0,10){$+$}\put(10,10){$\cdot$}\put(20,10){$\cdot$}
\put(0,20){$+$}\put(10,20){$+$}\put(20,20){$\cdot$}
\end{picture}
\begin{picture}(40,30)
\put(0,0){$\cdot$}\put(10,0){$+$}\put(20,0){$+$}
\put(0,10){$\cdot$}\put(10,10){$\cdot$}\put(20,10){$\cdot$}
\put(0,20){$+$}\put(10,20){$+$}\put(20,20){$\cdot$}
\end{picture}
\begin{picture}(40,30)
\put(0,0){$\cdot$}\put(10,0){$\cdot$}\put(20,0){$+$}
\put(0,10){$+$}\put(10,10){$\cdot$}\put(20,10){$+$}
\put(0,20){$+$}\put(10,20){$\cdot$}\put(20,20){$\cdot$}
\end{picture}
\begin{picture}(40,30)
\put(0,0){$\cdot$}\put(10,0){$+$}\put(20,0){$+$}
\put(0,10){$\cdot$}\put(10,10){$\cdot$}\put(20,10){$+$}
\put(0,20){$+$}\put(10,20){$\cdot$}\put(20,20){$\cdot$}
\end{picture}
\begin{picture}(40,30)
\put(0,0){$\cdot$}\put(10,0){$+$}\put(20,0){$+$}
\put(0,10){$\cdot$}\put(10,10){$+$}\put(20,10){$+$}
\put(0,20){$\cdot$}\put(10,20){$\cdot$}\put(20,20){$\cdot$}
\end{picture}
\end{figure}
\fi
We shall modify the definition of chute move described in \cite[Definition 16.6]{MS}, and use it to generate all concurrent vertex maps from a special one (see Proposition \ref{prop:PDC to Cmax}). 
\begin{definition}\label{df:chute}
Let $C\subseteq L$. We denote the points in $C$ by ``$\cdot$''  and points in $L\setminus C$ by ``$+$''.
For $\alpha\in V_{\rm target}$, a {\bf horizontal chutable rectangle} is a $2\times r$ block inside $A_\alpha$ consisting of  ``$+$'' and ``$\cdot$'' such that $r\ge 2$, and the only ``$\cdot$'' in the block are its NE, SE, SW corners, that is, there is a rectangle $[i,i+1]_\mathbb{R}\times[j,j+r-1]_\mathbb{R}\subseteq [1, a_\alpha]_\mathbb{R}\times[1, b_\alpha]_\mathbb{R}$ such that 
$$\phi_\alpha^{-1}(C)\cap \big( [i,i+1]_\mathbb{R}\times[j,j+r-1]_\mathbb{R}\big) =\{(i+1,j), (i,j+r-1), (i+1,j+r-1)\}.$$
Applying a {\bf horizontal chute move} to $C$ at $(i+1,j+r-1)$ (or $\phi_\alpha^{-1}(i+1,j+r-1)$) is accomplished by placing a ``$\cdot$'' in the NW corner of a horizontal chutable rectangle and removing the ``$\cdot$'' from the SE corner of that rectangle; in other words, 
$$C'= C  \cup \{\phi_\alpha((i,j)\}  \setminus \{\phi_\alpha(i+1,j+r-1)\}.$$
where $i,j,r$ are as above.
We denote the above horizontal chute move by $C\to C'$.  Reversely, we say that $C'\dashrightarrow C$ is an {\bf inverse horizontal chute move}. %
We say that $(i+1,j+r-1)$ (or $\phi_\alpha(i+1,j+r-1)$) is {\bf horizontal chute movable}. 
Similarly, we can define a {\bf vertical chutable rectangle}, a {\bf vertical chute move}, an {\bf inverse vertical chute move}, and {\bf vertical chute movable}. (See Figure \ref{fig: horizontal  chute move}.)
\begin{figure}[ht]
\begin{center}
\begin{tikzpicture}[scale=.5]
    \begin{scope}
      \fill[gray!20](0,0)--(3,0)--(3,1)--(0,1)--(0,0);
      \node at (0,0){.};
      \node [red] at (0,1) {+};
      \node at (1,0) {+};
      \node at (1,1) {+};
      \node at (2,0) {+};
      \node at (2,1) {+};
      \node at (3,0) {.};
      \node at (3,1) {.};
      \draw [red, <->,>=stealth] (0.2, 0.8) -- (2.8,0.2);
      \node at (4.5, 0.5) {$\longrightarrow$};
      \end{scope}
    \begin{scope}[shift={(6,0)}]
      \fill[gray!20](0,0)--(3,0)--(3,1)--(0,1)--(0,0);
      \node at (0,0){.};
      \node at (0,1) {.};
      \node at (1,0) {+};
      \node at (1,1) {+};
      \node at (2,0) {+};
      \node at (2,1) {+};
      \node at (3,0) {+};
      \node at (3,1) {.};
      \end{scope}
  \begin{scope}[shift={(14,0)}]
    \begin{scope}
      \fill[gray!20](0,0)--(0,3)--(1,3)--(1,0)--(0,0);
      \node at (0,0){.};
      \node  at (1,0) {.};
      \node at (0,1) {+};
      \node at (1,1) {+};
      \node at (0,2) {+};
      \node at (1,2) {+};
      \node [red] at (0,3) {+};
      \node at (1,3) {.};
      \draw [red, <->,>=stealth] (0.8, 0.2) -- (0.2,2.8);
      \node at (2.5, 0.5) {$\longrightarrow$};
      \end{scope}
    \begin{scope}[shift={(4,0)}]
      \fill[gray!20](0,0)--(0,3)--(1,3)--(1,0)--(0,0);
      \node at (0,0){.};
      \node at (1,0) {+};
      \node at (0,1) {+};
      \node at (1,1) {+};
      \node at (0,2) {+};
      \node at (1,2) {+};
      \node at (0,3) {.};
      \node at (1,3) {.};
      \end{scope}
  \end{scope}    
 \end{tikzpicture}
 \end{center}
 \caption{Left: a horizontal chutable rectangle and a horizontal chute move; Right:  a vertical chutable rectangle and a vertical chute move.}
 \label{fig: horizontal  chute move}
 \end{figure}
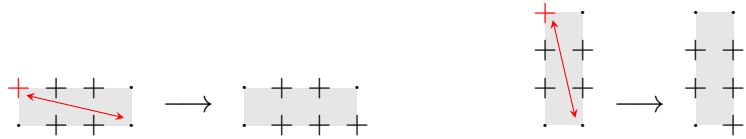

%Similarly, for For $\beta\in V_{\rm source}$, a {\bf vertical chutable rectangle} is a $r\times 2$ block inside $A_\beta$ consisting of ``$+$'' and ``$\cdot$'' such that $r\ge 2$, and the only ``$\cdot$'' in the block are its NE, SE, SW corners, that is, there is a rectangle $[i,i+r-1]\times[j,j+1]$ inside $[1\times a_\beta]\times[1\times b_\beta]$ such that 
%$$\phi_\beta^{-1}(C)\cap \big( [i,i+r-1]\times[j,j+1]\big) =\{(i+r-1,j), (i,j+1), (i+r-1,j+1)\}.$$
%Applying a {\bf vertical chute move} to $C$ at $(i+r-1,j+1)$ (or $\phi_\beta^{-1}(i+r-1,j+1)$) is accomplished by placing a ``$\cdot$'' in the NW corner of a vertical  chutable rectangle and removing the ``$\cdot$'' from the SE corner of that rectangle; in other words, we obtain $C'$ satisfying
%$$\phi^{-1}_\beta(C')=\big( \phi^{-1}_\beta(C)\setminus \{(i+r-1,j+1)\} \big) \cup\{(i,j)\}$$ where $i,j,r$ are as above.
%We denote the above vertical chute move by $C\to C'$.  Reversely, we say that $C'\dashrightarrow C$ is an {\bf inverse vertical chute move}.
%We say that $(i+r-1,j+1)$ (or $\phi_\beta(i+r-1,j+1)$) is {\bf vertical chute movable}. 

\end{definition}

\begin{remark}
(i) Note that if $P\in C$ is chute movable, then there is no ambiguity in determining which chute move to apply. Indeed, if both horizontal and vertical chute moves can be applied to $C$ at $P$, then the chutable rectangle is a $2\times 2$ block and the horizontal chute move coincides with the vertical one.

(ii) A horizontal (resp.~vertical) move defined here is analog to a chute move (resp.~ladder move) defined by Bergeron and Sara Billey; see \cite{Bergeron-Billey} and \cite[Def 3.7.2]{Knutson-Miller}.

(iii) Under the order ``$<_T$'', a (horizontal or vertical) chute move decreases an element, and
an inverse (horizontal or vertical) chute move increases an element.
\end{remark}

\begin{lemma}\label{lemma:chute move keep compatibility}
A (horizontal or vertical) chute move sends a $\uu$-compatible subset of $L$ to a $\uu$-compatible subset of $L$. In particular, it sends a concurrent vertex map $C$ to a concurrent vertex map $C'$ where $C' <_T C$. The analog holds for an inverse (horizontal or vertical) chute move, except that $C'$ satisfies $C'>_T C$ instead.
%Similarly, an inverse horizontal (resp. vertical) chute move sends a $\uu$-compatible subset of $L$ to a $\uu$-compatible subset of $L$. In particular, it sends a concurrent vertex map $C$ to a concurrent vertex map $C'$ where $C' >_T C$.
\end{lemma}
\begin{proof}
We only prove that if an inverse horizontal chute move sends a $\uu$-compatible subset $C$ to $C'$, then $C'$ is $\uu$-compatible; the  rest can be proved similarly. 
Let $\alpha,i,j,r$ be defined in Definition \ref{df:chute}. 
%We claim that $C'$ is $\uu$-compatible.
We assume by contradiction that there exists $\gamma\in {\rm V}_\mathcal{Q}$ such that ${C'}^\gamma$ contains a diagonal chain $D= \{ P_1,\dots,P_{u_\gamma+1}\}$. 
Since $D$ cannot be a diagonal chain in $C^\gamma$,  the point $\phi_\gamma^{-1}\phi_\alpha(i+1,j+r-1)$  must coincide with a point $P_\ell$ in $D$. 
Assume $P_\ell$ is in ${\rm Page}_k$ and the arrow $h_k$ is $\beta\to\alpha$. 
Define 
$P''_\ell=\phi_\gamma^{-1}\phi_\alpha(i,j+r-1)$,
and define (if possible)
$P'_\ell=\phi_\gamma^{-1}\phi_\alpha(i+1,j)$, 
$P'''_\ell=\phi_\gamma^{-1}\phi_\alpha(i,j)$.
We claim that either $D\cup \{P'_\ell\} \setminus \{P_\ell\}$ or $D\cup \{P''_\ell\} \setminus \{P_\ell\}$ is a diagonal chain in $C$ of size $(u_\gamma+1)$, contradicting the assumption that $C$ is a concurrent vertex map. We prove the claim by considering the following two cases separately.
\begin{figure}[ht]
\begin{tikzpicture}[scale=.35]
    \begin{scope}
      \fill[gray!20](0,0)--(7,0)--(7,1)--(0,1)--(0,0);
      \node at (0,1) {+}; \node at (1,1){+}; \node at (2,1){+}; \node at (5,1){+}; \node at (6,1){+};\node at (7,1) {.};
      \node at (0,0) {.};\node at (1,0){+};\node at (2,0){+};\node at (5,0){+};\node at (6,0){+};\node at (7,0) {.};
      \node at (3.5,0.5) {$\cdots$};
      \node [left] at (-.2,1.3) {$P'''_\ell$};
      \node [left] at (-.2,-.3) {$P'_\ell$};
      \node [right] at (7,1.3) {$P''_\ell$};
      \node [right] at (7,-.3) {$P_\ell$};
      \end{scope}
    \begin{scope}[shift={(14,-2)}]
      \draw  [gray] (-.5,-1.5) rectangle (7.5,2.5);
      \draw  [gray] (-.5,2.5) rectangle (7.5,7);
      \fill[gray!20](3,4)--(7.5,4)--(7.5,5)--(3,5)--(3,4);
      \fill[gray!20](-.5,0)--(1,0)--(1,1)--(-.5,1)--(-.5,0);
      \node at (3,5) {+}; \node at (4,5){+}; \node at (5,5){+}; \node at (6,5){+}; \node at (7,5){+};\node at (0,1) {+}; \node at (1,1) {.};
      \node at (3,4) {.};\node at (4,4){+};\node at (5,4){+};\node at (6,4){+};\node at (7,4){+};\node at (0,0) {+};\node at (1,0){.};
      \node [left] at (3,5.2) {$P'''_\ell$};
      \node [left] at (2.8,4) {$P'_\ell$};      
      \node [right] at (1,1.2) {$P''_\ell$};
      \node [right] at (1,0) {$P_\ell$};
      \node at (0,2) {.}; \node[left] at (0,2) {$P_{\ell-1}$};
%      \node at (6,3){.};\node [left] at (6,3) {$P_{\ell+1}$};
      \end{scope}
 \end{tikzpicture}
 \caption{Lemma \ref{lemma:chute move keep compatibility}. Pictures of $C'$. Left: Case 1; Right: Case 2.}
 \label{fig: chute_case12}
 \end{figure}
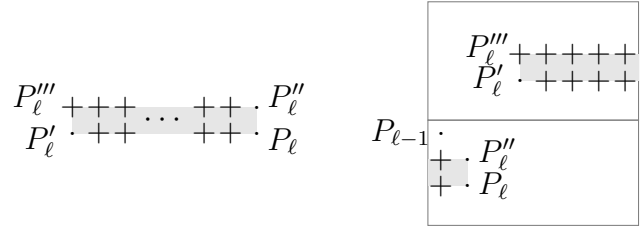

Case 1: either $\gamma=\alpha$, or ``$\gamma=\beta$ and $P'''_\ell$ is defined and lies in the same block of $A_\beta$ with $P_\ell$''.
Then $D\cup \{P'_\ell\} \setminus \{P_\ell\}$  is a diagonal chain if  $P_{\ell-1}$ is NW of $P'_\ell$; 
$D\cup \{P''_\ell\} \setminus \{P_\ell\}$ is a diagonal chain if $P_{\ell-1}$ is NW of $P''_\ell$. 
If none of the these two conditions hold, then $P_{\ell-1}$, which is a point in  $\phi_\alpha^{-1}(C')$, must be weakly SE of $P'''_\ell$ and (strongly) NW of $P_\ell$, which gives the desired contradiction because such a point does not exist. See Figure \ref{fig: chute_case12} Left. 

Case 2: $\gamma=\beta$, and either $P'''_\ell$ is undefined or  $P'''_\ell$ and $P_\ell$ lie in different blocks of $A_\beta$.  
In this case, $\phi_\beta^{-1}(C')$ has no point to the left of $P''_\ell$.
Note that $D\cup \{P''_\ell\} \setminus \{P_\ell\}$ is a diagonal chain in $C^\gamma$ of size $(u_\gamma+1)$;
%Indeed, it suffices to show that  $(P_{\ell-1})_x \le  (P_\ell)_x-2$. 
otherwise, $(P_{\ell-1})_x = (P_\ell)_x-1$, thus $P_{\ell-1}$ is a point in $\phi_\beta^{-1}(C')$ and lies to the left of $P''_\ell$, a contradiction. 
 See Figure \ref{fig: chute_case12} Right. 
\end{proof}

\begin{definition}
	We define the  {\bf initial concurrent vertex map} to be $C_{\rm max}(\emptyset)$. 
\end{definition}
%We have the following equivalent characterizations of  $C_{\rm max}(\emptyset)$. 
\begin{lemma}\label{lem:initial C} 
Let $C$ be a concurrent vertex map. The following are equivalent.

{\rm(i)} $C=C_{\rm max}(\emptyset)$;

{\rm(ii)} $C$ has no essential NW corners;

{\rm(iii)} $C$ is not inverse chute movable;

{\rm(iv)} for each $k$ with $h_k:\beta\to\alpha$,  define 
$S_{k}^\alpha := \sum_{\beta' \stackrel{>k}{\to} \alpha} u_{\beta'},$
where the sum is over all arrows $h_{k'}:\beta'\to \alpha$ satisfying $k'>k$ and $\target(h_k)=\alpha$; define
$S_{k}^\beta := \sum_{\alpha'\stackrel{>k}{\leftarrow}\beta} u_{\alpha'},$
where the sum is over all arrows $h_{k'}:\beta\to \alpha'$ satisfies the condition that $k'>k$ and $\source(h_k)=\beta$;
then 
\begin{equation}\label{eq:Ck equiv conditions}
\aligned
C|_k=
&\Big( (\textrm{the bottom $u_\alpha$ rows}) \bigcup (\textrm{the rightmost $[u_\alpha-S_k^\alpha]_+$ columns})\Big)\\
&\bigcap 
\Big( (\textrm{the bottom $[u_\beta-S_k^\beta]_+$ rows}) \bigcup (\textrm{the rightmost $u_\beta$ columns})\Big).
\endaligned
\end{equation}
It can be pictured in two cases as shown in Figure \ref{fig:C_k}. 

%Equivalently (see Figure \ref{fig:C_k}), 

%(Case 1) if $(u_\alpha-[u_\beta-S_k^\beta]_+)(u_\beta-[u_\alpha-S_k^\alpha]_+)\ge0$, then $(i,j)\in C|_k$ if and only if `` $i>a_\alpha-\min(u_\alpha,[u_\beta-S_k^\beta]_+)$'', or `` $j>b_\beta-\min(u_\beta,[u_\alpha-S_k^\alpha]_+)$'', or ``~$ i>a_\alpha-\max(u_\alpha,[u_\beta-S_k^\beta]_+)$ and $j>b_\beta-\max(u_\beta,[u_\alpha-S_k^\alpha]_+) $'';

%(Case 2) if $u_\alpha-[u_\beta-S_k^\beta]_+$ and  $u_\beta-[u_\alpha-S_k^\alpha]_+$ either have the opposite signs or at least one of them is zero, then $(i,j)\in C|_k$ if and only if `` $i>a_\alpha-\min(u_\alpha,[u_\beta-S_k^\beta]_+)$'', or ``~$j>b_\beta-\min(u_\beta,[u_\alpha-S_k^\alpha]_+)$''.

%(Note that when Case 1 and Case 2 overlap, the descriptions of $C|_k$ are consistent.)
\begin{figure}[ht]
\begin{center}
\begin{tikzpicture}[scale=.5]
    \begin{scope}[shift={(0,0)}]
      \draw[draw=black,fill=white] (0,0) rectangle ++(6,6);
      \draw[fill=gray!30] (0,0)--(0,1.5)--(3,1.5)--(3,3)--(4,3)--(4,6)--(6,6)--(6,0);
      \draw[dashed] (3,0)--(3,1.5) (4,3)--(6,3);
      \draw[decorate,decoration={brace,mirror,raise=3pt,amplitude=4pt}, thick]
    (3,0)--(6,0)  node [midway,yshift=-0.6cm] {\tiny $\max(u_\beta,[u_\alpha-S_k^\alpha]_+)$};
      \draw[decorate,decoration={brace,mirror,raise=3pt,amplitude=4pt}, thick]
    (6,6)--(4,6)  node [midway,yshift=.6cm] {\tiny $\min(u_\beta,[u_\alpha-S_k^\alpha]_+)$};
      \draw[decorate,decoration={brace,mirror,raise=3pt,amplitude=4pt}, thick]
    (0,1.5)--(0,0)  node [midway,yshift=-0.4cm] {};
      \node[label={[label distance=0.5cm,text depth=-1ex,rotate=90]right: {\tiny $\min(u_\alpha,[u_\beta-S_k^\beta]_+)$}}] at (-1.5,-1) {};
      \draw[decorate,decoration={brace,mirror, raise=3pt,amplitude=4pt}, thick]
    (6,0)--(6,3)  node [midway,xshift=0.8cm] {};
      \node[label={[label distance=0.5cm,text depth=-1ex,rotate=90]right: {\tiny $\max(u_\alpha,[u_\beta-S_k^\beta]_+)$}}] at (6.5,-.5) {};
      \node at(3,-2.5) {{\rm Case 1}};
      \end{scope}
    \begin{scope}[shift={(12,0)}]
      \draw[draw=black,fill=white] (0,0) rectangle ++(6,6);
      \draw[fill=gray!30] (0,0)--(0,1.5)--(4,1.5)--(4,6)--(6,6)--(6,0);
      \draw[dashed] (4,0)--(4,1.5)--(6,1.5);
      \draw[decorate,decoration={brace,mirror,raise=3pt,amplitude=4pt}, thick]
    (6,6)--(4,6)  node [midway,yshift=.6cm] {\tiny $\min(u_\beta,[u_\alpha-S_k^\alpha]_+)$};
      \draw[decorate,decoration={brace,mirror,raise=3pt,amplitude=4pt}, thick]
    (0,1.5)--(0,0)  node [midway,yshift=-0.4cm] {};
      \node[label={[label distance=0.5cm,text depth=-1ex,rotate=90]right: {\tiny $\min(u_\alpha,[u_\beta-S_k^\beta]_+)$}}] at (-1.5,-1) {};
      \node at(3,-2.5) {\rm Case 2};
      \end{scope}
\end{tikzpicture}
\caption{Lemma \ref{lem:initial C} (iv) $C|_k$ consists of the lattice points in the gray region. Case 1 is when $(u_\alpha-[u_\beta-S_k^\beta]_+)(u_\beta-[u_\alpha-S_k^\alpha]_+)\ge0$. Case 2 is when $(u_\alpha-[u_\beta-S_k^\beta]_+)(u_\beta-[u_\alpha-S_k^\alpha]_+)\le 0$. 
}
\label{fig:C_k}
\end{center}
\end{figure}
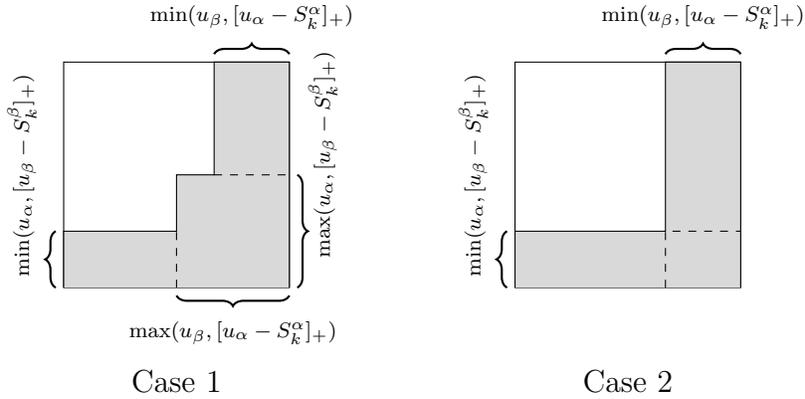

{\rm(v)} The straight road map $(\{H^\alpha_{p}\}, \{V^\beta_{q}\})$ corresponding to $C$ is as follows.
For $\alpha\in V_{\rm target}$ and $1\le p\le u_\alpha$, let $k$ be the unique integer such that the target of $h_k$ is $\alpha$, and 
\begin{equation}\label{eq:k quniquely determined}
(S_k^\alpha=) \sum_{\beta' \stackrel{>k}{\to} \alpha} u_{\beta'}<u_\alpha-p+1\le \sum_{\beta' \stackrel{\ge k}{\to} \alpha} u_{\beta'}.
\end{equation}
Define
$j^\alpha_p := p-u_\alpha+\sum_{\beta' \stackrel{\le k}{\to} \alpha} b_{\beta'}+  \sum_{\beta' \stackrel{> k}{\to} \alpha} u_{\beta'}.$
	Then $H_p^\alpha$ is the path obtained by connecting vertices $(a_\alpha-u_\alpha+p,1)$, $(a_\alpha-u_\alpha+p,j^\alpha_p)$, $(p, j^\alpha_p)$, $(p,b_\alpha)$. 
%In particular, on ${\rm Page}_k$ for $k=r_1,r_2,\dots,r_s$ as defined in Lemma \ref{lem:straight 2 conditions} (b), the bottom $S^\alpha_k$ and top $[u_\alpha-S^\alpha_k-u_{\source(h_k)}]_+$ paths are straight horizontal. 
%and those in between are from SW to NE obtained by SE-most construction. 
(See Figure \ref{fig:H and V} Left.)

For $\beta\in V_{\rm source}$ and $1\le q\le u_\beta$, let $k$ be the unique integer such that the source of $h_k$ is $\beta$, and
$$(S_k^\beta=) \sum_{\alpha'\stackrel{>k}{\leftarrow}\beta} u_{\alpha'}<u_\beta-q+1\le \sum_{\alpha'\stackrel{\ge k}{\leftarrow}\beta} u_{\alpha'}.$$
Define $j^\beta_q :=q-u_\beta+\sum_{\alpha' \stackrel{\le k}{\leftarrow} \beta} a_{\alpha'}+  \sum_{\alpha' \stackrel{> k}{\leftarrow} \beta} u_{\alpha'}.$ 
Then $V_q^\beta$ is the path obtained by connecting vertices $(1,b_\beta-u_\beta+q)$, $(j^\beta_q, b_\beta-u_\beta+q)$, $(j^\beta_q,q)$, $(a_\beta,q)$. 
%In particular, on ${\rm Page}_k$ for $k=r'_1,r'_2,\dots,r'_t$ as defined in Lemma \ref{lem:straight 2 conditions} (c), the rightmost $S^\beta_k$ and leftmost $[u_\beta-S^\beta_k-u_{\target(h_k)}]_+$ paths are straight vertical. 
%and those in between are from NE to SW obtained by SE-most construction. 
(See Figure \ref{fig:H and V} Right.)

\begin{figure}[ht]
\begin{center}
\begin{tikzpicture}[scale=.35]
\begin{scope}[shift={(0,13)}]
\draw  [gray] (2,0.5) rectangle (5.5,5.5);
\draw  [gray] (6,0.5) rectangle (11.5,5.5);
\draw  [gray] (12,0.5) rectangle (15.5,5.5);

\draw [blue,densely dotted,very thick] (2,2)--(11,2)--(11,4.5)--(15.5,4.5);

\draw (2,2) node [above] {\tiny $(a_\alpha-u_\alpha+p,1)$};

\draw (8,1) node [above] {\tiny $H^\alpha_p$};

\draw (11,2) node [right] {\tiny $(a_\alpha-u_\alpha+p,j^\alpha_p)$};

\draw (11,4.5) node [above] {\tiny $(p,j_p^\alpha)$};

\draw (15.5,4.5) node [right] {\tiny $(p,b_\alpha)$};

\begin{scope}[shift={(20,0)}]
\draw  [gray] (0,0.5) rectangle (5.5,3.5);
\draw  [gray] (0,4) rectangle (5.5,7.5);
\draw  [gray] (0,8) rectangle (5.5,10.5);

\draw [red,densely dotted,very thick](4,10.5)--(4,4.5)--(1.5,4.5)--(1.5,0.5);

\draw (4,10.5) node [above] {\tiny $(1,b_\beta-u_\beta+q)$};

\draw (4,9) node [left] {\tiny $V^\beta_q$};

\draw (4,4.5) node [right] {\tiny $(j^\beta_q, b_\beta-u_\beta+q)$};

\draw (1.5,4.5) node [above] {\tiny $(j_q^\beta,q)$};

\draw (2,0.5) node [below] {\tiny $(a_\beta,q)$};
\end{scope}
\end{scope}

\begin{scope}
\draw  [gray] (2,0.5) rectangle (5.5,5.5);
\draw  [gray] (6,0.5) rectangle (11.5,5.5);
\draw  [gray] (12,0.5) rectangle (15.5,5.5);

\draw [blue,densely dotted,very thick](2,0.5)--(15.5,0.5)--(15.5,3);
\draw [blue,densely dotted,very thick](2,1)--(15,1)--(15,3.5)--(15.5,3.5);
\draw [blue,densely dotted,very thick](2,1.5)--(11.5,1.5)--(11.5,4)--(15.5,4);
\draw [blue,densely dotted,very thick](2,2)--(11,2)--(11,4.5)--(15.5,4.5);
\draw [blue,densely dotted,very thick](2,2.5)--(5.5,2.5)--(5.5,5)--(15.5,5);
\draw [blue,densely dotted,very thick](2,3)--(5,3)--(5,5.5)--(15.5,5.5);

\draw (2,3) node [left] {\tiny $H^\alpha_1$};
\draw (2,0.5) node [left] {\tiny $H^\alpha_{u_\alpha}$};
\end{scope}

\begin{scope}[shift={(2,6)}]
\draw  [gray] (0,0.5) rectangle (5.5,5.5);

\draw [blue,densely dotted,very thick](0,0.5)--(5.5,0.5);
\draw [blue,densely dotted,very thick](0,1)--(5.5,1);
\draw [blue,densely dotted,very thick](0,1.5)--(5.5,1.5)--(5.5,4);
\draw [blue,densely dotted,very thick](0,2)--(5,2)--(5,4.5)--(5.5,4.5);
\draw [blue,densely dotted,very thick](0,5)--(5.5,5);
\draw [blue,densely dotted,very thick](0,5.5)--(5.5,5.5);
      \draw[decorate,decoration={brace,mirror, raise=3pt,amplitude=4pt}, thick]
    (5.7,0.5)--(5.7,1)  node [midway,xshift=0.8cm] {};
      \node[label={[label distance=0.5cm,text depth=-1ex]right: $S_k^\alpha$}] at (5,0.8) {};
      \draw[decorate,decoration={brace,mirror, raise=3pt,amplitude=4pt}, thick]
    (5.7,5)--(5.7,5.5)  node [midway,xshift=0.8cm] {};
      \node[label={[label distance=0.5cm,text depth=-1ex]right: $[u_\alpha-S_k^\alpha-u_{\source(h_k)}]_+$}] at (5,5.3) {};

\end{scope}

\begin{scope}[shift={(20,0)}]
\draw  [gray] (0,0.5) rectangle (5.5,3.5);
\draw  [gray] (0,4) rectangle (5.5,7.5);
\draw  [gray] (0,8) rectangle (5.5,10.5);

\draw [red,densely dotted,very thick](5.5,10.5)--(5.5,0.5)--(3,0.5);
\draw [red,densely dotted,very thick](5,10.5)--(5,1)--(2.5,1)--(2.5,0.5);
\draw [red,densely dotted,very thick](4.5,10.5)--(4.5,4)--(2,4)--(2,0.5);
\draw [red,densely dotted,very thick](4,10.5)--(4,4.5)--(1.5,4.5)--(1.5,0.5);
\draw [red,densely dotted,very thick](3.5,10.5)--(3.5,8)--(1,8)--(1,0.5);
\draw [red,densely dotted,very thick](3,10.5)--(3,8.5)--(0.5,8.5)--(0.5,0.5);
\draw [red,densely dotted,very thick](2.5,10.5)--(2.5,9)--(0,9)--(0,0.5);

\draw (2.5,10.5) node [above] {\tiny $V^\beta_1$};

\draw (5.5,10.5) node [above] {\tiny $V^\beta_{u_\beta}$};

\end{scope}

\end{tikzpicture}       
\end{center}
\caption{Lemma \ref{lem:initial C} (v), $H^\alpha_p$ and $V^\beta_q$. Left Top: a horizontal path $H^\alpha_p$; Left Middle: all horizontal paths in ${\rm Page}_k$; Left Bottom: all horizontal paths in $A_\alpha$; Right Top: a vertical path $V^\beta_q$; Right Bottom: all vertical paths in $A_\beta$.}
\label{fig:H and V}
\end{figure}
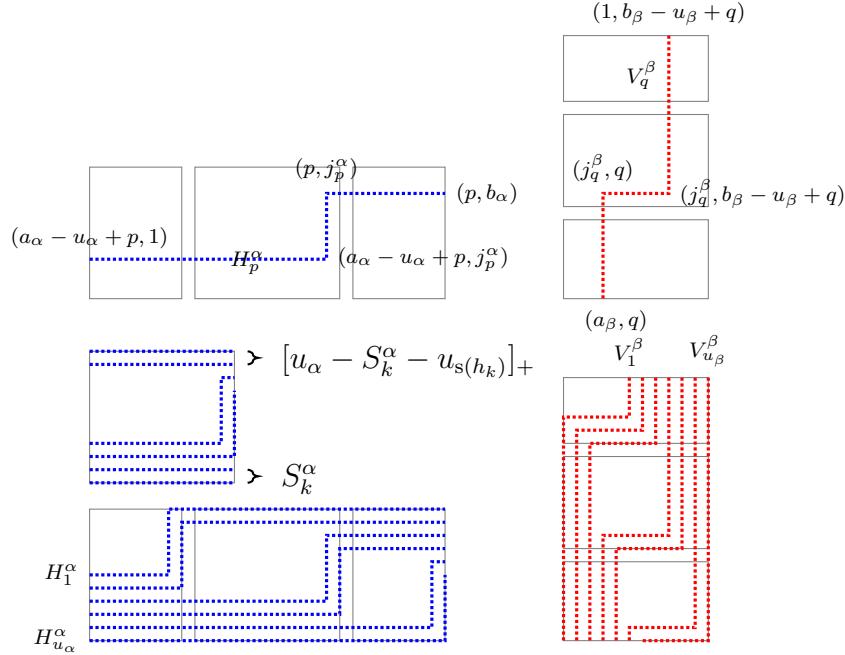

\end{lemma}
\begin{proof}
We first show that the road map $(\{H^\alpha_{p}\}, \{V^\beta_{q}\})$ given in (v) is straight.  Fix an arrow $h_k:\beta\to\alpha$. 
Note that 
if  $u_\alpha-S^\alpha_k\le u_\beta$, then $\cup_p H^\alpha_p$ covers  the bottom $u_\alpha$ rows and the rightmost $[u_\alpha-S^\alpha_k]_+=\min([u_\alpha-S^\alpha_k]_+,u_\beta)$ columns of ${\rm Page}_k$;  
if $u_\alpha-S^\alpha_k > u_\beta$, then $\cup_p H^\alpha_p$ covers the top $u_\alpha -S^\alpha_k - u_\beta$ rows, the bottom $u_\beta+S^\alpha_k$ rows, and the rightmost $u_\beta=\min([u_\alpha-S^\alpha_k]_+,u_\beta)$ columns of ${\rm Page}_k$.
Consequently, $\cup_p H^\alpha_p$ covers the bottom $\min(u_\alpha,u_\beta+S^\alpha_k)$ rows; if the corners of $H^\alpha_p$ are on ${\rm Page}_k$, they must lie in the rightmost $\min([u_\alpha-S^\alpha_k]_+,u_\beta)$ columns. 
%Similarly, 
%$\cup_q V^\beta_q$ covers the following lattice points on ${\rm Page}_k$: 
%
%\noindent -- if $u_\beta\le u_\alpha+S^\beta_k$: the rightmost $u_\beta$ rows and the bottom $[u_\beta-S^\beta_k]_+$ rows;
%
%\noindent -- if $u_\beta > u_\alpha+S^\beta_k$: the leftmost $u_\beta- u_\alpha-S^\beta_k$ columns, the rightmost $u_\alpha+S^\beta_k$ columns, and the bottom $u_\alpha$ rows.
%
%\noindent In particular, if the NW and SE corners of $V^\beta_q$ are in ${\rm Page}_k$, they must lie in the bottom $\min([u_\beta-S^\beta_k]_+,u_\alpha)$ rows of ${\rm Page}_k$. 
Similarly, $\cup_q V^\beta_q$ covers the rightmost  $\min(u_\beta,u_\alpha+S^\beta_k)$ columns; if the corners of $V^\beta_q$ are on ${\rm Page}_k$, they must lie in the bottom  $\min([u_\beta-S^\beta_k]_+,u_\alpha)$ rows. 
To show the corners of $H^\alpha_p$ are contained in vertical paths, it suffices to show that $\min([u_\alpha-S^\alpha_k]_+,u_\beta)\le \min( u_\beta, u_\alpha+S^\beta_k)$, which is obviously true. 
Similarly, any corner of a vertical path lies in a horizontal path. Thus the road map is straight.

\noindent(iv)$\Leftrightarrow$(v). Since a concurrent vertex map uniquely determines the  road map by Proposition \ref{prop:concurrent equiv}, we just need to show that $C$ given in (iv) is induced by the road map given in (v). The proof is easy so we skip.

%If  $u_\alpha\le u_\beta+S^\alpha_k$ and $u_\beta\le u_\alpha+S^\beta_k$: in this case \eqref{eq:Ck equiv conditions} obviously holds.

%If  $u_\alpha> u_\beta+S^\alpha_k$ (then $u_\beta <u_\alpha-S^\alpha_k \le u_\alpha+S^\beta_k$): this belongs to Case 2 in (iv). In this case, $(\cup_p H^\alpha_p)|_k$ is in the top $u_\alpha- u_\beta-S^\alpha_k$ rows, the bottom $u_\beta+S^\alpha_k$ rows, and the rightmost $u_\beta$ columns; 
%$\cup_q V^\beta_q$  is in the rightmost $u_\beta$ rows and the bottom $[u_\beta-S^\beta_k]_+$ rows. Note that $u_\beta+S^\alpha_k\ge [u_\beta-S^\beta_k]_+$. So $(\cup_p H^\alpha_k)|_k\cap (\cup_q V^\beta_q)|_k$ is 
% in the rightmost $u_\beta$ rows and the bottom $[u_\beta-S^\beta_k]_+$ rows. This is exactly (Case 2) in Figure \ref{fig:C_k}, because $\min(u_\alpha,[u_\beta-S_k^\beta]_+)=[u_\beta-S_k^\beta]_+$, $\min(u_\beta,[u_\alpha-S_k^\alpha]_+)=u_\beta$. 

%If  $u_\beta> u_\alpha+S^\beta_k$: the proof is similar to the previous case.

%Therefore  (iv)$\Leftrightarrow$(v). 

\noindent (v)$\Rightarrow$(ii). We shall show that the NW corner $(p,j^\alpha_p)$ of $H^\alpha_p$ is $Q_{p,p+v_\alpha-u_\alpha}$, so is non-essential. 
Assume  $(p,j^\alpha_p)$ is in ${\rm Page}_k$. Then it lies in the $q'$-th column from the right in ${\rm Page}_k$, where
$q'=\sum_{\beta' \stackrel{\le k}{\rightarrow} \alpha}b_{\beta'}-j_p^\alpha+1$. 
Let $\beta=\source(h_k)$ and $V^\beta_q$ be the $q'$-th vertical path in $A_\beta$ counting from the right side.  
So 
$q=u_\beta+1-q'=u_\beta-\sum_{\beta' \stackrel{\le k}{\rightarrow} \alpha}b_{\beta'}+j_p^\alpha
=u_\beta-\sum_{\beta' \stackrel{\le k}{\rightarrow} \alpha}b_{\beta'}+(p-u_\alpha+\sum_{\beta' \stackrel{\le k}{\to} \alpha} b_{\beta'}+  \sum_{\beta' \stackrel{> k}{\to} \alpha} u_{\beta'})
=p-u_\alpha+  \sum_{\beta' \stackrel{\ge k}{\to} \alpha} u_{\beta'}
$.  

We claim that the point $\phi_\beta^{-1}\phi_\alpha(p,j^\alpha_p) = (\sum_{\alpha' \stackrel{< k}{\leftarrow} \beta} a_{\alpha'}+p,b_\beta-q'+1)$ is contained in $V^\beta_q$. 
Indeed, note that in $A_\beta$, the corners of $V^\beta_q$ (whose $x$-coordinate is $j_q^\beta$) are weakly below the horizontal line $x=\sum_{\alpha' \stackrel{< k}{\leftarrow} \beta} a_{\alpha'}+p$ (which contains the point $\phi_\beta^{-1}\phi_\alpha(p,j^\alpha_p)$) because
$j_q^\beta-(\sum_{\alpha' \stackrel{< k}{\leftarrow} \beta} a_{\alpha'}+p)
=(q-u_\beta+\sum_{\alpha' \stackrel{\le k}{\leftarrow} \beta} a_{\alpha'}+  \sum_{\alpha' \stackrel{> k}{\leftarrow} \beta} u_{\alpha'})
-\sum_{\alpha' \stackrel{< k}{\leftarrow} \beta} a_{\alpha'}-p
=q-u_\beta + a_\alpha+ \sum_{\alpha' \stackrel{> k}{\leftarrow} \beta} u_{\alpha'}-p
%&=(p-u_\alpha+  \sum_{\beta' \stackrel{\ge k}{\to} \alpha} u_{\beta'}) -u_\beta + a_\alpha+  \sum_{\alpha' \stackrel{> k}{\leftarrow} \beta} u_{\alpha'}-p\\
=a_\alpha-u_\alpha + \sum_{\beta' \stackrel{> k}{\to} \alpha} u_{\beta'} +  \sum_{\alpha' \stackrel{> k}{\leftarrow} \beta} u_{\alpha'} 
\ge a_\alpha-u_\alpha\ge0. 
$
Since $V^\beta_q$ contains all the points whose $x$-coordinate is $\le j^\beta_q$ and $y$-coordinate is $b_\beta-q'+1$, it contains the point $\phi_\beta^{-1}\phi_\alpha(p,j^\alpha_p)$. this proves the claim, which then implies that $(p,j_p^\alpha)$ is contained in $\phi_\alpha^{-1}(V_q^\beta|_k)$. 

Next, we show that 
\begin{equation}\label{eq:V=V}
V^\alpha_{p+v_\alpha-u_\alpha} = \phi_\alpha^{-1}(V_q^\beta|_k)\textrm{ with }q=p-u_\alpha+  \sum_{\beta' \stackrel{\ge k}{\to} \alpha} u_{\beta'}.
\end{equation}
Indeed, $V_q^\beta|_k$ is the $q$-th vertical path in ${\rm Page}_k$, and there are another
$\sum_{\beta' \stackrel{< k}{\rightarrow} \alpha}u_{\beta'}$ vertical path in the blocks left to ${\rm Page}_k$ in $A_\alpha$, so the index of the vertical path $\phi_\alpha^{-1}(V_q^\beta|_k)$ in $A_\alpha$ is
$q+\sum_{\beta' \stackrel{< k}{\rightarrow} \alpha}u_{\beta'}=
(p-u_\alpha+  \sum_{\beta' \stackrel{\ge k}{\to} \alpha} u_{\beta'})+\sum_{\beta' \stackrel{< k}{\rightarrow} \alpha}u_{\beta'}
=p+v_\alpha-u_\alpha$. 
Therefore, $(p,j_p^\alpha)=Q_{p,p+v_\alpha-u_\alpha}$, thus the NW corner of $H^\alpha_p$ is nonessential. 
Similarly we can show that the NW corner of a vertical path is non-essential. Thus (v)$\Rightarrow$(ii).

\noindent (ii)$\Rightarrow$(v). Let $C$ be a concurrent vertex map without essential NW corners, $(\{H^\alpha_{p}\}, \{V^\beta_{q}\})$ be the corresponding road map. Each $H^\alpha_p$ contains exactly one NW corner:  $Q_{p,p+v_\alpha-u_\alpha}$. So $H^\alpha_p$ consists of only three segments:  two horizontal segments connected by a vertical segment. Each $V^\beta_q$ satisfies a similar condition. 

Now fix $\alpha$ and $p$. Assume  $Q_{p,p+v_\alpha-u_\alpha}$ is in ${\rm Page}_k$.
Then $h_k$ must have target $\alpha$, so we denote $h_k:\beta\to \alpha$. 
By \eqref{eq:V=V}, $V^\alpha_{p+v_\alpha-u_\alpha} = \phi_\alpha^{-1}(V_q^\beta|_k)$ where $p, q$ satisfy $p-q = u_\alpha -  \sum_{\beta' \stackrel{\ge k}{\to} \alpha} u_{\beta'}$. %Note that $k$ is uniquely determined by $\alpha$ and $p$ because of \eqref{eq:k quniquely determined}. 
Note that %$V^\alpha_{p+v_\alpha-u_\alpha}$ goes through $Q_{p,p+v_\alpha-u_\alpha}$ vertically without turn, that is, 
$Q_{p,p+v_\alpha-u_\alpha}$ is not a corner of $V^\alpha_{p+v_\alpha-u_\alpha}$; otherwise 
$Q_{p,p+v_\alpha-u_\alpha}$ must be a SE corner of $V^\alpha_{p+v_\alpha-u_\alpha}$
as well as a NW corner of  $H^\alpha_p$, which is impossible. Consequently, $V^\beta_q$ goes through  $\phi_\beta^{-1}\phi_\alpha(Q_{p,p+v_\alpha-u_\alpha})$ vertically without turn. So the horizontal segment of  $V^\beta_q$ is either strictly below or strictly above
$\phi_\beta^{-1}\phi_\alpha(Q_{p,p+v_\alpha-u_\alpha})$. 
We claim that the latter cannot hold.   
Otherwise, %the NW corner of $V^\beta_q$ is above the horizontal line containing $\phi_\beta^{-1}\phi_\alpha(Q_{p,p+v_\alpha-u_\alpha})$; that is, 
in $A_\beta$, $(P'_{q+v_\beta-u_\beta,q})_y < (\phi^{-1}_\beta\phi_\alpha( Q_{p,p+v_\alpha-u_\alpha}))_y$, so the path $H^\beta_{q+v_\beta-u_\beta}$ is above the path $\phi_\beta^{-1}(H^\alpha_p|_k) = H^\beta_{p+\sum_{\alpha' \stackrel{< k}{\leftarrow} \beta} u_{\alpha'}}$. 
Thus %their indices must satisfy the inequality
$q+v_\beta-u_\beta < p+\sum_{\alpha' \stackrel{< k}{\leftarrow} \beta} u_{\alpha'}.$
Since $v_\beta=\sum_{\alpha' \leftarrow \beta} u_{\alpha'}$, the above inequality implies
$p-q > \sum_{\alpha' \stackrel{\ge k}{\leftarrow} \beta} u_{\alpha'} -u_\beta$, 
thus
$u_\alpha - \sum_{\beta' \stackrel{\ge k}{\to} \alpha} u_{\beta'} = p-q >  \sum_{\alpha' \stackrel{\ge k}{\leftarrow} \beta} u_{\alpha'} -u_\beta$, 
which implies a contradiction that $\sum_{\alpha' \stackrel{ >k }{\leftarrow} \beta} u_{\alpha'} + 
 \sum_{\beta' \stackrel{ > k}{\to} \alpha} u_{\beta'}  
< 0$.

The claim implies that the $y$-coordinate of  $Q_{p,p+v_\alpha-u_\alpha}$ is uniquely determined, since $\phi_\beta^{-1}\phi_\alpha(Q_{p,p+v_\alpha-u_\alpha})$ is on the  vertical line $y=b_\beta-u_\beta+q$. In other words, each $H^\alpha_p$ is uniquely determined. Similarly, each $V^\beta_q$ is uniquely determined. Thus the road map has to be the unique one given by (v). 
%So condition (ii) uniquely determines $(\{H^\alpha_{p}\}, \{V^\beta_{q}\})$, which has to be the one given in (v).
%This proves (ii)$\Rightarrow$(v). Moreover, now we know that (iv) gives the unique $C$ satisfying (ii). 

\noindent (iv)$\Rightarrow$(i). Indeed, by Lemma \ref{lem:larger C'},  $C_{\rm max}(\emptyset)$ has no essential NW corners, so it satisfies (ii). By the proof of (ii)$\Rightarrow$(v), $C_{\rm max}(\emptyset)$  coincides with $C$ defined in (iv).

\noindent (iii)$\Rightarrow$(ii). We show the contrapositive, that is, if $C$ has an essential NW corner, then it is inverse chute movable. Indeed, consider the two cases given in Lemma \ref{lem:larger C'}. 
In Case 1 (see Figure \ref{fig:essentialAB}), the point $\phi_\alpha^{-1}(i',j',k')-(1,1)$ in $C$ can be inverse chute moved to $\phi_\alpha^{-1}(i',j',k')$. 
In Case 2 (see Figure \ref{fig:essentialA}), the point $R_{p+t,q+t}$ in $C$ can be inverse chute moved to $R_{p+t,q+t}+(1,1)$ if $R_{p+t,q+t}\neq Q_{p+t,q+t}$, or to $\phi_\alpha^{-1}(i',j',k')$ if $R_{p+t,q+t} = Q_{p+t,q+t}$. 

\noindent  (i)$\Rightarrow$(iii). Indeed, if $C$ is inverse chute movable, then after a inverse chute move we get $C'>_T C_{\rm max}(\emptyset)$, contradicting the definition of $C_{\rm max}(\emptyset)$.

Now we have proved that (iv)$\Leftrightarrow$(v)$\Leftrightarrow$(ii)$\Leftarrow$(iii)$\Leftarrow$(i)$\Leftarrow$(iv).
This completes the proof of Lemma \ref{lem:initial C}.
\end{proof}

\begin{remark}
(a) We describe the initial concurrent vertex map $C_{\max}(\emptyset)$ for double determinantal ideals  $I^{(r)}_{m,n,u,v}$ for $\mathcal{Q}:2\stackrel{r}{\to}1$ which has $r$ arrows from $2$ to $1$. Assume without loss of generality that $u\ge v$. 
It is easy to check that the region for $C_{\max}(\emptyset)$ given in Lemma \ref{lem:initial C} (iv) becomes Figure \ref{fig:P0}.

\begin{figure}[ht]
\begin{center}
\begin{tikzpicture}[scale=.35]
    \begin{scope}[shift={(2,0)}]
      \draw[draw=black] (0,0) rectangle ++(4,6);
      \draw[fill=gray!30] (2.3,0)--(4,0)--(4,4)--(2.3,4)--(2.3,0);
      \draw[decorate,decoration={brace,mirror,raise=3pt,amplitude=4pt}, thick]
    (2.3,0)--(4,0)  node [midway,yshift=-0.4cm] {${v}$};
      \draw[decorate,decoration={brace,mirror, raise=3pt,amplitude=4pt}, thick]
    (4,0)--(4,4)  node [midway,xshift=0.4cm] {${u}$};
      \node at(1,-3) {Pages $1,\dots,r-\lfloor \frac{{u}}{{v}}\rfloor-1$.};
      \end{scope}
    \begin{scope}[shift={(12,0)}]
      \draw[draw=black] (0,0) rectangle ++(4,6);
      \draw[fill=gray!30] (2.3,0)--(4,0)--(4,6)--(3,6)--(3,4)--(2.3,4)--(2.3,0);
      \draw[dashed] (3,4)--(4,4);
      \draw[decorate,decoration={brace,mirror,raise=3pt,amplitude=4pt}, thick]
    (2.3,0)--(4,0)  node [midway,yshift=-0.4cm] {${v}$};
      \draw[decorate,decoration={brace,mirror, raise=3pt,amplitude=4pt}, thick]
    (4,0)--(4,4)  node [midway,xshift=0.4cm] {${u}$};
      \draw[decorate,decoration={brace,raise=3pt,amplitude=4pt}, thick]
    (3,6)--(4,6)  node [midway,yshift=0.5cm] {$u{\rm \, mod\, }v$};;
      \node at(2,-3) {Page $r-\lfloor \frac{{u}}{{v}}\rfloor$.};
      \end{scope}
    \begin{scope}[shift={(24,0)}]
      \draw[draw=black] (0,0) rectangle ++(4,6);
      \draw[fill=gray!30] (2.3,0)--(4,0)--(4,6)--(2.3,6)--(2.3,0);
      \draw[decorate,decoration={brace,mirror,raise=3pt,amplitude=4pt}, thick]
    (2.3,0)--(4,0)  node [midway,yshift=-0.4cm] {${v}$};
      \node at(2,-3) {Pages $r-\lfloor \frac{{u}}{{v}}\rfloor+1,\dots,r-1$.};
      \end{scope}
    \begin{scope}[shift={(34,0)}]
      \draw[draw=black] (0,0) rectangle ++(4,6);
      \draw[fill=gray!30] (0,0)--(0,1.7)--(2.3,1.7)--(2.3,6)--(4,6)--(4,0)--(0,0);
      \draw[dashed] (2.3,0)--(2.3,1.7)--(4,1.7);
      \draw[decorate,decoration={brace,mirror,raise=3pt,amplitude=4pt}, thick]
    (2.3,0)--(4,0)  node [midway,yshift=-0.4cm] {${v}$};
      \draw[decorate,decoration={brace,mirror, raise=3pt,amplitude=4pt}, thick]
    (4,0)--(4,1.7)  node [midway,xshift=0.4cm] {${v}$};
      \node at(2,-3){Page $r$.};
      \end{scope}
\end{tikzpicture}
\caption{$C_{\max}(\emptyset)$ for a double determinantal ideal (the grey region).}\label{fig:P0}
\end{center}
\end{figure}

(b) In the special case of (a) when $u=v$, Figure \ref{fig:P0} specializes to Figure \ref{fig:P0,u=v}. 
\begin{figure}[ht]
\begin{center}
\begin{tikzpicture}[scale=.35]
    \begin{scope}[shift={(0,0)}]
      \draw[draw=black] (0,0) rectangle ++(4,6);
      \draw[fill=gray!30] (2.3,0)--(4,0)--(4,1.7)--(2.3,1.7)--(2.3,0);
      \draw[decorate,decoration={brace,mirror,raise=3pt,amplitude=4pt}, thick]
    (2.3,0)--(4,0)  node [midway,yshift=-0.4cm] {$v$};
      \draw[decorate,decoration={brace,mirror, raise=3pt,amplitude=4pt}, thick]
    (4,0)--(4,1.7)  node [midway,xshift=0.4cm] {$v$};
      \node at(2,-3) {First $(r-1)$ pages};
      \end{scope}
    \begin{scope}[shift={(10,0)}]
      \draw[draw=black,fill=white] (0,0) rectangle ++(4,6);
      \draw[fill=gray!30] (0,0)--(0,1.7)--(2.3,1.7)--(2.3,6)--(4,6)--(4,0)--(0,0);
      \draw[dashed] (2.3,0)--(2.3,1.7)--(4,1.7);
      \draw[decorate,decoration={brace,mirror,raise=3pt,amplitude=4pt}, thick]
    (2.3,0)--(4,0)  node [midway,yshift=-0.4cm] {$v$};
      \draw[decorate,decoration={brace,mirror, raise=3pt,amplitude=4pt}, thick]
    (4,0)--(4,1.7)  node [midway,xshift=0.4cm] {$v$};
      \node at(2,-3) {Last page};
      \end{scope}
\end{tikzpicture}
\caption{$C_{\max}(\emptyset)$ for a double determinantal ideal in the special case $u=v$.}
\label{fig:P0,u=v}
\end{center}
\end{figure}

\end{remark}

\begin{lemma}\label{lem:essential move vs move}
If $(i,j,k)\in C$ is an essential horizontal (resp.~vertical)  NW corner and $C'$ is obtained as in Lemma \ref{lem:larger C'}, then $C'$ can be obtained from $C$ by a sequence of inverse horizontal (resp.~vertical) chute moves. A similar statement holds for an essential SE corner. 
\end{lemma}
\begin{proof}
We only prove the first statement. Use the notations in the proof of Lemma \ref{lem:larger C'}. Without loss of generality, assume $(i,j,k)$ is an essential horizontal NW corner.  %We construct a sequence of horizontal chute moves to get $C'$ from $C$ as follows.
If $(i,j,k)$ is a vertical NW corner (see Fig \ref{fig:essentialAB}): we obtain $C'$ from $C$ by applying inverse horizontal chute moves to $(i'-1,j'-1,k), (i'-2,j'-2,k), \dots, (i,j,k)$ in order.
Otherwise (see Fig \ref{fig:essentialA}): we obtain $C'$ from $C$ by applying inverse horizontal moves to the following points in order: 
$R_{p+t,q+t}, R_{p+t,q+t}+(0,1),\dots, Q_{p+t,q+t}$; 
$R_{p+t-1,q+t-1}, R_{p+t-1,q+t-1}+(0,1), \dots, Q_{p+t-1,q+t-1}$; 
$\dots$; 
$R_{p+1,q+1}, R_{p+1,q+1}+(0,1),\dots, Q_{p+1,q+1}$; 
$Q_{pq}$. 
\end{proof}

\begin{proposition}\label{prop:PDC to Cmax}
Every concurrent vertex map $C\neq C_{\max}(\emptyset)$ can be obtained from $C_{\max}(\emptyset)$ by a sequence of (horizontal or vertical) chute moves. A similar statement holds for $C_{\min}(\emptyset)$ with chute moves replaced by inverse chute moves. 
\end{proposition}
\begin{proof}
We only prove the first statement. Start with $C$ and iteratively apply inverse chute movables if possible, we will get $C_{\max}(\emptyset)$ by Lemma \ref{lem:initial C}. The proposition is proved by reversing the process.
\end{proof}

\subsection{Vertex-decomposable}
First recall the definition of (pure) vertex-decomposable (see \cite[Def 1.8.5]{Knutson-Miller} or \cite[Definition 16.41]{MS}). Given a vertex $v$ in a simplicial complex $\Delta$, we define  the deletion operator  ${\rm del}_v\Delta=\{G\in\Delta | v\notin G\}$ and 
the link operator  ${\rm link}_v\Delta =\{G\in\Delta| v\notin G, G\cup\{v\}\in\Delta\}$.
(Note that ${\rm del}_v\Delta$ and ${\rm link}_v\Delta$ are undefined if $v$ is not a vertex in $\Delta$.) 
Then $\Delta$ is called vertex-decomposable if $\Delta$ is pure and either (i) $\Delta=\{\emptyset\}$, or (ii) for some vertex $v\in\Delta$, both ${\rm del}_v \Delta$ and ${\rm link}_v\Delta$ are vertex-decomposable.

\begin{proposition}\label{prop:vertex decomposable}
The simplicial complex  $\Delta_{\mathcal{Q},\mm,\uu}$  is vertex-decomposable.
Consequently, $\Delta_{\mathcal{Q},\mm,\uu}$ is pure, shellable, and the bipartite determinantal ideal $I_{\mathcal{Q},\mm,\uu}$ is Cohen-Macaulay.
\end{proposition}
\begin{proof}
Write $L=\{v_1,\dots,v_{|L|}\}$ such that $v_1<_T v_2 <_T \dots <_T v_{|L|}$.
We claim that for any $0\le \ell \le |L|$ and  any subset $S\subseteq\{v_1,\dots,v_\ell\}$, the simplicial complex $\Delta_{\ell,S}:=F_{v_\ell}\cdots F_{v_2}F_{v_1} \Delta_{\mathcal{Q},\mm,\uu}$ is pure if it defined, 
where
$$F_{v_i}
:=
\begin{cases}
&{\rm link}_{v_i}, \textrm{ if $v_i\in S$};\\
&{\rm del}_{v_i}, \textrm{ if $v_i \notin S$}.\\
\end{cases}
$$ 
The claim implies that all the well-defined $\Delta_{\ell,S}$ are vertex-decomposable, in particular (when $\ell=0$) $\Delta_{\mathcal{Q},\mm,\uu}$ is vertex-decomposable. 
Note that $\Delta_{\ell,S}$ is well-defined if and only if $v_i\in F_{v_{i-1}}\cdots F_{v_1}\Delta_{\mathcal{Q},\mm,\uu}$ for $1\le i\le \ell$. 
The idea to prove the claim is that, for each well-defined $\Delta_{\ell,S}$, we construct a facet $C_{\ell,S}$ such that any facet of $\Delta_{\ell,S}$ can be converted into  $C_{\ell,S}$ by repeatedly applying inverse chute moves. Thus all facets of $\Delta_{\ell,S}$ have the same dimension, which implies that 
$\Delta_{\ell,S}$ is pure.

To prove the claim, first note that the faces of $\Delta_{\ell,S}$ have the following easy description: 
\begin{equation}\label{eq:DLS}
\Delta_{\ell,S}=\big\{C_2\subseteq\{v_{\ell+1},\dots,v_{|L|}\} \, \big| \, C_2\cup S \textrm{ is $\uu$-compatible}\big\}.
\end{equation}
Indeed this can be proved by induction on $\ell$. For the base case $\ell=0$, we have $S=\emptyset$, the faces of $\Delta_{\ell,S}=\Delta_{\mathcal{Q},\mm,\uu}$ are exactly those $\uu$-compatible sets. For the inductive step, assume $\ell>0$ and the \eqref{eq:DLS} holds for $\ell-1$. We consider two cases separately. 
If $v_\ell\in S$, then $\Delta_{\ell,S}={\rm link}_{v_\ell}\Delta_{\ell-1,S\setminus\{v_\ell\}}
=\{G  |  v_\ell \notin G, G\cup\{v_\ell\} \in   \Delta_{\ell-1,S\setminus\{v_\ell\}} \}
=\{G  |  v_\ell \notin G, G\cup\{v_\ell\}\subseteq \{v_\ell,\dots,v_{|L|}\}, G\cup\{v_\ell\}\cup (S\setminus\{v_\ell\}) \textrm{ is $\uu$-compatible}
\}
=\{G\subseteq\{v_{\ell+1},\dots,v_{|L|}\}  |  G\cup S \textrm{ is $\uu$-compatible}
\}
$. 
If $v_\ell\notin S$, then 
$\Delta_{\ell,S}={\rm del}_{v_\ell}\Delta_{\ell-1,S}
=\{G \in \Delta_{\ell-1,S} |  v_\ell \notin G \}
=\{G\subseteq \{v_\ell,\dots,v_{|L|}\}  |  v_\ell \notin G,  G\cup S \textrm{ is $\uu$-compatible}
\}
=\{G\subseteq\{v_{\ell+1},\dots,v_{|L|}\}  |  G\cup S \textrm{ is $\uu$-compatible}
\}
$. 
So the inductive step is done in both cases. 

Note that $\Delta_{\ell,S}$, if well-defined, contains $\emptyset$.
Define $C_{\ell,S} := C_{\max}(S)\cap\{v_{\ell+1},\dots,v_{|L|}\},$
which is obviously a face of $\Delta_{\ell,S}$.
Fix a facet $C_2$ of $\Delta_{\ell,S}$. We shall show that
\begin{equation}\label{eq:C2=C_l,S}
|C_2|=|C_{\ell,S}|.
\end{equation}
%As a consequence, all facets of $\Delta_{\ell,S}$ have the same dimension, thus $\Delta_{\ell,S}$ is pure.
Note that 
\begin{equation}\label{eq:C2=C cap}
C_2=C_{\max}(S\cup C_2)\cap\{v_{\ell+1},\dots,v_{|L|}\}
\end{equation} 
Indeed, let $C := C_{\max}(S\cup C_2)$. The inclusion ``$\subseteq$'' is obvious; 
%true because $C_2\subseteq C$ and $C_2\subseteq\{v_{\ell+1},\dots,v_{|L|}\}$;  then
the right side is a face of $\Delta_{\ell,S}$ containing the facet $C_2$, thus the equality must hold. 
Denote $C_1=C \setminus C_2=C\cap\{v_1,\dots,v_\ell\}$. 
So $C$ is the disjoint union $C_1\cup C_2$, and $S\subseteq C_1$.  
%Also note that $S\subseteq C$ and $S\cap C_2=\emptyset$ imply $S\subseteq C_1$. 

To prove \eqref{eq:C2=C_l,S}, we first reduce to the case when $C$ is not inverse chute movable at any point in $C_2$. Indeed, assume that $C\dashrightarrow C'$ is an inverse chute move that moves $P\in C_2$ to $P'\in C'$.  Then $P'\notin C_2$,  $P'>_T P>_T v_\ell$, $C_1\subseteq C\setminus\{P\}\subset C\setminus\{P\}\cup \{P'\}=C'$. 
Define  
\begin{equation}\label{eq:C'2}
C'_2:=C'\setminus C_1=C_2\setminus\{P\}\cup\{P'\}.
\end{equation}
Then $P'\in C'_2\subseteq\{v_{\ell+1},\dots,v_{|L|}\}$, $C'=C_1\cup C'_2$. 
Without loss of generality, assume $C\dashrightarrow C'$ is a horizontal inverse chute move that occurs in $A_\alpha$. 
The corresponding horizontal chute rectangle $R$ has $\phi_\alpha^{-1}(P)$ as its NW vertex and $\phi_\alpha^{-1}(P')$ as its SE vertex.  
Since $C'_2\cup S\subseteq C'$ is $\uu$-compatible,  $C'_2$ is a face of  $\Delta_{\ell,S}$.

\noindent{\bf Claim $\bigstar$}: If $C_2$ is a facet of $\Delta_{\ell,S}$, then $C'_2$  constructed in \eqref{eq:C'2} is also a facet of $\Delta_{\ell,S}$. 

\noindent Proof of Claim  $\bigstar$. Assume $C'_2$ is not a facet. Then there exists a larger face $C''_2=C'_2\cup\{P'' \}\supsetneq C'_2$ of $\Delta_{\ell,S}$. So $S\cup C'_2\cup\{P''\}$ is $\uu$-compatible. 
We consider three cases depending on the position of $P''$.

Case 1: $\phi_\alpha^{-1}(P'')$ lies outside the horizontal chute rectangle $R$. In this case, 
$S\cup C_2'\cup\{P''\}\to S\cup C_2\cup \{P''\}$ is a horizontal chute move. By Lemma \ref{lemma:chute move keep compatibility}, $S\cup C_2\cup\{P''\}$ is $\uu$-compatible, %Thus $C_2\cup\{P''\}$ is a face of $\Delta_{\ell,S}$, 
contradicting the assumption that $C_2$ is a facet of $\Delta_{\ell,S}$.

 Case 2: $\phi_\alpha^{-1}(P'')$ lies in the first row of $R$. We shall show that this case will not occur by proving the statement that $ C_2\cup\{P'\}\ (=C_2'\cup\{P\})$ is a face of $\Delta_{\ell,S}$, which contradicts the assumption that $C_2$ is a facet. %To prove the statement, we need to show that $S\cup C_2\cup\{P'\}$ is $\uu$-compatible.  
The proof is similar to the one of Lemma \ref{lemma:chute move keep compatibility}.
We prove by contradiction by assuming $S\cup C_2\cup\{P'\}(=S\cup C_2'\cup\{P\})$ is not $\uu$-compatible. Then there exists $\gamma\in {\rm V}_\mathcal{Q}$ such that
$D= \{ P_1,\dots,P_{u_\gamma+1}\}\subseteq ( S\cup C_2' \cup\{P\})^\gamma$
is a diagonal chain of size $u_\gamma+1$.
Since $S\cup C_2$ and $S\cup C_2'$ are $\uu$-compatible, 
$\phi_\gamma^{-1}(P')$ and $\phi_\gamma^{-1}(P)$ must both be in $D$.
Assume $\phi_\gamma^{-1} (P')=P_\ell$, and $P'$ is in ${\rm Page}_k$. 
Then $h_k$ must have target $\alpha$ so we denote $h_k:\beta\to\alpha$. 
Note that $\gamma=\alpha$ or $\beta$. 
Write $(i,j)=\phi_\alpha^{-1}(P)$, $(i+1,j+r-1)=\phi_\alpha^{-1}(P')$. 
Define $P''_\ell=\phi_\gamma^{-1}\phi_\alpha(i,j+r-1)$, and if possible, define $P'_\ell=\phi_\gamma^{-1}\phi_\alpha(i+1,j)$, $P'''_\ell = \phi_\gamma^{-1}\phi_\alpha(i,j)$. 
We need to consider two subcases:

--Subcase 1: either $\gamma=\alpha$, or ``$\gamma=\beta$ and $P'''_\ell$ is defined and lies in the same block of $A_\beta$ with $P_\ell$''.
Then $\phi_\gamma^{-1}(P)=P_{\ell-1}$, and $D\setminus\{P_{\ell-1}\} \cup \{\phi_\gamma^{-1}(P'')\}\subseteq (S\cup C'_2\cup\{P''\})^\gamma$ is a diagonal chain of size $u_\gamma+1$ because $\phi_\gamma^{-1}(P'')$ is to the east of $P_{\ell-1}$ and to the NW of $P_\ell$. But that contradicts the assumption that  $S\cup C'_2\cup\{P''\}$ is $\uu$-compatible. 

--Subcase 2: $\gamma=\beta$, and either $P'''_\ell$ is undefined or  $P'''_\ell$  lies in different blocks of $A_\beta$ with $P_\ell$. 
As we have shown in the proof of Lemma \ref{lemma:chute move keep compatibility}, $(P_{\ell-1})_x \neq (P_\ell)_x-1$. Then 
$(P_{\ell-1})_x \le (P_\ell)_x-2$, thus 
 $D\setminus\{P_{\ell}\} \cup \{P''_\ell\}\subseteq (S\cup C_2)^\gamma$ is a diagonal chain of size $u_\gamma+1$, contradicting the assumption that $S\cup C_2$ is $\uu$-compatible.

Case 3: $\phi_\alpha^{-1}(P'')$ lies in the second row of $R$. In this case, let $Q$ be the point determined by $\phi_\alpha^{-1}(Q)=\phi_\alpha^{-1}(P'')-(1,0)$ be the point above $\phi_\gamma^{-1}(P'')$. Then $C_1\cup C_2''=C'\cup\{P''\}\to C'\cup \{P'',Q\}\setminus\{P'\}$ and $C'\cup \{P'',Q\}\setminus\{P'\} \to C'\cup \{P,Q\}\setminus\{P'\}=C\cup\{Q\} $ are both horizontal chute move. By Lemma \ref{lemma:chute move keep compatibility}, $C\cup \{Q\}$ is $\uu$-compatible. Then $C_2\cup\{Q\}$ is a face of $\Delta_{\ell,S}$, contradicting the assumption that $C_2$ is a facet. 
(See Figure \ref{fig: vertex decomposable Case 3.}.)
\begin{figure}[ht]
\begin{center}
\begin{tikzpicture}[scale=.5]
    \begin{scope}[shift={(0,0)}]
      \fill[gray!20](0,0)--(4,0)--(4,1)--(0,1)--(0,0);
      \node at (0,1) {+};\node at (1,1) {+};\node at (2,1) {+};\node at (3,1) {+};\node at (4,1) {.};
      \node at (0,0){.};\node at (1,0) {+};\node at (2,0) {.};\node at (3,0) {+};\node at (4,0) {.};      
      \node at (0, 1.8) {$P$};
      \node at (4, -.7) {$P'$};
      \node at (2,-.7) {$P''$};  
      \node at (2, 1.8) {$Q$};     
      \node at (6,0.5)  {$\longrightarrow$}; 
    \end{scope}

    \begin{scope}[shift={(8,0)}]
      \fill[gray!20](0,0)--(4,0)--(4,1)--(0,1)--(0,0);
      \node at (0,1) {+};\node at (1,1) {+};\node at (2,1) {.};\node at (3,1) {+};\node at (4,1) {.};
      \node at (0,0){.};\node at (1,0) {+};\node at (2,0) {.};\node at (3,0) {+};\node at (4,0) {+};      
      \node at (0, 1.8) {$P$};
      \node at (4, -.7) {$P'$};
      \node at (2,-.7) {$P''$};  
      \node at (2, 1.8) {$Q$};    
      \node at (6,0.5)  {$\longrightarrow$};   
    \end{scope}

    \begin{scope}[shift={(16,0)}]
      \fill[gray!20](0,0)--(4,0)--(4,1)--(0,1)--(0,0);
      \node at (0,1) {.};\node at (1,1) {+};\node at (2,1) {.};\node at (3,1) {+};\node at (4,1) {.};
      \node at (0,0){.};\node at (1,0) {+};\node at (2,0) {+};\node at (3,0) {+};\node at (4,0) {+};      
      \node at (0, 1.8) {$P$};
      \node at (4, -.7) {$P'$};
      \node at (2,-.7) {$P''$};  
      \node at (2, 1.8) {$Q$};      
    \end{scope}

 \end{tikzpicture}
 \end{center}
 \caption{Claim $\bigstar$ Case 3, $C_1\cup C_2''\to C'\cup \{P'',Q\}\setminus\{P'\}\to C\cup\{Q\} $.}
 \label{fig: vertex decomposable Case 3.}
 \end{figure}

This completes the proof of Claim $\bigstar$. 
Thanks for this claim, we can apply inverse chute move repeatedly to $C_2$ until no more inverse chute move can be applied. 
So we only need to prove \eqref{eq:C2=C_l,S} under the assumption that the facet $C_2$ is not inverse chute movable. 
%In fact, we will show that $C_2=C_{\ell,S}$.
%We assert that $C_{\max}(S\cup C_2)=C_{\max}(S)$. Indeed, 
Since $C_2$ is a facet of $\Delta_{\ell,S}$, we have $C_2= C_{\max}(S\cup C_2)\cap\{v_{\ell+1},\dots,v_{|L|}\}$ by \eqref{eq:C2=C cap}. 
Let $T$ be the set of all essential NW corners of $C_{\max}(S\cup C_2)$. 
By Lemma \ref{lem:larger C'}, $T\subseteq S\cup C_2$; on the other hand, $C_2$ contains no inverse movable points, so  by Lemma \ref{lem:essential move vs move}, it contains no essential NW corners. That implies $T\subseteq S$. Now we have a chain of inequalities:
$C_{\max}(S)\ge_T C_{\max}(S\cup C_2)=C_{\max}(T)\ge_T C_{\max}(S)$.
So  $C_{\max}(S\cup C_2)=C_{\max}(S)$.
Therefore $C_2= C_{\max}(S\cup C_2)\cap\{v_{\ell+1},\dots,v_{|L|}\}=C_{\max}(S)\cap\{v_{\ell+1},\dots,v_{|L|}\}=C_{\ell,S}$.  This completes the proof of \eqref{eq:C2=C_l,S}. Therefore $\Delta_{\ell,S}$ is pure, and the proposition follows.
\end{proof}

\begin{corollary}\label{corollary:shelling order}
The increasing order under $<_T$ given in Definition \ref{df:order} (ii) is a shelling order of $\Delta_{\mathcal{Q},\mm,\uu}$.
\end{corollary}
\begin{proof}
We shall give two proofs.

\noindent Proof 1 (as a corollary of Proposition \ref{prop:vertex decomposable}).  From the proof of the previous proposition we get the shelling order by a recursive argument. First shelling 
${\rm del}_{v_1}\Delta_{\mathcal{Q},\mm,\uu}= \{S\in \Delta_{\mathcal{Q},\mm,\uu} | v_1\notin S\}$ 
then shelling 
$\{S\in \Delta_{\mathcal{Q},\mm,\uu} | v_1\in S\}$.
The shelling order of the former is obtained by first shelling
$\{S\in \Delta_{\mathcal{Q},\mm,\uu} | v_1,v_2\notin S\}$
then shelling
$\{S\in \Delta_{\mathcal{Q},\mm,\uu} | v_1\notin S,v_2\in S\}$; 
the shelling order of the latter is obtained by first shelling
$\{S\in \Delta_{\mathcal{Q},\mm,\uu} | v_1\in S,v_2\notin S\}$
then shelling
$\{S\in \Delta_{\mathcal{Q},\mm,\uu} | v_1, v_2\in S\}$.
Following this argument, the shelling order we obtain is the descreasing lex order:
for $C=(P_1,\dots,P_N)$ and $C'=(P'_1,\dots,P'_N)$, where $P_1<_T P_2<_T\cdots<_T P_N$ and  $P'_1<_T P'_2<_T\cdots<_T P'_N$,  $C$ is before $C'$ if there is an integer $1\le t\le N$ such that $P_t >_T P'_t$ and $P_s=P'_s
$ for all $s<t$. 
By symmetry, another valid shelling order is: (following the above notations of $C$ and $C'$) $C$ is before $C'$ if there is an integer $1\le t\le N$ such that $P_t <_T P'_t$ and $P_s=P'_s$ for all $s>t$. This order coincides with the increasing order under $<_T$ given in Definition \ref{df:order} (ii).

\noindent Proof 2 (as a corollary of Proposition \ref{prop:C(S)>=C}). It suffices to show that $\bigcup_{C' <_T C}C'\cap C$ is a union of codimension 1 faces of $C$ for each $C$. 
Alternatively, let $C'<_T C$ and $S=C'\cap C$; it suffices to prove the existence of $C''$ such that $C'' <_T C$, $|C''\cap C|=|C|-1$ and $S\subseteq C''\cap C$. 

First we assert that $S$ does not contain at least one essential SE corner, say $(i,j,k)$, of $C$. Otherwise by Proposition \ref{prop:C(S)>=C}, we have the equality $C_{\min}(S)=C$; then $C_{\min}(S) >_T C'$. But by definition of $C_{\max}(S)$ and the fact that $S\subseteq C'$, we must have $C'\ge_T C_{\min}(S)$, a contraditcion. 

Next, let $C''=C_{\min}(C\setminus\{(i,j,k)\})$. Again by Proposition \ref{prop:C(S)>=C}, $C'' <_T C$. Note that $|C''\cap C|=|C\setminus\{(i,j,k)\}|=|C|-1$, so $C''\cap C$ is a codimension 1 face of $C$ which contains $S$. 
Therefore  $\Delta_{\mathcal{Q},\mm,\uu}$ is shellable.
\end{proof}

\subsection{Ball}
Recall \cite[Proposition 4.7.22]{BLSWZ99} (the readers are also referred to \cite[Lemma 3.6 and Theorem 3.7]{Knutson-Miller2004} to see how it is applied) asserts that, if every codimension 1 face of a shellable simlicial complex $\Delta$ is contained in at most two facets, then $\Delta$ is a topological manifold-with-boundary that is homeomorphic to a ball or a sphere. The facets of the topological boundary are codimension-1 faces contained in exactly one facet. In particular, if $\Delta$ has a nonempty boundary (that is, there exists a codimension-1 face contained in exactly one facet) then it is homeomorphic to a ball. 

\begin{lemma}\label{lem:one contained in two}
 Each codimension-1 face $S$ of  $\Delta_{\mathcal{Q},\mm,\uu}$  is contained in facets $C_{\max}(S)$ and $C_{\min}(S)$, and is not contained in other facets. Consequently, 
 ``$S$ is contained in exactly one facet'' 
 $\Leftrightarrow$ $C_{\max}(S)=C_{\min}(S)$ 
 $\Leftrightarrow$ $S$ contains all the essential SE corners of $C_{\max}(S)$
 $\Leftrightarrow$ $S$ contains all the essential NW corners of $C_{\min}(S)$.
\end{lemma}
\begin{proof}
Assume the contrary that there is a $C$ such that $C_{\min}(S)<_T C<_T C_{\max}(S)$. By Proposition \ref{prop:C(S)>=C},

$C<_T C_{\max}(S)$ implies $S$ does not contain at least one essential NW corner  $P$ of $C$;

$C>_T C_{\min}(S)$ implies $S$ does not contain at least one essential SE corner $Q$ of $C$.

\noindent Note that $P\neq Q$ since $C$ uniquely determines the horizontal and vertical paths, and NW corner and SE corner do not coincide (indeed, since a NW corner in a horizontal path cannot also be a SE corner of a horizontal path, we can assume $P$ is a NW corner of a horizontal path, $Q$ is a SE corner of a vertical path; but this still cannot happen according to the description of straight road map given in Lemma \ref{lem:straight 2 conditions}). Thus $|S|\le |C\setminus\{P,Q\}| = |C|-2$ contradicting the asumption $|S|=N_{\mathcal{Q},\mm,\uu}-1=|C|-1$. This proves the first statement.

The consequence follows from Proposition \ref{prop:C(S)>=C}.
\end{proof}

\begin{proposition}\label{prop:ball}
$\Delta_{\mathcal{Q},\mm,\uu}$ is homeomorphic to a ball of dimension $N_{\mathcal{Q},\mm,\uu}-1$. (Recall that $N_{\mathcal{Q},\mm,\uu}$ is defined in \eqref{eq:cardinality of max C}, and equals to the cardinality of any concurrent vertex map.)
\end{proposition}
\begin{proof}
By Lemma \ref{lem:one contained in two}, each codimension-1 face of  $\Delta_{\mathcal{Q},\mm,\uu}$  is inside at most 2 facets. 
It then suffices to show that some codimension-1 face of  $\Delta_{\mathcal{Q},\mm,\uu}$  is contained in exactly one facet.  To show that, it suffices to show the existence of  a codimension-1 face $S$ inside of the facet $C_{\max}(\emptyset)$ that is not contained in any other facet. 
Assume that such a face does not exist. Then for any $P\in C_{\max}(\emptyset)$, let $S=C_{\max}(\emptyset)\setminus \{P\}$, we have $C_{\min}(S)<_T C_{\max}(\emptyset)$. Then $S$ does not contain at least one essential SE corner of $C_{\max}(\emptyset)$, which implies that $P$ is an essential SE corner of $C_{\max}(\emptyset)$. Therefore, every point in  $C_{\max}(\emptyset)$ is an essential SE corner. 
Let %$P^{(r)}_{\rm SW}$ be the SW-most point and $P^{(r)}_{\rm NE}$ be the NE-most point in the last page; that is, 
$P^{(r)}_{\rm SW}=(m_{\target(h_r)},1,r)$,  $P^{(r)}_{\rm NE}=(1,m_{\source(h_r)},r)$. Then $P^{(r)}_{\rm SW}$, $P^{(r)}_{\rm NE}$ are in $C_{\max}(\emptyset)$ by Lemma \ref{lem:initial C} (see Figure \ref{fig:C_k} Case 2).
We claim that $P^{(r)}_{\rm SW}$ and  $P^{(r)}_{\rm NE}$ cannot both be SE corners of  $C_{\max}(\emptyset)$. Since $P^{(r)}_{\rm SW}$ is at the left-most column of $A_{\source(h_r)}$, it cannot be a SE corner of any vertical path; it is at the bottom row of $A_{\target(h_r)}$, so is in the bottom-most horizontal path which has only one SE corner (the NE-most point of page $r$), this is only possible if page $r$ has one column, that is, $m_{\source(h_k)}=1$. Then  $P^{(r)}_{\rm NE}$ is not a SE corner of any horizontal path (for the same reason as before), and not a SE corner of any vertical path (because $A_{\source(h_r)}$ has only one column, so a vertical path in it has not SE corner). 
So we conclude that $\Delta_{\mathcal{Q},\mm,\uu}$ is always a ball. %Its dimension is obviously $N_{\mathcal{Q},\mm,\uu}-1$. 
\end{proof}

\section{Hilbert series, Multiplicity, and $h$-polynomial}

\begin{corollary}\label{cor:Hilbert series}
The  Hilbert series for ${\bf k}[X]/{\rm init }(I_{\mathcal{Q},\mm,\uu})$, which is defined as the sum of all monomials that lie outside 
${\rm init }(I_{\mathcal{Q},\mm,\uu})$, is 
\begin{equation}\label{eq:Hilb1X}
{\rm Hilb}({\bf k}[X]/{\rm init }(I_{\mathcal{Q},\mm,\uu}); {\bf X}) 
=\sum_C\frac{  \prod_{(i,j,k)\in C}x^{(k)}_{ij} \prod_{(i,j,k)\in L\setminus C} (1-x^{(k)}_{ij}) }{\prod_{(i,j,k)\in L} (1-x^{(k)}_{ij})},
\end{equation} where 
$C$ runs through all (not necessarily maximal) $\uu$-compatible subsets of $L$. (Note that $C$ could be the empty set.)  
In particular, the $\mathbb{Z}$-graded Hilbert series is
\begin{equation}\label{eq:Hilb1t}
{\rm Hilb}({\bf k}[X]/I_{\mathcal{Q},\mm,\uu}; t)
={\rm Hilb}({\bf k}[X]/{\rm init }(I_{\mathcal{Q},\mm,\uu}); t) 
%=\sum_{i\ge0}\dim ({\bf k}[X]/I_{\mathcal{Q},\mm,\uu})_i\; t^i
=\sum_C \Big(\frac{  t}{1-t}\Big)^{|C|} ,
\end{equation} where 
$C$ runs through all (not necessarily maximal) $\uu$-compatible subsets of $L$;
the multiplicity of the bipartite determinantal ideal is
$${\rm mult} ({\bf k}[X]/I_{\mathcal{Q},\mm,\uu})=\textrm{ the number of concurrent vertex maps.}$$
Moreover, we have
\begin{equation}\label{eq:Hilb2X}
{\rm Hilb}({\bf k}[X]/{\rm init }(I_{\mathcal{Q},\mm,\uu}); {\bf X}) 
=\sum_C\frac{  (-1)^{N_{\mathcal{Q},\mm,\uu}-|C|} \prod_{(i,j,k)\in L\setminus C} (1-x^{(k)}_{ij}) }{\prod_{(i,j,k)\in L} (1-x^{(k)}_{ij})},
\end{equation} 
where $C$ runs through all interior faces. In particular,
\begin{equation}\label{eq:Hilb2t}
{\rm Hilb}({\bf k}[X]/I_{\mathcal{Q},\mm,\uu}; t) 
=\sum_C\frac{  (-1)^{N_{\mathcal{Q},\mm,\uu}-|C|} }{(1-t)^{|C|}},
\end{equation} 
where $C$ runs through all interior faces.
\end{corollary}

\begin{proof}
The only nontrivial statement is \eqref{eq:Hilb2X}, its proof is similar to \cite[Theorem 4.1]{Knutson-Miller2004}. A key ingredient is that  $\Delta_{\mathcal{Q},\mm,\uu}$ is homeomorphic to a ball, which we proved in Proposition \ref{prop:ball}.
\end{proof}

The Hilbert series of the bipartite determinantal ideal $I_{\mathcal{Q},\mm,\uu}$ can be written as 
$$
{\rm Hilb}({\bf k}[X]/I_{\mathcal{Q},\mm,\uu}; t)
=\frac{H(t)}{(1-t)^{N_{\mathcal{Q},\mm,\uu}}}
$$
where  $H(t)=h_0+h_1t+\cdots$ is called the $h$-polynomial and $(h_0,h_1,\dots)$ is called the $h$-vector of $\Delta_{\mathcal{Q},\mm,\uu}$. Since ${\bf k}[X]/I_{\mathcal{Q},\mm,\uu}$ is Cohen-Mcaulay, its $h$-polynomial must be positive, that is, be contained in $\mathbb{Z}_{\ge0}[t]$ (see \cite[Chapter II \S3]{Stanley}). Such polynomials are studied for local rings of Schubert varieties in \cite{LY2}.
It is natural to look for a combinatorially manifestly positive formula for $H(t)$. This can be answered as follows:

\begin{corollary}\label{cor:h-poly}
The $h$-polynomial $h_0+h_1t+\cdots$ of $\Delta_{\mathcal{Q},\mm,\uu}$ is determined as follows: for $i=0,1,\dots, N_{\mathcal{Q},\mm,\uu}$, 
$$\aligned
h_i
&=\textrm{the number of concurrent vertex maps that contain exactly $i$ essential SE corners}\\
&=\textrm{the number of concurrent vertex maps that contain exactly $i$ essential NW corners}. 
\endaligned
$$
\end{corollary}
\begin{proof}
Arrange the facets of  $\Delta_{\mathcal{Q},\mm,\uu}$ as $C_1,C_2,\dots, C_N$        sasdby the increasing order under $<_T$. 
By Corollary \ref{corollary:shelling order}, this is a shelling order of $\Delta_{\mathcal{Q},\mm,\uu}$.
By a result of McMullen and Walkup (see \cite[Corollary 5.1.14]{BH}), 
$h_i=|\{j: r_j=i\}|$ where $r_1=0$ and $r_j$ is the number of facets of $\langle C_j\rangle \cap \langle C_1,\dots,C_{j-1}\rangle$ for $2\le j\le N$.  
Note that $S$ is a facet of $\langle C_j\rangle \cap \langle C_1,\dots,C_{j-1}\rangle$ if and only if $|S|=N_{\mathcal{Q},\mm,\uu}-1$ and $S=F_j\cap F_i$ for some $i<j$. Let $S$ be such a facet. By Lemma \ref{lem:larger C'} and \ref{lem:one contained in two}, $C_i=C_{\min}(S)=S\cup\{P\}$ and $C_j=C_{\max}(S)=S\cup\{P'\}$, where $P$ is an essential NW corner of $C_i$, $P'$ is an essential SE corner of $C_j$. Note that $P$ and $P'$ uniquely determine each other. So $S$ is a desired facet if and only if $S$ is obtained from $C_j$ by removing an essential SE corner of $C_j$. Therefore $r_j$ is the number of essential SE corners of $C_j$, and thus the first equality holds. The second equality holds by symmetry. 
\end{proof}
It is relatively fast for a computer to generate all concurrent vertex maps, in particular to compute the multiplicity, by applying chute moves iteratively to $C_{\max}(\emptyset)$.
It is slow to compute the Hilbert series using \eqref{eq:Hilb1t} or \eqref{eq:Hilb2t}. Instead, it is relatively fast to compute the  Hilbert series using Corollary \ref{cor:h-poly}.
For the readers convenience, the related Sagemath codes used for study the concurrent vertex maps are available at 

{\tt https://sites.google.com/oakland.edu/bipartite-determinantal}

\section{Examples, algorithms, and beyond}
\subsection{Some examples}
\begin{example}\label{eg:double det}
Consider a determinantal ideal $I^{\rm det}_{m,n,u+1}$ of $(u+1)$-minors in an $m\times n$ matrix. It is bipartite determinantal ideal for $\mathcal{Q}:2\to 1$, ${\bf m}=(m_1,m_2)=(m,n)$, ${\bf u}=(u_1,u_2)=(u,u)$. Each concurrent vertex map $C$ consists of lattice points in the union of $u$ nonintersecting paths $(m-u+i,i)\rightsquigarrow(i,n-u+i)$ for $1\le i\le u$ (see Figure \ref{fig:lemma contained in path} Left) and points in the two triangles at the SW and NE corners: 
$$\{(i,j)\in[1,m]\times[1,n] \ | \ j\le i-(m-u) \textrm { or } i\le j-(n-u)\}.$$ 
The corresponding road map $(\{H^1_i\}, \{V^2_i\})_{1\le i\le u}$ can be obtained as follows:

$H^1_i$ is obtained by joining $(m-u+i,1), (m-u+i,i)\rightsquigarrow(i,n-u+i), (i,n)$;% (see Figure \ref{fig:lemma contained in path} Right);

$V^2_i$ is obtained by joining $(1,n-u+1), (i,n-u+i)\rightsquigarrow (m-u+i,i), (m,i)$.

So in this case, we can view $C$ as lattice points in the union of $H^1_i$, or view as lattice points in the union of $V^2_i$. 
See Figure \ref{fig:det} for the concurrent vertex maps and the corresponding road maps when $m=3,n=3,u=2$.

\begin{figure}[ht]
\begin{center}
\begin{tikzpicture}[scale=.8]

    \begin{scope}[shift={(0,0)}]
      \draw[draw=black] (-.5,-.5) rectangle ++(3,3);
      \node at (0,2) {+};
      \node at (1,2) {.};
      \node at (2,2) {.};
      \node at (0,1) {.};
      \node at (1,1) {.};
      \node at (2,1) {.};
      \node at (0,0) {.};
      \node at (1,0) {.};
      \node at (2,0) {.};
      \end{scope}

    \begin{scope}[shift={(5,0)}]
      \draw[draw=black] (-.5,-.5) rectangle ++(3,3);
      \node at (0,2) {.};
      \node at (1,2) {.};
      \node at (2,2) {.};
      \node at (0,1) {.};
      \node at (1,1) {+};
      \node at (2,1) {.};
      \node at (0,0) {.};
      \node at (1,0) {.};
      \node at (2,0) {.};
      \end{scope}

    \begin{scope}[shift={(10,0)}]
      \draw[draw=black] (-.5,-.5) rectangle ++(3,3);
      \node at (0,2) {.};
      \node at (1,2) {.};
      \node at (2,2) {.};
      \node at (0,1) {.};
      \node at (1,1) {.};
      \node at (2,1) {.};
      \node at (0,0) {.};
      \node at (1,0) {.};
      \node at (2,0) {+};
      \end{scope}

\begin{scope}[shift={(0,-4)}]
\draw  [gray] (0,0) grid (2,2);
\begin{scope}[shift={(.07,-.07)}]
\draw [blue,densely dotted,very thick](-.2,1)--(1,1)--(1,2)--(2,2);
\draw [blue,densely dotted,very thick](-.2,0)--(2,0)--(2,1);

\draw (0,1) node [left] {\tiny $H^1_1$};
\draw (0,0) node [left] {\tiny $H^1_2$};
\end{scope}

\begin{scope}[shift={(-.07,.07)}]
\draw [red,densely dotted,very thick](1,2)--(1,1)--(0,1)--(0,0);
\draw [red,densely dotted,very thick](2,2)--(2,0)--(1,0);

\draw (1,2) node [above] {\tiny $V^2_1$};
\draw (2,2) node [above] {\tiny $V^2_2$};
\end{scope}
\end{scope}

\begin{scope}[shift={(5,-4)}]
\draw  [gray] (0,0) grid (2,2);
\begin{scope}[shift={(.07,-.07)}]
\draw [blue,densely dotted,very thick](0,1)--(0,2)--(2,2);
\draw [blue,densely dotted,very thick](0,0)--(2,0)--(2,1);

\draw (0,1) node [left] {\tiny $H^1_1$};
\draw (0,0) node [left] {\tiny $H^1_2$};
\end{scope}

\begin{scope}[shift={(-.07,.07)}]
\draw [red,densely dotted,very thick](1,2)--(0,2)--(0,0);
\draw [red,densely dotted,very thick](2,2)--(2,0)--(1,0);

\draw (1,2) node [above] {\tiny $V^2_1$};
\draw (2,2) node [above] {\tiny $V^2_2$};
\end{scope}
\end{scope}

\begin{scope}[shift={(10,-4)}]
\draw  [gray] (0,0) grid (2,2);
\begin{scope}[shift={(.07,-.07)}]
\draw [blue,densely dotted,very thick](0,1)--(0,2)--(2,2);
\draw [blue,densely dotted,very thick](0,0)--(1,0)--(1,1)--(2,1);

\draw (0,1) node [left] {\tiny $H^1_1$};
\draw (0,0) node [left] {\tiny $H^1_2$};
\end{scope}

\begin{scope}[shift={(-.07,.07)}]
\draw [red,densely dotted,very thick](1,2)--(0,2)--(0,0);
\draw [red,densely dotted,very thick](2,2)--(2,1)--(1,1)--(1,0);

\draw (1,2) node [above] {\tiny $V^2_1$};
\draw (2,2) node [above] {\tiny $V^2_2$};
\end{scope}
\end{scope}

\end{tikzpicture}       
\end{center}
\caption{Example \ref{eg:double det}.}
\label{fig:det}
\end{figure}

\end{example}

\begin{example}\label{eg:double det r=2, m=3, n=2, u=v=1}
Consider the double determinantal ideal $I^{(r)}_{m,n,u,v}$ where $r=2, m=3, n=2, u=v=1$. Denote the corresponding quiver to be $\mathcal{Q}:2\stackrel{r}{\to}1$ with $r$ arrows.
The ideal is generated by 2 minors in the following two matrices:
$$A_1=\begin{bmatrix}
x^{(1)}_{11}&x^{(1)}_{12}&x^{(2)}_{11}&x^{(2)}_{12}\\
x^{(1)}_{21}&x^{(1)}_{22}&x^{(2)}_{21}&x^{(2)}_{22}\\
x^{(1)}_{31}&x^{(1)}_{32}&x^{(2)}_{31}&x^{(2)}_{32}\\
\end{bmatrix},
\quad
A_2=\begin{bmatrix}
x^{(1)}_{11}&x^{(1)}_{12}\\
x^{(1)}_{21}&x^{(1)}_{22}\\
x^{(1)}_{31}&x^{(1)}_{32}\\
x^{(2)}_{11}&x^{(2)}_{12}\\
x^{(2)}_{21}&x^{(2)}_{22}\\
x^{(2)}_{31}&x^{(2)}_{32}\\
\end{bmatrix}.
$$
The facets of its Stanley-Reisner complex $\Delta$ are illustrated in Figure \ref{fig:double det}, where arrows denote all possible chute moves, red ``$\color{red}+$''s indicate where the chute move can apply, the numbers indicate the shelling order given by the decreasing order under $<_T$. 
\begin{figure}[ht]
\begin{center}
\begin{tikzpicture}[scale=.35]
    \begin{scope}[shift={(0,0)}]
      \node at (-3,1) {\tiny $C_{\max}(\emptyset)=$};
      \draw[draw=black] (-.5,-.5) rectangle ++(2,3);
      \draw[draw=black] (1.5,-.5) rectangle ++(2,3);
      \node at (0,0) {+};
      \node at (1,0) {.};
      \node at (2,0) {.};
      \node at (3,0) {.};
      \node at (0,1) {+};
      \node at (1,1) {+};
      \node [red] at (2,1) {+};
      \node at (3,1) {.};
      \node at (0,2) {+};
      \node at (1,2) {+};
      \node at (2,2) {+};
      \node at (3,2) {.};
      \draw[->,>=stealth] (4,1) to (5,1);
      \node at (1.5,-1) {\tiny $1$};

      \end{scope}
    \begin{scope}[shift={(6,0)}]
      \draw[draw=black] (-.5,-.5) rectangle ++(2,3);
      \draw[draw=black] (1.5,-.5) rectangle ++(2,3);
      \node at (0,0) {+};
      \node at (1,0) {.};
      \node at (2,0) {.};
      \node at (3,0) {+};
      \node at (0,1) {+};
      \node [red] at (1,1) {+};
      \node at (2,1) {.};
      \node at (3,1) {.};
      \node at (0,2) {+};
      \node at (1,2) {+};
      \node [red] at (2,2) {+};
      \node at (3,2) {.};
      \draw[->,>=stealth] (4,1) to (5,1);
      \draw[->,>=stealth] (-1,-1) to (-2,-2);
      \node at (1.5,-1) {\tiny $2$};
      \end{scope}
    \begin{scope}[shift={(12,0)}]
      \draw[draw=black] (-.5,-.5) rectangle ++(2,3);
      \draw[draw=black] (1.5,-.5) rectangle ++(2,3);
      \node [red] at (0,0) {+};
      \node at (1,0) {.};
      \node at (2,0) {.};
      \node at (3,0) {+};
      \node at (0,1) {+};
      \node [red] at (1,1) {+};
      \node at (2,1) {.};
      \node at (3,1) {+};
      \node at (0,2) {+};
      \node at (1,2) {+};
      \node at (2,2) {.};
      \node at (3,2) {.};
      \draw[->,>=stealth] (4,1) to (5,1);
      \draw[->,>=stealth] (-1,-1) to (-2,-2);
      \node at (1.5,-1) {\tiny $3$};
      \end{scope}
    \begin{scope}[shift={(18,0)}]
      \draw[draw=black] (-.5,-.5) rectangle ++(2,3);
      \draw[draw=black] (1.5,-.5) rectangle ++(2,3);
      \node at (0,0) {.};
      \node at (1,0) {.};
      \node at (2,0) {.};
      \node at (3,0) {+};
      \node at (0,1) {+};
      \node [red] at (1,1) {+};
      \node at (2,1) {.};
      \node at (3,1) {+};
      \node at (0,2) {+};
      \node at (1,2) {+};
      \node at (2,2) {.};
      \node at (3,2) {+};
      \draw[->,>=stealth] (4,1) to (5,1);
      \node at (1.5,-1) {\tiny $4$};
      \end{scope}
    \begin{scope}[shift={(24,0)}]
      \draw[draw=black] (-.5,-.5) rectangle ++(2,3);
      \draw[draw=black] (1.5,-.5) rectangle ++(2,3);
      \node at (0,0) {.};
      \node at (1,0) {.};
      \node at (2,0) {+};
      \node at (3,0) {+};
      \node [red] at (0,1) {+};
      \node at (1,1) {.};
      \node at (2,1) {.};
      \node at (3,1) {+};
      \node at (0,2) {+};
      \node [red] at (1,2) {+};
      \node at (2,2) {.};
      \node at (3,2) {+};
      \draw[->,>=stealth] (4,1) to (5,1);
      \draw[->,>=stealth] (-1,-1) to (-2,-2);
      \node at (1.5,-1) {\tiny $7$};
      \end{scope}
    \begin{scope}[shift={(30,0)}]
      \draw[draw=black] (-.5,-.5) rectangle ++(2,3);
      \draw[draw=black] (1.5,-.5) rectangle ++(2,3);
      \node at (0,0) {.};
      \node at (1,0) {+};
      \node at (2,0) {+};
      \node at (3,0) {+};
      \node at (0,1) {.};
      \node at (1,1) {.};
      \node at (2,1) {.};
      \node at (3,1) {+};
      \node at (0,2) {+};
      \node [red] at (1,2) {+};
      \node at (2,2) {.};
      \node at (3,2) {+};
      \draw[->,>=stealth] (-1,-1) to (-2,-2);
      \node at (1.5,-1) {\tiny $8$};
      \end{scope}
     %
     %the second row 
    \begin{scope}[shift={(0,-5)}]
      \draw[draw=black] (-.5,-.5) rectangle ++(2,3);
      \draw[draw=black] (1.5,-.5) rectangle ++(2,3);
      \node at (0,0) {+};
      \node at (1,0) {.};
      \node at (2,0) {+};
      \node at (3,0) {+};
      \node at (0,1) {+};
      \node at (1,1) {.};
      \node at (2,1) {.};
      \node at (3,1) {.};
      \node at (0,2) {+};
      \node at (1,2) {+};
      \node [red] at (2,2) {+};
      \node at (3,2) {.};
      \draw[->,>=stealth] (4,1) to (5,1);
      \node at (1.5,-1) {\tiny $5$};
      \end{scope}
    \begin{scope}[shift={(6,-5)}]
      \draw[draw=black] (-.5,-.5) rectangle ++(2,3);
      \draw[draw=black] (1.5,-.5) rectangle ++(2,3);
      \node [red] at (0,0) {+};
      \node at (1,0) {.};
      \node at (2,0) {+};
      \node at (3,0) {+};
      \node at (0,1) {+};
      \node at (1,1) {.};
      \node at (2,1) {.};
      \node at (3,1) {+};
      \node at (0,2) {+};
      \node [red] at (1,2) {+};
      \node at (2,2) {.};
      \node at (3,2) {.};
      \draw[->,>=stealth] (4,1) to (5,1);
      \draw[->,>=stealth] (4,2.5) to (17,4.5);      
      \node at (1.5,-1) {\tiny $6$};
      \end{scope}
    \begin{scope}[shift={(12,-5)}]
      \draw[draw=black] (-.5,-.5) rectangle ++(2,3);
      \draw[draw=black] (1.5,-.5) rectangle ++(2,3);
      \node [red] at (0,0) {+};
      \node at (1,0) {.};
      \node at (2,0) {+};
      \node at (3,0) {+};
      \node at (0,1) {+};
      \node at (1,1) {.};
      \node at (2,1) {+};
      \node at (3,1) {+};
      \node at (0,2) {+};
      \node at (1,2) {.};
      \node at (2,2) {.};
      \node at (3,2) {.};
      \draw[ ->,>=stealth] (4,1) to (5,1);
      \node at (1.5,-1) {\tiny $9$};
      \end{scope}
    \begin{scope}[shift={(18,-5)}]
      \draw[draw=black] (-.5,-.5) rectangle ++(2,3);
      \draw[draw=black] (1.5,-.5) rectangle ++(2,3);
      \node at (0,0) {.};
      \node at (1,0) {.};
      \node at (2,0) {+};
      \node at (3,0) {+};
      \node [red] at (0,1) {+};
      \node at (1,1) {.};
      \node at (2,1) {+};
      \node at (3,1) {+};
      \node at (0,2) {+};
      \node at (1,2) {.};
      \node at (2,2) {.};
      \node at (3,2) {+};
      \draw[->,>=stealth] (4,1) to (5,1);
      \node at (1.5,-1) {\tiny $10$};
      \end{scope}
    \begin{scope}[shift={(24,-5)}]
      \draw[draw=black] (-.5,-.5) rectangle ++(2,3);
      \draw[draw=black] (1.5,-.5) rectangle ++(2,3);
      \node at (0,0) {.};
      \node at (1,0) {+};
      \node at (2,0) {+};
      \node at (3,0) {+};
      \node at (0,1) {.};
      \node at (1,1) {.};
      \node at (2,1) {+};
      \node at (3,1) {+};
      \node [red] at (0,2) {+};
      \node at (1,2) {.};
      \node at (2,2) {.};
      \node at (3,2) {+};
      \draw[->,>=stealth] (4,1) to (5,1);
      \node at (1.5,-1) {\tiny $11$};
      \end{scope}
    \begin{scope}[shift={(30,-5)}]
      \draw[draw=black] (-.5,-.5) rectangle ++(2,3);
      \draw[draw=black] (1.5,-.5) rectangle ++(2,3);
      \node at (0,0) {.};
      \node at (1,0) {+};
      \node at (2,0) {+};
      \node at (3,0) {+};
      \node at (0,1) {.};
      \node at (1,1) {+};
      \node at (2,1) {+};
      \node at (3,1) {+};
      \node at (0,2) {.};
      \node at (1,2) {.};
      \node at (2,2) {.};
      \node at (3,2) {+};
      \node at (6,1) {\tiny $=C_{\min}(\emptyset)$};
      \node at (1.5,-1) {\tiny $12$};
      \end{scope}

 \end{tikzpicture}
\end{center}
\caption{Example \ref{eg:double det r=2, m=3, n=2, u=v=1}.}
\label{fig:double det}
\end{figure}
So $N_{\mathcal{Q},\mm,\uu} = \textrm{the number of dots} = 5$,
$\dim \Delta =N_{\mathcal{Q},\mm,\uu}-1=4$,  
multiplicity $=$ the number of facets of $\Delta=$ the number of concurrent vertex maps $=12$.
The numbers of $d$-faces in $\Delta$ are easily counted by a computer:
\begin{center}
\begin{tabular}{|c| c c c c c c | c|} 
 \hline
$d$ & $-1$ & $0$ & $1$ & $2$ & $3$ & $4$ & total\\
 \hline
$\#$ of $d$-faces & $1$ & $12$ & $42$ & $64$ & $45$ & $12$ & $176$ \\ 
 \hline
$\#$ of interior $d$-faces & $0$ & $0$ & $0$ & $4$ & $15$ & $12$ & $31$ \\ 
 \hline
\end{tabular}
\end{center}
Consequently, by \eqref{eq:Hilb1t},
$$\aligned
&{\rm Hilb}({\bf k}[X]/I_{\mathcal{Q},\mm,\uu}; t)
=\sum_{\textrm{faces } C} \Big(\frac{  t}{1-t}\Big)^{|C|} 
=\sum_{d} (\textrm{\# of $d$-faces}) \Big(\frac{  t}{1-t}\Big)^{d+1}\\
&=
1\Big(\frac{  t}{1-t}\Big)^0
+12\Big(\frac{  t}{1-t}\Big)^1
+42\Big(\frac{  t}{1-t}\Big)^2
+64\Big(\frac{  t}{1-t}\Big)^3
+45\Big(\frac{  t}{1-t}\Big)^4
+12\Big(\frac{  t}{1-t}\Big)^5
\\
&=\frac{1+7t+4t^2}{(1-t)^5}.
\endaligned
$$ 
Alternatively, by \eqref{eq:Hilb2t},
 $$\aligned
{\rm Hilb}({\bf k}[X]/I_{\mathcal{Q},\mm,\uu}; t)
&=\sum_{\textrm{interior faces } C} \frac{  (-1)^{N_{\mathcal{Q},\mm,\uu}-|C|} }{(1-t)^{|C|}}
=\sum_{d} (\textrm{\# of interior $d$-faces}) \frac{  (-1)^{5-(d+1)} }{(1-t)^{d+1}}\\
&=
4\frac{1}{(1-t)^3}
+15\frac{-1}{(1-t)^4}
+12\frac{1}{(1-t)^5}
=\frac{1+7t+4t^2}{(1-t)^5}.
\endaligned
$$ 
Alternatively, by Corollary \ref{cor:h-poly}, the number of essential SE corners of the concurrent vertex maps in Figure \ref{fig:double det} are $1, 2, 2, 1, 2, 1; 1, 2, 1, 1, 1, 0$, respectively; the number of essential NW corners of the concurrent vertex maps in Figure \ref{fig:double det} are $0, 1, 1, 1, 2, 1; 1, 2, 1, 2, 2, 1$, respectively. 
So the $h$-polynomial is $1+7t+4t^2$. 
\end{example}

\begin{example}\label{eg:bipartite 234 to 1}
Continue Example \ref{eg:234to1}. 
\begin{figure}[ht]
\begin{center}
\begin{tikzpicture}[scale=.3]
    \begin{scope}[shift={(0,0)}] %1
      \node at (-3.3,1) {\tiny $C_{\max}(\emptyset)=$};
      \draw[draw=black] (-.5,-.5) rectangle ++(2,3);
      \draw[draw=black] (1.5,-.5) rectangle ++(2,3);
      \draw[draw=black] (3.5,-.5) rectangle ++(2,3);
      \node at (0,0) {.};
      \node at (1,0) {.};
      \node at (2,0) {.};
      \node at (3,0) {.};
      \node at (4,0) {.};
      \node at (5,0) {.};
      \node [red] at (0,1) {+};
      \node at (1,1) {.};
      \node [red] at (2,1) {+};
      \node at (3,1) {.};
      \node [red] at (4,1) {+};
      \node at (5,1) {.};
      \node at (0,2) {+};
      \node [red] at (1,2) {+};
      \node at (2,2) {+};
      \node at (3,2) {.};
      \node at (4,2) {+};
      \node  at (5,2) {.};
      \draw[->,>=stealth] (-1,-1) to (-7,-2);
      \draw[->,>=stealth] (1,-1) to (-1,-2);
      \draw[->,>=stealth] (3,-1) to (5,-2);
      \draw[->,>=stealth] (5,-1) to (11,-2);
      \end{scope}
    \begin{scope}[shift={(-12,-5)}] %2
      \draw[draw=black] (-.5,-.5) rectangle ++(2,3);
      \draw[draw=black] (1.5,-.5) rectangle ++(2,3);
      \draw[draw=black] (3.5,-.5) rectangle ++(2,3);
      \node at (0,0) {.};
      \node at (1,0) {.};
      \node at (2,0) {.};
      \node at (3,0) {.};
      \node at (4,0) {.};
      \node at (5,0) {.};
      \node[red] at (0,1) {+};
      \node at (1,1) {.};
      \node at (2,1) {+};
      \node at (3,1) {+};
      \node[red] at (4,1) {+};
      \node at (5,1) {.};
      \node at (0,2) {+};
      \node at (1,2) {.};
      \node[red] at (2,2) {+};
      \node at (3,2) {.};
      \node at (4,2) {+};
      \node at (5,2) {.};
      \draw[->,>=stealth] (0,-1) to (-9,-4);
      \draw[->,>=stealth] (2,-1) to (-2,-4);
      \draw[->,>=stealth] (4,-1) to (4,-4);
      \end{scope}
    \begin{scope}[shift={(-4,-5)}] %3
      \draw[draw=black] (-.5,-.5) rectangle ++(2,3);
      \draw[draw=black] (1.5,-.5) rectangle ++(2,3);
      \draw[draw=black] (3.5,-.5) rectangle ++(2,3);
      \node at (0,0) {.};
      \node at (1,0) {+};
      \node at (2,0) {.};
      \node at (3,0) {.};
      \node at (4,0) {.};
      \node at (5,0) {.};
      \node at (0,1) {.};
      \node at (1,1) {.};
      \node[red] at (2,1) {+};
      \node at (3,1) {.};
      \node[red] at (4,1) {+};
      \node at (5,1) {.};
      \node at (0,2) {+};
      \node[red] at (1,2) {+};
      \node at (2,2) {+};
      \node at (3,2) {.};
      \node at (4,2) {+};
      \node at (5,2) {.};
      \draw[->,>=stealth] (0,-1) to (-16,-4);
      \draw[->,>=stealth] (2,-1) to (2,-4);
      \draw[->,>=stealth] (4,-1) to (8,-4);
      \end{scope}
    \begin{scope}[shift={(4,-5)}] %4
      \draw[draw=black] (-.5,-.5) rectangle ++(2,3);
      \draw[draw=black] (1.5,-.5) rectangle ++(2,3);
      \draw[draw=black] (3.5,-.5) rectangle ++(2,3);
      \node at (0,0) {.};
      \node at (1,0) {.};
      \node at (2,0) {.};
      \node at (3,0) {+};
      \node at (4,0) {.};
      \node at (5,0) {.};
      \node[red] at (0,1) {+};
      \node at (1,1) {.};
      \node at (2,1) {.};
      \node at (3,1) {.};
      \node[red] at (4,1) {+};
      \node at (5,1) {.};
      \node at (0,2) {+};
      \node at (1,2) {+};
      \node[red] at (2,2) {+};
      \node at (3,2) {.};
      \node at (4,2) {+};
      \node at (5,2) {.};
      \draw[->,>=stealth] (0,-1) to (-3,-4);
      \draw[->,>=stealth] (2,-1) to (7,-4);
      \draw[->,>=stealth] (4,-1) to (14,-4);
      \end{scope}
    \begin{scope}[shift={(12,-5)}] %5
      \draw[draw=black] (-.5,-.5) rectangle ++(2,3);
      \draw[draw=black] (1.5,-.5) rectangle ++(2,3);
      \draw[draw=black] (3.5,-.5) rectangle ++(2,3);
      \node at (0,0) {.};
      \node at (1,0) {.};
      \node at (2,0) {.};
      \node at (3,0) {.};
      \node at (4,0) {.};
      \node at (5,0) {+};
      \node[red] at (0,1) {+};
      \node at (1,1) {.};
      \node[red] at (2,1) {+};
      \node at (3,1) {.};
      \node at (4,1) {.};
      \node at (5,1) {.};
      \node at (0,2) {+};
      \node[red] at (1,2) {+};
      \node at (2,2) {+};
      \node at (3,2) {.};
      \node[red] at (4,2) {+};
      \node at (5,2) {.};
      \draw[->,>=stealth] (0,-1) to (-25,-4);
      \draw[->,>=stealth] (2,-1) to (-7,-4);
      \draw[->,>=stealth] (4,-1) to (7,-4);
      \draw[->,>=stealth] (6,-1) to (13,-4);
      \end{scope}
    \begin{scope}[shift={(-25,-12)}] %6
      \draw[draw=black] (-.5,-.5) rectangle ++(2,3);
      \draw[draw=black] (1.5,-.5) rectangle ++(2,3);
      \draw[draw=black] (3.5,-.5) rectangle ++(2,3);
      \node at (0,0) {.};
      \node at (1,0) {+};
      \node at (2,0) {.};
      \node at (3,0) {.};
      \node at (4,0) {.};
      \node at (5,0) {.};
      \node at (0,1) {.};
      \node at (1,1) {.};
      \node at (2,1) {+};
      \node at (3,1) {+};
      \node[red] at (4,1) {+};
      \node at (5,1) {.};
      \node[red] at (0,2) {+};
      \node at (1,2) {.};
      \node[red] at (2,2) {+};
      \node at (3,2) {.};
      \node at (4,2) {+};
      \node at (5,2) {.};
      \end{scope}
    \begin{scope}[shift={(-18,-12)}] %7
      \draw[draw=black] (-.5,-.5) rectangle ++(2,3);
      \draw[draw=black] (1.5,-.5) rectangle ++(2,3);
      \draw[draw=black] (3.5,-.5) rectangle ++(2,3);
      \node at (0,0) {.};
      \node at (1,0) {.};
      \node at (2,0) {.};
      \node at (3,0) {.};
      \node at (4,0) {.};
      \node at (5,0) {+};
      \node[red] at (0,1) {+};
      \node at (1,1) {.};
      \node at (2,1) {+};
      \node[red] at (3,1) {+};
      \node at (4,1) {.};
      \node at (5,1) {.};
      \node at (0,2) {+};
      \node at (1,2) {.};
      \node[red] at (2,2) {+};
      \node at (3,2) {.};
      \node[red] at (4,2) {+};
      \node at (5,2) {.};
      \end{scope}

    \begin{scope}[shift={(-11,-12)}] %8
      \draw[draw=black] (-.5,-.5) rectangle ++(2,3);
      \draw[draw=black] (1.5,-.5) rectangle ++(2,3);
      \draw[draw=black] (3.5,-.5) rectangle ++(2,3);
      \node at (0,0) {.};
      \node at (1,0) {.};
      \node at (2,0) {.};
      \node at (3,0) {+};
      \node at (4,0) {.};
      \node at (5,0) {.};
      \node[red] at (0,1) {+};
      \node at (1,1) {.};
      \node at (2,1) {+};
      \node at (3,1) {+};
      \node[red] at (4,1) {+};
      \node at (5,1) {.};
      \node at (0,2) {+};
      \node at (1,2) {.};
      \node at (2,2) {+};
      \node at (3,2) {.};
      \node at (4,2) {+};
      \node at (5,2) {.};
      \end{scope}
    \begin{scope}[shift={(-4,-12)}] %9
      \draw[draw=black] (-.5,-.5) rectangle ++(2,3);
      \draw[draw=black] (1.5,-.5) rectangle ++(2,3);
      \draw[draw=black] (3.5,-.5) rectangle ++(2,3);
      \node at (0,0) {.};
      \node at (1,0) {+};
      \node at (2,0) {.};
      \node at (3,0) {+};
      \node at (4,0) {.};
      \node at (5,0) {.};
      \node at (0,1) {.};
      \node at (1,1) {.};
      \node at (2,1) {.};
      \node at (3,1) {.};
      \node[red] at (4,1) {+};
      \node at (5,1) {.};
      \node at (0,2) {+};
      \node at (1,2) {+};
      \node[red] at (2,2) {+};
      \node at (3,2) {.};
      \node at (4,2) {+};
      \node at (5,2) {.};
      \end{scope}
    \begin{scope}[shift={(3,-12)}] %10
      \draw[draw=black] (-.5,-.5) rectangle ++(2,3);
      \draw[draw=black] (1.5,-.5) rectangle ++(2,3);
      \draw[draw=black] (3.5,-.5) rectangle ++(2,3);
      \node at (0,0) {.};
      \node at (1,0) {+};
      \node at (2,0) {.};
      \node at (3,0) {.};
      \node at (4,0) {.};
      \node at (5,0) {+};
      \node at (0,1) {.};
      \node at (1,1) {.};
      \node[red] at (2,1) {+};
      \node at (3,1) {.};
      \node at (4,1) {.};
      \node at (5,1) {.};
      \node at (0,2) {+};
      \node[red] at (1,2) {+};
      \node at (2,2) {+};
      \node at (3,2) {.};
      \node[red] at (4,2) {+};
      \node at (5,2) {.};
      \end{scope}
    \begin{scope}[shift={(10,-12)}] %11
      \draw[draw=black] (-.5,-.5) rectangle ++(2,3);
      \draw[draw=black] (1.5,-.5) rectangle ++(2,3);
      \draw[draw=black] (3.5,-.5) rectangle ++(2,3);
      \node at (0,0) {.};
      \node at (1,0) {.};
      \node at (2,0) {.};
      \node at (3,0) {+};
      \node at (4,0) {.};
      \node at (5,0) {.};
      \node[red] at (0,1) {+};
      \node at (1,1) {.};
      \node at (2,1) {.};
      \node at (3,1) {+};
      \node[red] at (4,1) {+};
      \node at (5,1) {.};
      \node at (0,2) {+};
      \node[red] at (1,2) {+};
      \node at (2,2) {.};
      \node at (3,2) {.};
      \node at (4,2) {+};
      \node at (5,2) {.};
      \end{scope}
    \begin{scope}[shift={(17,-12)}] %12
      \draw[draw=black] (-.5,-.5) rectangle ++(2,3);
      \draw[draw=black] (1.5,-.5) rectangle ++(2,3);
      \draw[draw=black] (3.5,-.5) rectangle ++(2,3);
      \node at (0,0) {.};
      \node at (1,0) {.};
      \node at (2,0) {.};
      \node at (3,0) {+};
      \node at (4,0) {.};
      \node at (5,0) {+};
      \node[red] at (0,1) {+};
      \node at (1,1) {.};
      \node at (2,1) {.};
      \node at (3,1) {.};
      \node at (4,1) {.};
      \node at (5,1) {.};
      \node at (0,2) {+};
      \node at (1,2) {+};
      \node[red] at (2,2) {+};
      \node at (3,2) {.};
      \node[red] at (4,2) {+};
      \node at (5,2) {.};
      \end{scope}
    \begin{scope}[shift={(24,-12)}] %13
      \draw[draw=black] (-.5,-.5) rectangle ++(2,3);
      \draw[draw=black] (1.5,-.5) rectangle ++(2,3);
      \draw[draw=black] (3.5,-.5) rectangle ++(2,3);
      \node at (0,0) {.};
      \node at (1,0) {.};
      \node at (2,0) {.};
      \node at (3,0) {.};
      \node at (4,0) {.};
      \node at (5,0) {+};
      \node[red] at (0,1) {+};
      \node at (1,1) {.};
      \node[red] at (2,1) {+};
      \node at (3,1) {.};
      \node at (4,1) {.};
      \node at (5,1) {+};
      \node at (0,2) {+};
      \node[red] at (1,2) {+};
      \node at (2,2) {+};
      \node at (3,2) {.};
      \node at (4,2) {.};
      \node at (5,2) {.};
      \end{scope}
 \end{tikzpicture}
\end{center}
\caption{Example \ref{eg:bipartite 234 to 1}.}
\label{fig:234 to 1}
\end{figure}
Figure \ref{fig:234 to 1} illustrates those facets of the Stanley-Reisner complex  obtained from $C_{\max}(\emptyset)$ by applying at most two chute moves; the arrows denote chute moves, red ``$+$''s indicate where the chute move can apply.
So $N_{\mathcal{Q},\mm,\uu} = 11$,
$\dim \Delta=10$.
The numbers of $d$-faces in $\Delta$ are easily counted by a computer:
\begin{center}
\begin{adjustbox}{width=\columnwidth,center}
\begin{tabular}{|c| c c c c c c c c c c c c| c|} 
 \hline
$d$ & $-1$ & $0$ & $1$ & $2$ & $3$ & $4$ &$5$  & $6$ & $7$ & $8$ & $9$ & $10$ & total\\
 \hline
$\#$ of $d$-faces & $1$ & $18$ & $144$ & $670$ & $2013$ & $4110$ & $5837$  & $5784$ & $3930$ & $1748$ & $459$ & $54$ & $24768$\\ 
 \hline
$\#$ of interior $d$-faces & $0$ & $0$ & $0$ & $0$ & $0$ & $0$ & $1$ & $12$ & $57$ & $128$ & $135$ & $54$ & $387$ \\  
 \hline
\end{tabular}
\end{adjustbox}
\end{center}
So multiplicity$=54$.
By \eqref{eq:Hilb1t} or  \eqref{eq:Hilb2t}, 
$${\rm Hilb}({\bf k}[X]/I_{\mathcal{Q},\mm,\uu}; t)
=\frac{1+7t+19t^2+19t^3+7t^4+t^5}{(1-t)^{11}}
=\frac{(1+t)(1+6t+13t^2+6t^3+t^4)}{(1-t)^{11}}. 
$$
\end{example}

\subsection{Secant varieties}\label{subsection:Secant varieties}
In the special case when $r=2$, $u=v$, we recover some results on $t$-secant varieties by Conca, De Negri, and Stojanac \cite[Theorem 3.2, Theorem 4.1, Proposition 4.7]{CNS}. Recall the notation therein:  $I_{t+1}=I_{t+1}(\mathcal{X}^{\{2\}})+I_{t+1}(\mathcal{X}^{\{3\}})$ is the defining ideal of the $t$-secant variety $\sigma_t(2,a,b)$ of the Segre variety ${\rm Seg}(2,a,b)$. This ideal is, in our terminology, the double determinantal ideal $I^{(2)}_{a,b,t,t}$. 
It is easy to obtain the following corollary:
\begin{corollary}[\cite{CNS}]
{\rm(i)} The natural generators of $I^{(2)}_{a,b,t,t}$ are a Gr\"obner basis (with respect to an appropriate monomial order).

{\rm(ii)} ${\bf k}[X]/I^{(2)}_{a,b,t,t}$ is a Cohen-Macaulay domain. 

{\rm(iii)} The corresponding simplicial complex $\Delta$ is shellable. 

{\rm(iv)} Each facet of $\Delta$ can be expressed by a collection of $t$ disjoint paths $F_1,\dots,F_t$ such that $F_i$ is a lattice path starting in $(1,b+h_i)$ and ending in $(a,h_i)$, from NW to SE, and walk along the direction $(1,0)$ or $(0,-1)$ in each step, for some integers $h_1,\dots,h_t$ satisfying $1\le h_1<\cdots<h_t\le b$.
\end{corollary}
\begin{proof}
(i) follows from Theorem \ref{general}, (ii) and (iii) follow from Theorem \ref{thm:complex}. 
To show (iv), consider 
$\mathcal{Q}:2\stackrel{2}{\to}1, \quad A_1=[X^{(1)},X^{(2)}], \quad A_2=\begin{bmatrix}X^{(1)}\\X^{(2)}\end{bmatrix}.$
For $p=1,\dots,t$, let the segment joining $(c_p,b)$ and $(c_p,b+1)$ be the horizontal segment of $H^1_p$ in $A_1$ that crosses the two pages. 
For $q=1,\dots,t$, let the segment joining $(a,d_q)$ and $(a+1,d_q)$ be the vertical segment of $V^2_q$ in $A_2$ that crosses the two pages.  
We use our previous notation as in Figure \ref{fig:straight relabel} (taking $\alpha=1$) to denote the path $P_{pq}\rightsquigarrow Q_{pq}=H^1_p\cap V^1_q$ for $1\le p\le t$, $1\le q\le 2t$. Note that for $1\le p\le t$, $H^1_p$ contains no vertical segment between the two vertices $(a-t+p,1)$ and $P_{pp}$, $V^1_p$ contains no horizontal segment between the two vertices $(1,b-t+p)$ and $Q_{pp}$,  by Lemma \ref{lem:straight 2 conditions}. 
So in the first block,
for $p>q$, the points $P_{pq},Q_{pq}$ have the same $x$-coordinate $a-t+p$;
for $p<q$, the points $P_{pq},Q_{pq}$ have the same $y$-coordinate $b-t+q$. 
Similarly in the second block, where $t+1\le q\le 2t$,   
for $p>q-t$, the points $P_{pq},Q_{pq}$ have the same $y$-coordinate $q$;
for $p<q-t$, the points $P_{pq},Q_{pq}$ have the same $x$-coordinate $p$. 
For $1\le i\le t$, define 
$F_i$ to be the unique path that, if walking from SW to NE, contains the following subpaths in order:

$P_{pi}\rightsquigarrow Q_{pi}$ for $p=t,t-1,\dots,i$ (in decreasing order); 

$P_{iq}\rightsquigarrow Q_{iq}$ for $q=i,i+1,\dots,i+t$ (in increasing order); 

$P_{p,i+t}\rightsquigarrow Q_{p,i+t}$ for $p=i-1,i-2,\dots,1$ (in decreasing order);

\noindent In other words, if we walk in the direction from SW to NE, then $F_i$ starts at the point $P_{ti}$, follows $V^1_i$, then follows $H^1_i$, then follows $V^1_{t+i}$, and ends at the point $Q_{1,i+t}$. 
To be consistent with the convention in \cite{CNS}, we reverse the direction of $F_i$ by walking from NE to SW instead; then $F_i$ starts in $(1,b+h_i)=Q_{1,i+t}$, and ends in $(a,h_i)=P_{ti}$. Note that the $(Q_{1,i+t})_y = b + (P_{ti})_y$ because $\phi_2^{-1}\phi_1(V^1_i)$ and $\phi_2^{-1}\phi_1(V^1_{i+t})$ have to merge into a vertical path $V^2_i$ in $A_2$. Also note that $Q_{1,1+t},Q_{1,2+t},\dots,Q_{1,2t}$ are ordered  from left to right between $(1,b+1)$ and $(1,2b)$, so $1\le h_1< \cdots< h_t\le b$.  
\end{proof}
\begin{example}
Let $a=b=4, t=2$. An example of straight road map is given in Figure \ref{fig:secant}, where $h_1=2,h_2=4$. 

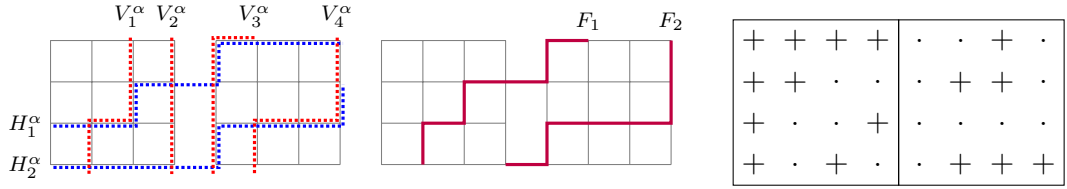
\begin{figure}[ht]
\begin{center}
\begin{tikzpicture}[scale=.55]

\draw  [gray] (0,0) grid (3,3);
\draw  [gray] (4,0) grid (7,3);

\begin{scope}[shift={(.07,-.07)}]
\draw [blue,densely dotted,very thick](0,1)--(2,1)--(2,2)--(4,2)--(4,3)--(7,3);
\draw [blue,densely dotted,very thick](0,0)--(4,0)--(4,1)--(7,1)--(7,2);

\draw (0,1) node [left] {\tiny $H^\alpha_1$};
\draw (0,0) node [left] {\tiny $H^\alpha_2$};
\end{scope}

\begin{scope}[shift={(-.07,.07)}]
\draw [red,densely dotted,very thick](2,3)--(2,1)--(1,1)--(1,-.3);
\draw [red,densely dotted,very thick](3,3)--(3,-.3);
\draw [red,densely dotted,very thick](5,3)--(4,3)--(4,-.3);
\draw [red,densely dotted,very thick](7,3)--(7,1)--(5,1)--(5,-.3);

\draw (2,3) node [above] {\tiny $V^\alpha_1$};
\draw (3,3) node [above] {\tiny $V^\alpha_2$};
\draw (5,3) node [above] {\tiny $V^\alpha_3$};
\draw (7,3) node [above] {\tiny $V^\alpha_4$};
\end{scope}

\begin{scope}[shift={(8,0)}]
\draw  [gray] (0,0) grid (3,3);
\draw  [gray] (4,0) grid (7,3);

\draw [purple,very thick](1,0)--(1,1)--(2,1)--(2,2)--(4,2)--(4,3)--(5,3);
\draw [purple,very thick](3,0)--(4,0)--(4,1)--(7,1)--(7,3);

\draw (5,3) node [above] {\tiny $F_1$};
\draw (7,3) node [above] {\tiny $F_2$};
\end{scope}

    \begin{scope}[shift={(17,0)}]
      \draw[draw=black] (-.5,-.5) rectangle ++(4,4);
      \draw[draw=black] (3.5,-.5) rectangle ++(4,4);
      \node at (0,3) {+};
      \node at (1,3) {+};
      \node at (2,3) {+};
      \node at (3,3) {+};
      \node at (4,3) {.};
      \node at (5,3) {.};
      \node at (6,3) {+};
      \node at (7,3) {.};
      \node at (0,2) {+};
      \node at (1,2) {+};
      \node at (2,2) {.};
      \node at (3,2) {.};
      \node at (4,2) {.};
      \node at (5,2) {+};
      \node at (6,2) {+};
      \node at (7,2) {.};
      \node at (0,1) {+};
      \node at (1,1) {.};
      \node at (2,1) {.};
      \node at (3,1) {+};
      \node at (4,1) {.};
      \node at (5,1) {.};
      \node at (6,1) {.};
      \node at (7,1) {.};
      \node at (0,0) {+};
      \node at (1,0) {.};
      \node at (2,0) {+};
      \node at (3,0) {.};
      \node at (4,0) {.};
      \node at (5,0) {+};
      \node at (6,0) {+};
      \node at (7,0) {+};
      \end{scope}

\end{tikzpicture}       
\end{center}
\caption{Left: an example of straight road map; Middle: the corresponding disjoint paths $F_i$'s; Right: the corresponding concurrent vertex map.}
\label{fig:secant}
\end{figure}
\end{example}

\subsection{Some questions to explore further}
There are many algebraic, geometric, or combinatorial questions we may ask about the bipartite determinantal ideals. To name a few:

\subsubsection{Redued Gr\"obner basis}
Recall that the natural generators of a classical determinantal ideal forms a reduced Gr\"obner basis \cite{Sturmfels}. Similar statements are false in general for bipartite determinantal varieties, or even for double determinantal varieties. Consider the double determinantal ideal $I^{(r)}_{m,n,u,v}$ with $r=2$, $m=n=2$, $u=v=1$. Denote 
$$X^{(1)}=\begin{bmatrix}a&b\\c&d\end{bmatrix},  X^{(2)}=\begin{bmatrix}a'&b'\\c'&d'\end{bmatrix}.
$$ 
The reduced Gr\"obner basis of $I^{(r)}_{m,n,u,v}$  is 
$$\{a'd'-b'c',cd'-dc',cb'-da', bd'-db',bc'-da',ac'-ca',ab'-ba',ad-bc,ad'-da'\}$$
Note that the last generator in the above list is not a natural generator. It is natural to ask

\noindent {\bf Question}: Can we combinatorially characterize the reduced Gr\"obner basis?

\subsubsection{Multiplicity formula in terms of determinants} The multiplicity of a determinantal ideal, which is the number of nonintersecting paths with certain fixed starting and ending points,  can be written as a determinant (see \cite{Herzog-Trung}). Similar formulas are found for more general ideals, for example see \cite{Krattenthaler-Prohaska} and the references therein. 

\noindent {\bf Question}: Can we find a multiplicity formula for a bipartite determinantal ideal in terms of determinants, or equivalently, a determinantal formula for the number of concurrent vertex maps?

\subsubsection{Gorenstein property} (We thank Aldo Conca  for suggesting the question on Gorenstein property.) 
 Given a Cohen-Macaulay standard graded domain, it has a symmetric $h$-polynomial if and only if  it is Gorenstein (by a theorem of Stanley; see \cite[Corollary 4.4.6]{BH}).  In Example \ref{eg:bipartite 234 to 1}, the $h$-polynomial is symmetric (also called palindromic); so ${\bf k}[X]/I_{\mathcal{Q},\mm,\uu}$ is Gorenstein for that example. In \cite[Conjecture 4.9]{CNS}, Conca, Negri and Stojanac proposed a conjecture about the Gorenstein property in their setting.

\noindent {\bf Question}: Can we explicitly identify all bipartite determinantal ideals with Gorenstein property?

\end{document}